\documentclass[10pt]{article}

\input xy
\xyoption{all}

\usepackage{bbm}
\usepackage{amsfonts}
\usepackage[colorlinks,linkcolor=blue,urlcolor=cyan,citecolor=cyan,pagebackref]{hyperref}
\usepackage[a4paper,left=2.5cm,right=2.5cm,bottom=3.5cm,top=3cm]{geometry}
\usepackage{appendix}
\usepackage{mathrsfs}
\usepackage{amsmath}
\usepackage{amssymb}
\usepackage{amsthm}
\usepackage{lscape}
\usepackage{array}
\usepackage{mathdots}

\theoremstyle{plain} \newtheorem{prop}{Proposition}[section]
\theoremstyle{plain} \newtheorem{theo}[prop]{Theorem}
\theoremstyle{plain} \newtheorem{lem}[prop]{Lemma}
\theoremstyle{plain} \newtheorem{conj}[prop]{Conjecture}
\theoremstyle{plain} \newtheorem{cor}[prop]{Corollary}
\theoremstyle{definition} \newtheorem{defn}[prop]{Definition}
\theoremstyle{remark} \newtheorem{rem}[prop]{Remark}
\theoremstyle{remark} 

\numberwithin{equation}{section}

\renewcommand{\(}{\left(}
\renewcommand{\)}{\right)}
\renewcommand{\[}{\left[}
\renewcommand{\]}{\right]}
\renewcommand{\1}{\left\{}
\renewcommand{\2}{\right\}}
\newcommand{\LA}{\left\langle}
\newcommand{\RA}{\right\rangle}
\newcommand{\R}{\rightarrow}
\newcommand{\HR}{\hookrightarrow}
\newcommand{\LR}{\longrightarrow}
\newcommand{\B}{\backslash}
\newcommand{\T}{\mapsto}
\newcommand{\BIB}[1]{\bibitem[\textbf{#1}]{#1}}
\newcommand{\intl}{\int\limits}
\newcommand{\iintl}{\iint\limits}

\newcommand{\vep}{\varepsilon}
\newcommand{\FJ}{\mathcal{F}\!\!\mathcal{J}}
\newcommand{\CF}{\mathbbm{1}}

\renewcommand{\b}[1]{\mathbf{#1}}
\renewcommand{\c}[1]{\mathcal{#1}}
\renewcommand{\d}[1]{\mathbb{#1}}
\newcommand{\f}[1]{\mathfrak{#1}}
\renewcommand{\r}[1]{\mathrm{#1}}
\newcommand{\s}[1]{\mathscr{#1}}

\begin{document}

\title{\textbf{Relative trace formulae toward Bessel and Fourier-Jacobi periods of unitary groups}}
\author{Yifeng Liu
\\ Columbia University}
\date{November 23, 2010}

\maketitle

\begin{abstract}
We propose a relative trace formula approach and state the
corresponding fundamental lemma toward the global restriction
problem involving Bessel or Fourier-Jacobi periods of unitary groups
$\r{U}_n\times\r{U}_m$, extending the work of Jacquet-Rallis for
$m=n-1$ (which is a Bessel period). In particular, when $m=0$, we
recover a relative trace formula proposed by Flicker concerning
Kloosterman/Fourier integrals on quasi-split unitary groups. As
evidence for our approach, we prove the fundamental lemma for
$\r{U}_n\times\r{U}_n$ in positive characteristics.
\end{abstract}

{\small \tableofcontents}

\section{Introduction}
\label{sec1}

Recently, Jacquet and Rallis \cite{JR} propose a new approach to the
Gan-Gross-Prasad conjecture for unitary groups
$\r{U}_n\times\r{U}_{n-1}$ relating certain periods with special
$L$-values (cf. \cite{GGP}). It is based on a relative trace
formula. In this paper, we extend this approach to all kinds of
pairs $\r{U}_n\times\r{U}_m$ where $0\leq m\leq n$. If $n-m$ is odd,
the related period is a Bessel period. If $n-m$ is even, the
related period is a Fourier-Jacobi period.\\

\textbf{Notations.} We denote by $\d{C}^1$ the unit circle in the
complex line $\d{C}$. For a number field $k$, we denote by $\d{M}_k$
the set of places of $k$.

We denote by $\r{M}_{r,s}$ the affine group of $r\times s$ matrices
and $\r{M}_r=\r{M}_{r,r}$. We denote by $\b{1}_r$ the identity
matrix of rank $r$.

We denote by $|S|$ the cardinality of a finite set $S$ and by
$\CF_T$ the characteristic function of any set $T$.

For a locally compact abelian topological group or a vector space
$X$, we denote $X^{\vee}$ for its dual. For a smooth representation
$\pi$, we denote $\widetilde{\pi}$ for its contragredient representation.\\

\textbf{Acknowledgements.} The author would like to thank Shou-Wu
Zhang for proposing this question. He also thanks Herv\'{e} Jacquet
and Shou-Wu Zhang for stimulating discussions and their interest in
this work.\\

\subsection{Periods and special values of $L$-functions}
\label{sec1psv}

Let us consider a quadratic extension of number fields $k/k'$. Let
$\tau$ be the nontrivial element in $\r{Gal}(k/k')$ and $\eta$ the
corresponding quadratic character of the id\`{e}le group
$\d{A}'^{\times}$ of $k'$. Let $V$ be a (non-degenerated) hermitian
space over $k$ (with respect to $\tau$) of dimension $n$ with the
hermitian form $(-,-)$ and $W\subset V$ a subspace of dimension $m$
such that the restricted hermitian form $(-,-)|_W$ is
non-degenerate.

Let $\r{U}_n=\r{U}(V)$ and $\r{U}_m=\r{U}(W)$ be the unitary groups.
We identify $\r{U}_m$ as a subgroup of $\r{U}_n$ fixing all elements
in the orthogonal complement $W^{\perp}\subset V$ of $W$. We define
a unipotent subgroup $U'=U'_{1^r,m}\subset\r{U}_n$ (resp.
$U'=U'_{1^r,m+1}$) normalized and hence acted through conjugation by
$\r{U}_m$ when $n-m$ is even (resp. odd). We define
$H'=U'\rtimes\r{U}_m$ which is a non-reductive group viewed as a
subgroup of $\r{U}_n\times\r{U}_m$ via the embedding into the first
factor and the projection onto the second factor (see Section
\ref{sec4bmp} and \ref{sec5fmp} for the precise definitions). Let
$\pi$ (resp. $\sigma$) be an irreducible tempered representation of
$\r{U}_n(\d{A}')$ (resp. $\r{U}_m(\d{A}')$) which occurs with
multiplicity one in the space of cusp forms $\c{A}_0(\r{U}_n)$
(resp. $\c{A}_0(\r{U}_m)$). We denote by $\c{A}_{\pi}$ (resp.
$\c{A}_{\sigma}$) the unique irreducible $\pi$ (resp.
$\sigma$)-isotypic subspace in $\c{A}_0(\r{U}_n)$ (resp.
$\c{A}_0(\r{U}_m)$).\\

First, we consider the case of Bessel periods where we require that
$n-m=2r+1$ is odd. We have a generic character $\nu'$ of $U'(k')\B
U'(\d{A}')$ and extend it to a character of $H'(k')\B H'(\d{A}')$
trivially on $\r{U}_m(\d{A}')$. For $\varphi_{\pi}\in\c{A}_{\pi}$
and $\varphi_{\sigma}\in\c{A}_{\sigma}$, we define
 \begin{align*}
 \c{B}^{\nu'}_r(\varphi_{\pi},\varphi_{\sigma})=\intl_{H'(k')\B
 H'(\d{A}')}\varphi_{\pi}\otimes\varphi_{\sigma}(h')\nu'(h')^{-1}\r{d}h'
 \end{align*}
to be a Bessel period of $\pi\otimes\sigma$. It is conjectured that
there is a nonzero Bessel period of representations in the Vogan
$L$-packet of $\pi\otimes\sigma$ if and only if the central special
$L$-value $L(\frac{1}{2},\r{BC}(\pi)\times\r{BC}(\sigma))\neq 0$,
where $\r{BC}$ stands for the standard base change and the
$L$-function is the one defined by the Rankin-Selberg convolution on
general linear groups (cf. \cite{JPSS83}).

Second, we consider the case of Fourier-Jacobi periods where we
require that $n-m=2r$ is even. We have an automorphic (essentially
Weil) representation $\nu'_{\psi',\mu}$ of $H'(\d{A}')$ by choosing
a nontrivial character $\psi':k'\B\d{A}'\R\d{C}^1$ and a character
$\mu:k^{\times}\B\d{A}^{\times}\R\d{C}^{\times}$ such that
$\mu|_{\d{A}'^{\times}}=\eta$, realizing on certain space of
Bruhat-Schwartz functions $\c{S}$. For
$\varphi_{\pi}\in\c{A}_{\pi}$, $\varphi_{\sigma}\in\c{A}_{\sigma}$
and $\phi\in\c{S}$, we define
 \begin{align*}
 \FJ^{\nu'_{\psi',\mu}}_r(\varphi_{\pi},\varphi_{\sigma};\phi)=\intl_{H'(k')\B
 H'(\d{A}')}\varphi_{\pi}\otimes\varphi_{\sigma}(h')\theta(h',\phi)\r{d}h'
 \end{align*}
to be a Fourier-Jacobi period of $\pi\otimes\sigma$ (with respect to
$\mu$), where $\theta(h',\phi)$ is a certain theta series on
$H'(\d{A}')$ attached to $\phi$. It is conjectured that there is a
nonzero Fourier-Jacobi period of representations  in the Vogan
$L$-packet of $\pi\otimes\sigma$ if and only if the central special
$L$-value
$L(\frac{1}{2},\r{BC}(\pi)\times\r{BC}(\sigma)\otimes\mu^{-1})\neq
0$.\\

There has been significant progress toward this conjecture in a
series of papers of Ginzburg-Jiang-Rallis: \cite{GJR09} for unitary
groups, \cite{GJR04} for symplectic groups and \cite{GJR05} for
orthogonal groups. In all these cases, they prove one direction:
nontrivial Bessel or Fourier-Jacobi periods imply the non-vanishing
of corresponding central $L$-values. For the other direction, they
also obtain some conditional results. Their approach is to study the
residue of certain Eisenstein series and some Fourier coefficients
attached to it.

But one can ask more about the precise relation between these
periods and central special $L$-values, known as the Ichino-Ikeda
conjecture in the context of $\r{SO}_n\times\r{SO}_{n-1}$ (cf.
\cite{II10}). The advantage of the relative trace formula approach
is that it is possible to prove the explicit formula relating
$|\c{B}^{\nu'}_r(\varphi_{\pi},\varphi_{\sigma})|^2$ (or
$|\FJ^{\nu'_{\psi',\mu}}_r(\varphi_{\pi},\varphi_{\sigma};\phi)|^2$)
and the product of certain local periods of positive type, within
$L$-values as the scaling factor of these two periods. In
particular, one can prove the positivity of the corresponding
central special $L$-value, for example, as in \cite{JC01}. We will
not pursue the formulation of this explicit relation in this paper,
but instead, formulate the relative trace
formulae toward it for both periods.\\

\subsection{Relative trace formulae and fundamental lemmas}
\label{sec1rtf}

We briefly describe our relative trace formula. To be simple for the
introduction, we only do this for the case of Bessel periods, i.e.,
we will assume that $n-m=2r+1$ is odd.

Let $f_n\in\c{H}(\r{U}_n(\d{A}'))$ (resp.
$f_m\in\c{H}(\r{U}_m(\d{A}'))$) be a smooth function on
$\r{U}_n(\d{A}')$ (resp. $\r{U}_m(\d{A}')$) with compact support. We
introduce a distribution
 \begin{align*}
 \c{J}_{\pi,\sigma}(f_n\otimes f_m):=\sum\c{B}^{\nu'}_r
 (\rho(f_n)\varphi_{\pi},\rho(f_m)\varphi_{\sigma})
 \overline{\c{B}^{\nu'}_r(\varphi_{\pi},\varphi_{\sigma})}
 \end{align*}
where the sum is taken over orthonormal bases of $\c{A}_{\pi}$ and
$\c{A}_{\sigma}$ and $\rho$ denotes the action by right translation.

As usual, we associate to $f_n\otimes f_m$ a kernel function on
$\(\r{U}_n(k')\B\r{U}_n(\d{A}')\times\r{U}_m(k')\B\r{U}_m(\d{A}')\)^2$:
 \begin{align*}
 \c{K}_{f_n\otimes
 f_m}(g'_1,g'_2;g'_3,g'_4)=\sum_{\zeta'\in\r{U}_n(k')}f_n(g'^{-1}_1\zeta'
 g'_3)\sum_{\xi'\in\r{U}_m(k')}f_m(g'^{-1}_2\xi' g'_4)
 \end{align*}
and consider the following distribution which is formally the
``sum'' of $\c{J}_{\pi,\sigma}$ over all $\pi\otimes\sigma$:
 \begin{align*}
 \c{J}(f_n\otimes f_m):=\iintl_{\(H'(k')\B
 H'(\d{A}')\)^2}\c{K}_{f_n\otimes
 f_m}\(h'_1,h'_1;h'_2,h'_2\)
 \nu'(h'^{-1}_1h'_2)\r{d}h'_1\r{d}h'_2.
 \end{align*}
The above integral is not absolutely convergent in general and need
to be regularized. It turns out that the regular part of this
distribution has the following decomposition:
 \begin{align*}
 \c{J}_{\r{reg}}(f_n\otimes
 f_m)=\sum_{\zeta'\in[\r{U}_n(k')_{\r{reg}}]/\b{H}'(k')}\c{J}_{\zeta'}(f)
 \end{align*}
where $[\r{U}_n(k')_{\r{reg}}]/\b{H}'(k')$ is the set of certain
regular orbits in $\r{U}_n(k')$ which will be discussed in Section
\ref{sec4co} and $f\in\c{H}(\r{U}_n(\d{A}'))$ is obtained from
$f_n\otimes f_m$. Moreover, each summand $\c{J}_{\zeta'}$ is an
ad\`{e}lic weighted orbital integral:
 \begin{align*}
 \c{J}_{\zeta'}(f)=\intl_{\r{U}_m(\d{A}')}\iintl_{\(U'_{1^r,m+1}(\d{A'})\)^2}f(g'^{-1}u'^{-1}_1\zeta'u'_2g')
 \nu'(u'^{-1}_1u'_2)\r{d}u'_1\r{d}u'_2\r{d}g'.
 \end{align*}\\

To connect the $L$-function, one need to go to the
$\r{GL}_n\times\r{GL}_m$ side. Let $\Pi=\r{BC}(\pi)$ and
$\Sigma=\r{BC}(\sigma)$ and assume that they remain cuspidal. We
define similarly a unipotent subgroup $U_{1^r,m+1,1^r}$ of
$\r{GL}_n$, a non-reductive group $H=U_{1^r,m+1,1^r}\rtimes\r{GL}_m$
viewed as a subgroup of $\r{GL}_n\times\r{GL}_m$ and a character
$\nu$ on $H$ (see Section \ref{sec2bm} for precise definitions). For
$\varphi_{\Pi}\in\c{A}_{\Pi}$ and
$\varphi_{\Sigma}\in\c{A}_{\Sigma}$, consider the following general
linear group version of the Bessel period:
 \begin{align*}
 \c{B}^{\nu}_{r,r}(\varphi_{\Pi},\varphi_{\Sigma}):=\intl_{H(k)\B
 H(\d{A})}\varphi_{\Pi}\otimes\varphi_{\Sigma}(h)\nu(h)^{-1}\r{d}h.
 \end{align*}
We remark that the above integral is the usual Rankin-Selberg
convolution for $\r{U}_n\times\r{U}_m$ when $r=0$, but not when
$r>0$. But in fact, it is also an integral presentation of
$L(s,\Pi\times\Sigma)$.

To single out the cuspidal representations being the standard base
change from the unitary groups, we follow \cite{JR}. Saying that $n$
is odd, let
 \begin{align*}
 \c{P}_n(\varphi_{\Pi})=\intl_{Z'_n(\d{A'})\r{GL}_n(k')\B\r{GL}_n(\d{A}')}\varphi_{\Pi}(g_1)\r{d}g_1;\qquad
 \c{P}_m(\varphi_{\Sigma})=\intl_{Z'_m(\d{A'})\r{GL}_m(k')\B\r{GL}_m(\d{A}')}\varphi_{\Pi}(g_2)\eta(\det g_2)\r{d}g_2
 \end{align*}
where $Z'_?$ is the center of $\r{GL}_{?,k'}$.

Let $F_n\in\c{H}(\r{GL}_n(\d{A}))$ and
$F_m\in\c{H}(\r{GL}_m(\d{A}))$. We introduce another distribution
 \begin{align*}
 \c{J}_{\Pi,\Sigma}(F_n\otimes F_m):=\sum\c{B}^{\nu}_{r,r}(\rho(F_n)\varphi_{\Pi},\rho(F_m)\varphi_{\Sigma})
 \overline{\c{P}_n(\varphi_{\Pi})\c{P}_m(\varphi_{\Sigma})}
 \end{align*}
where the sum is taken over orthonormal bases of $\c{A}_{\Pi}$ and
$\c{A}_{\Sigma}$. Repeat the process for the unitary groups, we have
the kernel function $\c{K}_{F_n\otimes F_m}$ and the distribution
$\c{J}(F_n\otimes F_m)$ whose regular part has the following
decomposition:
 \begin{align*}
 \c{J}_{\r{reg}}(F_n\otimes
 F_m)=\sum_{\zeta\in[\r{S}_n(k')_{\r{reg}}]/\b{H}(k')}\c{J}_{\zeta}(F)
 \end{align*}
where $[\r{S}_n(k')_{\r{reg}}]/\b{H}(k')$ is the set of certain
regular orbits in the symmetric space $\r{S}_n(k')$ which will be
discussed in Section \ref{sec4co} and $F\in\c{H}(\r{S}_n(\d{A}'))$
is obtained from $F_n\otimes F_m$. Moreover, each summand
$\c{J}_{\zeta}$ is an ad\`{e}lic weighted orbital integral:
 \begin{align*}
 \c{J}_{\zeta}(F)=\intl_{\r{GL}_m(\d{A})}\intl_{U_{1^r,m+1,1^r}(\d{A'})}F(g^{-1}u^{-1}\zeta u^{\tau}g)
 \nu(u^{-1})\r{d}u\r{d}g.
 \end{align*}\\

We prove in Proposition \ref{prop4orbit} that there is a natural
bijection
 \begin{align*}
 \b{N}:[\r{S}_n(k')_{\r{reg}}]/\b{H}(k')\overset{\sim}{\LR}\coprod_{W\subset
 V}[\r{U}_n(k')_{\r{reg}}]/\b{H}'(k')
 \end{align*}
where the disjoint union is taken over all isometry classes of
$W\subset V$. Hence, we pose the global matching condition for
$\b{N}(\zeta)=\zeta'$ as $\c{J}_{\zeta}(F)=\c{J}_{\zeta'}(f)$. The
precise conjecture of smooth matching of functions is proposed as
Conjecture \ref{conj4sm}.\\

As the most important and interesting problem in this kind of trace
formula, we now discuss the corresponding fundamental lemma. Here,
again to be simple, we only state it for the unit elements in the
following cases: $n-m$ odd, $m=0$, or $n=m$.

Now let $k'$ be a non-archimedean local field and $k/k'$ an
unramified quadratic field extension. Let $\f{o}'$ (resp. $\f{o}$)
be the ring of integers of $k'$ (resp. $k$). There are only two
non-isomorphic hermitian spaces of dimension $m>0$ over $k$. Let
$\r{U}^+_m\subset\r{U}^+_n$ be the pair associate to $W^+\subset
V^+$ both with trivial discriminant and $\r{U}^-_m\subset\r{U}^-_n$
be another one. Then $W^+$ will have a self-dual $\f{o}$-lattice
$L_W$ which extends to a self-dual $\f{o}$-lattice $L_V$ of $V^+$.
The unitary group $\r{U}^+_m$ (resp. $\r{U}^+_n$) is unramified and
has a model over $\f{o}'$. The group of $\f{o}'$-points
$\r{U}^+_m(\f{o}')$ (resp. $\r{U}^+_n(\f{o}')$) is a hyperspecial
maximal subgroup of $\r{U}^+_m(k')$ (resp. $\r{U}^+_n(k')$). We also
denote by $\r{GL}_n(L_V)\cong\r{GL}_n(\f{o})$ a hyperspecial maximal
subgroup of $\r{GL}_n(k)$ and $\r{S}_n(\f{o}'):=\r{S}_n(k')\cap
K_n$.

When $n-m$ is even, we define
$\r{S}_{n,m}(k')=\r{S}_n(k')\times\r{M}_{1,m}(k')\times\r{M}_{m,1}(k')$
and $\r{U}_{n,m}(k')=\r{U}_n(k')\times\r{M}_{1,m}(k)$. There is also
a notion of regular elements in both sets and we have a natural
bijection
$\b{N}:[\r{S}_{n,m}(k')_{\r{reg}}]/\b{H}(k')\overset{\sim}{\LR}\coprod_{W\subset
 V}[\r{U}_{n,m}(k')_{\r{reg}}]/\b{H}'(k')$ (cf. Section
 \ref{sec5co}). We propose the following conjecture.

\begin{conj}[\textbf{The fundamental lemma for unit elements}]\label{conj1fl}
(1) When $n-m$ is odd or $m=0$, we have
 \begin{align*}
 &\intl_{\r{GL}_m(k')}\intl_{U_{1^r,m+1,1^r}(k)}\CF_{\r{S}_n(\f{o}')}
 (g^{-1}u^{-1}\zeta u^{\tau}g)\nu(u^{-1})\eta(\det
 g)\r{d}u\r{d}g\\
 =&\begin{cases}
 \displaystyle
 \b{t}(\zeta)\intl_{\r{U}_m^+(k')}\iintl_{\(U'_{1^r,m+1}\)^2}
 \CF_{\r{U}^+_n(\f{o}')}(g'^{-1}u'^{-1}_1\zeta^+u'_2g')\nu'(u'^{-1}_1u'_2)
 \r{d}u'_1\r{d}u'_2\r{d}g' &\b{N}(\zeta)=\zeta^+\in\r{U}^+_n(k');\\
 0 &\b{N}(\zeta)=\zeta^-\in\r{U}^-_n(k')
 \end{cases}
 \end{align*}
where $\b{t}(\zeta)\in\1\pm1\2$ is a certain transfer factor defined
in \eqref{4tfactor}. In particular, when $m=0$, the second
case of the above identity does not happen.\\
 (2) When $n=m$, we have
 \begin{align*}
 &\intl_{\r{GL}_n(k')}\CF_{\r{S}_n(\f{o}')}
 (g^{-1}\zeta g)\CF_{\r{M}_{1,n}(\f{o}')}(xg)\CF_{\r{M}_{n,1}(\f{o}')}(g^{-1}y)\eta(\det g)\r{d}g\\
 =&\begin{cases}
 \displaystyle
 \b{t}([\zeta,x,y])\intl_{\r{U}^+_n(k')}\CF_{\r{U}^+_n(\f{o}')}(g'^{-1}\zeta^+g')\CF_{\r{M}_{1,n}(\f{o})}(zg')
 \r{d}g' &\b{N}([\zeta,x,y])=[\zeta^+,z]\in\r{U}^+_{n,n}(k');\\
 0 &\b{N}([\zeta,x,y])=[\zeta^-,z]\in\r{U}^-_{n,n}(k')
 \end{cases}
 \end{align*}
where $\b{t}([\zeta,x,y])\in\1\pm1\2$ is a certain transfer factor
defined in \eqref{5tfactor}. In particular, when $m=0$, the second
case of the above identity does not happen.\\
\end{conj}

We have

\begin{theo}\label{theo1fl}
(1) The second case of the identity in both cases in Conjecture
\ref{conj1fl}
holds;\\
(2) When $n=m+1$, the fundamental lemma holds if
$\r{char}(k)=p>\r{max}\1 n,2\2$ or $\r{char}(k)=0$ and the residue
characteristic
is sufficiently large with respect to $n$;\\
(3) When $n\leq3,m=0$, the fundamental lemma holds if $\r{char}(k)\neq 2$;\\
(4) When $n=m$, the fundamental lemma holds if
$\r{char}(k)=p>\r{max}\1 n,2\2$.
\end{theo}
\begin{proof}
(1) is proved in Proposition \ref{prop4half} and Proposition
\ref{prop5half} in this paper.\\
(2) is proved by Yun in \cite{Yun09} where the transfer to
characteristic $0$ is accomplished by Gordon in the appendix.\\
(3) is proved by Jacquet in \cite{J92} when $n=3$; (essentially)
proved by Ye in \cite{Ye89} when $n=2$;
and trivial when $n=1$.\\
(4) is proved in Theorem \ref{theo5fl}.\\
\end{proof}

\begin{rem}
When $n=m+1$, the fundamental lemma is just the one proposed by
Jacquet-Rallis in \cite{JR}. When $m=0$, the fundamental lemma is
the one proposed by Flicker in \cite{Fl91}, which is the unitary
group version of the Jacquet-Ye fundamental lemma (cf. \cite{JY92}).
When $m>0$ (and $n-m$ is odd), the fundamental lemma is a kind of
hybrid of the Jacquet-Rallis fundamental lemma and the Flicker
fundamental lemma. We hope that there is a geometric method toward
this fundamental lemma, which is also a kind of hybrid of those in
\cite{Yun09} by Yun and \cite{Ngo99} by Ng\^{o}.

When $0<n-m<n$ is even, we also formulate a corresponding
fundamental lemma for Fourier-Jacobi periods. See Chapter \ref{sec5}
for details.

The fundamental lemma for all elements in the spherical Hecke
algebra when $n=3,m=0$ is also proved, by Mao in \cite{Mao93}.\\
\end{rem}

\subsection{Variants of Rankin-Selberg convolutions}
\label{sec1rsc}

As we see in the last section, we need to consider certain periods
of Bessel and Fourier-Jacobi types on the general linear groups as
well. We will generalize the notion of Bessel (resp. Fourier-Jacobi)
models and periods for $\r{GL}_n\times\r{GL}_m$ for a pair $(r,r^*)$
of nonnegative integers such that $n=m+1+r+r^*$ (resp. $n=m+r+r^*$).
When $r=r^*$, they are introduced and considered in \cite{GGP}.\\

For simplicity, let us only describe the first case. Let $(r,r^*)$
be as above, we introduce a unipotent subgroup $U_{1^r,m+1,1^{r^*}}$
of $\r{GL}_n$ and $H=U_{1^r,m+1,1^{r^*}}\rtimes\r{GL}_m$ viewed as a
subgroup of $\r{GL}_n\times\r{GL}_m$. We have a character $\nu$ of
$H$ which is automorphic if $k$ is a number field.

Let $\pi$ (resp. $\sigma$) be an irreducible cuspidal automorphic
representation of $\r{GL}_n(\d{A})$ (resp. $\r{GL}_m(\d{A})$). We
introduce the Bessel integral
 \begin{align*}
 \c{B}^{\nu}_{r,r^*}(s;\varphi_{\pi},\varphi_{\sigma}):=\intl_{H(k)\B
 H(\d{A})}\varphi_{\Pi}\otimes\varphi_{\Sigma}(h)\nu(h)^{-1}|\det h|_{\d{A}}^{s-\frac{1}{2}}\r{d}h
 \end{align*}
and the Bessel period
$\c{B}^{\nu}_{r,r^*}(\varphi_{\pi},\varphi_{\sigma}):=\c{B}^{\nu}_{r,r^*}(\frac{1}{2};\varphi_{\pi},\varphi_{\sigma})$
for $\varphi_{\pi}\in\c{A}_{\pi}$,
$\varphi_{\sigma}\in\c{A}_{\sigma}$ and $s\in\d{C}$. Then
$\c{B}^{\nu}_{r,r^*}$ is the usual Rankin-Selberg convolution on
$\r{GL}_n\times\r{GL}_m$ if and only if $r=0$.

For general $(r,r^*)$ and also for the Fourier-Jacobi integrals
$\FJ^{\nu_{\mu}}_{r,r^*}(s;\varphi_{\pi},\varphi_{\sigma};\Phi)$, we
have the following theorem. We would like to remark that the
following result is actually known by Jacquet, Piatetskii-Shapiro
and Shalika long time ago. But since it is not recorded in the
literature, we would like to write it down with proof just for the
completeness.

\begin{theo}[see Section \ref{sec2bi} and \ref{sec3fi} for
notations]\label{theo1fe} (1) The Bessel integrals are holomorphic
in $s$ and satisfy the
 following functional equation
  \begin{align*}
  \c{B}^{\nu}_{r,r^*}(s;\varphi_{\pi},\varphi_{\sigma})
  =\c{B}^{\overline{\nu}}_{r^*,r}(1-s;\rho(w_{n,m})\widetilde{\varphi_{\pi}},\widetilde{\varphi_{\sigma}}).
  \end{align*}
For $\varphi_{\pi}\in\c{A}_{\pi}$ and
$\varphi_{\sigma}\in\c{A}_{\sigma}$ such that
$W^{\psi}_{\varphi_{\pi}}=\otimes_vW_v$ and
$W^{\overline{\psi}}_{\varphi_{\sigma}}=\otimes_vW^-_v$ are
factorizable,
 \begin{align*}
 \c{B}^{\nu}_{r,r^*}(\varphi_{\pi},\varphi_{\sigma})=L(\frac{1}{2},\pi\times\sigma)
 \prod_{v\in\d{M}_k}\left.\frac{\Psi_{v,r}(s;W_v,W^-_v)}{L_v(s,\pi_v\times\sigma_v)}\right|_{s=\frac{1}{2}}
 \end{align*}
where in the last product almost all factors are $1$. In particular,
there is a nontrivial Bessel period of $\pi\otimes\sigma$ if and
only if $L(\frac{1}{2},\pi\times\sigma)\neq 0$.\\
(2) The Fourier-Jacobi integrals are holomorphic in $s$ (when $n>m$)
and satisfy the following functional equation
  \begin{align*}
  \FJ^{\nu_{\mu}}_{r,r^*}(s;\varphi_{\pi},\varphi_{\sigma};\Phi)
  =\FJ^{\overline{\nu_{\mu}}}_{r^*,r}(1-s;\rho(w_{n,m})\widetilde{\varphi_{\pi}},\widetilde{\varphi_{\sigma}};
  \widehat{\Phi}).
  \end{align*}
For $\varphi_{\pi}\in\c{A}_{\pi}$,
$\varphi_{\sigma}\in\c{A}_{\sigma}$ and
$\Phi\in\c{S}(W^{\vee}(\d{A}))$ such that
$W^{\psi}_{\varphi_{\pi}}=\otimes_vW_v$,
$W^{\overline{\psi}}_{\varphi_{\sigma}}=\otimes_vW^-_v$ and
$\Phi=\otimes\Phi_v$ are factorizable,
 \begin{align*}
 \FJ^{\nu_{\mu}}_{r,r^*}(\varphi_{\pi},\varphi_{\sigma};\Phi)=L(\frac{1}{2},\pi\times\sigma\otimes\mu^{-1})
 \prod_{v\in\d{M}_k}\left.\frac{\Psi_{v,r}(s;W_v,W^-_v\otimes\mu^{-1}_v;\Phi_v)}
 {L_v(s,\pi_v\times\sigma_v\otimes\mu_v^{-1})}\right|_{s=\frac{1}{2}}
 \end{align*}
where in the last product almost all factors are $1$. In particular,
there is a nontrivial Fourier-Jacobi period of $\pi\otimes\sigma$
for $\nu_{\mu}$ if
and only if $L(\frac{1}{2},\pi\times\sigma\otimes\mu^{-1})\neq 0$.\\
\end{theo}

\begin{rem}
The above theorem completely confirms \cite[Conjecture 24.1]{GGP}
for split unitary groups, i.e., general linear groups.\\
\end{rem}

It is clear that the Bessel (resp. Fourier-Jacobi) period defines an
element in the space of invariant functionals
$\r{Hom}_{H(\d{A})}(\pi\otimes\sigma\otimes\widetilde{\nu},\d{C})$
where $\nu$ is a character (resp. an infinite dimensional
representation) of $H(\d{A})$. It has a decomposition into spaces of
local invariant functionals. The following multiplicity one result
is a generalization of \cite[Corollary 15.3, 16.3]{GGP} for general
linear groups.

\begin{theo}\label{theo1m1}
Let $k$ be a non-archimedean local field of characteristic zero. Let
$\pi$ (resp. $\sigma$) be an irreducible admissible representation
of $\r{GL}_n$ (resp. $\r{GL}_m$). Then
$\r{dim}_{\d{C}}\r{Hom}_H(\pi\otimes\sigma\otimes\widetilde{\nu},\d{C})\leq1$.
Moreover, if $\pi$ and $\sigma$ are generic, then
$\r{dim}_{\d{C}}\r{Hom}_H(\pi\otimes\sigma\otimes\widetilde{\nu},\d{C})=1$.\\
\end{theo}

The following is an outline of the paper.

In Chapter \ref{sec2}, we focus on the Bessel model and period on
general linear groups. We prove the Bessel part of Theorem
\ref{theo1fe} and Theorem \ref{theo1m1}. The proof of the local
multiplicity one result will occupy the first section, while we
follow the idea of \cite{GGP}. In particular, our proof includes the
case $r=r^*$ which was left as an exercise to readers in \cite{GGP}.
In the second section, as we have said, we will give a proof for the
global integral being Eulerian for the completeness of literature.

In Chapter \ref{sec3}, we focus on the Fourier-Jacobi model and
period on general linear groups. We prove the Fourier-Jacobi part of
Theorem \ref{theo1fe} and Theorem \ref{theo1m1}.

After briefly recalling the Bessel model and period for unitary
groups, we introduce the relative trace formula in Chapter
\ref{sec4}. We prove the matching of orbits in Section \ref{sec4co}
and the smooth matching of functions at split places. We formulate
the fundamental lemma and prove the easy half in \ref{sec4fl}.

Chapter \ref{sec5} is repeating the previous chapter, but for the
Fourier-Jacobi model and period. What we do more is the proof of the
full fundamental lemma for $\r{U}_n\times\r{U}_n$ in positive
characteristics, which is achieved in Section \ref{sec5nn}. The
proof is reduced to a known combinatorial identity proved in
\cite{Yun09}.

The last chapter is an appendix on integrals of local Whittaker
functions for general linear groups. We collect all the results we
need in Chapter \ref{sec2} and \ref{sec3} from existing literature.
In particular, we have to use all kinds of auxiliary local Whittaker
integrals in the theory of Rankin-Selberg convolutions.\\

\section{Bessel periods of $\r{GL}_n\times\r{GL}_m$}
\label{sec2}

\subsection{Bessel models}
\label{sec2bm}

Let $k$ be a local field and $|\;|_k$ the normalized absolute value
on $k$. Let $V$ be a $k$-vector space of dimension $n$. Suppose that
$V$ has a decomposition $V=X\oplus W\oplus E\oplus X^*$ where $W$,
$X$ and $X^*$ have dimension $m$, $r$ and $r^*$ respectively and
$E=\LA e\RA$ with $e\neq 0$, hence $n=m+r+r^*+1$. We want to
generalize the construction of the pair $(H,\nu)$ as in
\cite[Section 13]{GGP}. Let $P_{r,m+1,r^*}$ be the parabolic
subgroup of $\r{GL}(V)$ stabilizing the flag $0\subset X\subset
X\oplus W\oplus E\subset V$ and $U_{r,m+1,r^*}$ its maximal
unipotent subgroup. Then $U_{r,m+1,r^*}$ fits into the following
exact sequence:
 \begin{align*}
 \xymatrix@C=0.5cm{
   0 \ar[r] & \r{Hom}(X^*,X) \ar[r]& U_{r,m+1,r^*} \ar[r]& \r{Hom}(X^*,W\oplus E)+\r{Hom}(W\oplus E,X) \ar[r] & 0
   }.
 \end{align*}
We may write the above sequence as:
 \begin{align*}
 \xymatrix@C=0.5cm{
   0 \ar[r] & \(X^*\)^{\vee}\otimes X \ar[r]& U_{r,m+1,r^*} \ar[r]& \(X^*\)^{\vee}\otimes(W\oplus E)
   +(W^{\vee}\oplus E^{\vee})\otimes X \ar[r] & 0}.
 \end{align*}
Let $\ell_X:X\R k$ (resp. $\ell_{X^*}:k\R X^*$) be any nontrivial
homomorphism (if exists) and let $U_{X}$ (resp. $U_{X^*}$) be a
maximal unipotent subgroup of $\r{GL}(X)$ (resp. $\r{GL}(X^*)$)
stabilizing $\ell_X$ (resp. $\ell_{X^*}$). Moreover, let
 \begin{align*}
 \ell_W:(W\oplus E)+(W^{\vee}\oplus E^{\vee})\LR k
 \end{align*}
be a bilinear form which is trivial on $W+W^{\vee}$ and nontrivial
on $E$ and $E^{\vee}$. The composition of $\ell_X+\ell_{X^*}^{\vee}$
and $\ell_W$ defines a homomorphism
 \begin{align*}
 \ell:U_{r,m+1,r^*} \LR \(X^*\)^{\vee}\otimes(W\oplus E)
   +(W^{\vee}\oplus E^{\vee})\otimes X \overset{\ell_X+\ell_{X^*}^{\vee}}{\LR}
   (W\oplus E)+(W^{\vee}\oplus E^{\vee})\overset{\ell_W}{\LR} k
 \end{align*}
which is fixed by $(U_X\times U_{X^*})\times\r{GL}(W)$. Hence we can
extend $\ell$ trivially to it and define a homomorphism from
$H=U_{r,m+1,r^*}\rtimes((U_X\times U_{X^*})\times\r{GL}(W))$ to $k$.
Let $\psi:k\R\d{C}^1$ be a nontrivial character and
$\lambda:U_X\times U_{X^*}\R \d{C}^1$ a generic character which can
be viewed as a character of $H$. Let $\delta_W$ be the modulus
function of $\r{GL}(W)$ with respect to the adjoint action on
$U_{r,m+1,r^*}$, i.e., $\delta_W(g)=|\det g|_k^{r^*-r}$. Then we can
form a character
$\nu=(\psi\circ\ell)\otimes\lambda\otimes\delta_W^{-\frac{1}{2}}$ of
$H$. There is a natural embedding $\vep:H\HR\r{GL}(V)$ and a
projection $\kappa:H\R\r{GL}(W)$ which together induce an injective
morphism $(\vep,\kappa):H\HR\r{GL}(V)\times\r{GL}(W)$. Then the pair
$(H,\nu)$ is uniquely determined up to conjugacy in the group
$\r{GL}(V)\times\r{GL}(W)$ by the pair $W\subset V$ and $(r,r^*)$.
The following theorem generalizes the result in \cite{GGP}.

\begin{theo}\label{theo2bm}
Let $k$ be of characteristic zero. Let $\pi$ (resp. $\sigma$) be an
irreducible admissible representation of $\r{GL}(V)$ (resp.
$\r{GL}(W)$). Then
$\r{dim}_{\d{C}}\r{Hom}_H(\pi\otimes\sigma,\nu)\leq 1$. Moreover, if
$\pi$ and $\sigma$ are generic, then
$\r{dim}_{\d{C}}\r{Hom}_H(\pi\otimes\sigma,\nu)=1$.\\
\end{theo}

The existence part is due to Corollary \ref{6cor} (1). We consider
the uniqueness part. The proof for $k$ non-archimedean is similar to
that in \cite[Section 15]{GGP} with mild modifications for general
linear groups and a general pair $(r,r^*)$. The proof for $k$
archimedean should follow similarly as in \cite{JSZ10} which we
omit.

Recall that we have $V=X\oplus W\oplus E\oplus X^*$. Let $E^+$ be a
$k$-line generated by $e^+$. Let $v_0=e+e^+$, $v^*_0=e-e^+$ and
 \begin{align*}
 Y=X\oplus\LA v_0\RA;\qquad Y^*=X^*\oplus\LA v^*_0\RA;\qquad
 V^+=V\oplus E^+.
 \end{align*}
Let $P_0$ be the parabolic subgroup of $\r{GL}(V^+)$ stabilizing the
flag $\s{F}_0:0\subset Y\subset Y\oplus W\subset V^+$ and $M_0$ its
Levi subgroup such that
$M_0\cong\r{GL}(Y)\times\r{GL}(W)\times\r{GL}(Y^*)$. The group
$\r{GL}(V)$ embeds into $\r{GL}(V^+)\times\r{GL}(V)$ diagonally. Let
$\tau$ (resp. $\tau^*$) be an irreducible supercuspidal
representation of $\r{GL}(Y)$ (resp. $\r{GL}(Y^*)$) and let
 \begin{align*}
 \r{I}\(\tau,\sigma,\tau^*\):=\r{Ind}^{\r{GL}(V^+)}_{P_0}(\tau\otimes\sigma\otimes\tau^*)
 \end{align*}
be the unnormalized (smoothly) induced representation of
$\r{GL}(V^+)$ of the representation $\tau\otimes\sigma\otimes\tau^*$
viewed as a representation of $P_0$ through the projection $P_0\R
M_0$. We have the following proposition which is similar to
\cite[Theorem 15.1]{GGP}.

 \begin{prop}
 With the notations as above and let $k$ be non-archimedean, we have
  \begin{align*}
  \r{Hom}_{\r{GL}(V)}(\r{I}\(\tau,\sigma,\tau^*\)\otimes\pi,\d{C})=\r{Hom}_H(\pi\otimes\sigma\delta_W^{-\frac{1}{2}},\nu)
  \end{align*}
 as long as $\widetilde{\pi}$ does not belong to the Bernstein component
 of $\r{GL}(V)$ associated to the data
 $(\r{GL}(Y_0)\times M,\tau\otimes\varsigma)$ and $(M^*\times\r{GL}(Y_0^*),\varsigma^*\otimes\tau^*)$
 where $M$ (resp. $M^*$) is any Levi subgroup of $\r{GL}(Z)$ (resp. $\r{GL}(Z^*)$) and $\varsigma$
 (resp. $\varsigma^*$) is any irreducible supercuspidal representation of $M$ (resp. $M^*$). Here $V=Y_0\oplus
 Z$ (resp. $V=Z^*\oplus Y_0^*$) with $\r{dim}(Y_0)=\r{dim}(Y)=r+1$ (resp. $\r{dim}(Y_0^*)=\r{dim}(Y^*)=r^*+1$).
 \end{prop}
 \begin{proof}
 We need to calculate the restriction
 $\Pi:=\r{I}\(\tau,\sigma,\tau^*\)|_{\r{GL}(V)}$. By the Bruhat
 decomposition, there are six elements in the double coset
 $\r{GL}(V)\B\r{GL}(V^+)/P_0$ whose representatives are:
  \begin{description}
    \item[Big cell:]$g_0=\b{1}_{n+1}$.
    \item[Medium cells:]$g_1$ sends $\s{F}_0$ to $\s{F}_1:0\subset Y_1\subset Y_1\oplus
    W_1$ with $Y_1=Y$ and $E^+\subset Y_1\oplus W_1$;\\
    $g_2$ sends $\s{F}_0$ to $\s{F}_2:0\subset Y_2\subset Y_2\oplus
    W_2$ with $Y_2\subset V$ and $E^+\not\subset Y_2\oplus W_2\not\subset
    V$.
    \item[Small cells:]$g_3$ sends $\s{F}_0$ to $\s{F}_3:0\subset Y_3\subset Y_3\oplus
    W_3$ with $E^+\subset Y_3$;\\
    $g_4$ sends $\s{F}_0$ to $\s{F}_4:0\subset Y_4\subset Y_4\oplus
    W_4$ with $Y_4\subset V$ and $E^+\subset Y_4\oplus W_4$;\\
    $g_5$ sends $\s{F}_0$ to $\s{F}_5:0\subset Y_5\subset Y_5\oplus
    W_5$ with $Y_5\oplus W_5\subset V$.
  \end{description}
 Let $P_i$ ($i=0,1,2,3,4,5$) be the parabolic subgroup of $\r{GL}(V^+)$
 stabilizing $\s{F}_i$, $Q_i=P_i\cap\r{GL}(V)$ and $\pi_i=\(\tau\otimes\sigma\otimes\tau^*\)|_{Q_i}$.
 By Mackey theory, there is a filtration
 $0\subset\Pi_0\subset\Pi_1\subset\Pi_2=\Pi$ such that
  \begin{align*}
  &\Pi_0\cong\r{cInd}_{Q_0}^{\r{GL}(V)}\pi_0;\qquad\Pi_1/\Pi_0=\r{cInd}_{Q_1}^{\r{GL}(V)}\pi_1
  \oplus\r{cInd}_{Q_2}^{\r{GL}(V)}\pi_2;\\
  &\Pi_2/\Pi_1\cong\r{cInd}_{Q_3}^{\r{GL}(V)}\pi_3\oplus\r{cInd}_{Q_4}^{\r{GL}(V)}\pi_4
  \oplus\r{cInd}_{Q_5}^{\r{GL}(V)}\pi_5
  \end{align*}
 where $\r{cInd}$ means the (unnormalized smooth) induction with
 compact support. Applying the functor
 $\r{Hom}_{\r{GL}(V)}(-,\widetilde{\pi})$, we have the following
 exact sequence:
  \begin{align*}
  &\xymatrix{
    0 \ar[r] & \r{Hom}_{\r{GL}(V)}\(\bigoplus_{i=3}^5\r{cInd}_{Q_i}^{\r{GL}(V)}\pi_i,\widetilde{\pi}\)
    \ar[r] & \r{Hom}_{\r{GL}(V)}\(\Pi,\widetilde{\pi}\)
    \ar[r] & \r{Hom}_{\r{GL}(V)}\(\Pi_1,\widetilde{\pi}\) }\\
  &\xymatrix{
    \;\; \ar[r] & \r{Ext}^1_{\r{GL}(V)}\(\bigoplus_{i=3}^5\r{cInd}_{Q_i}^{\r{GL}(V)}\pi_i,\widetilde{\pi}\)
    }.
  \end{align*}
 But for $i=3,4$ (resp. $4,5$), $\r{cInd}_{Q_i}^{\r{GL}(V)}\pi_i$ is of the form
 $\r{Ind}_{P_0^*}^{\r{GL}(V)}\varsigma^*\otimes\tau^*$ (resp. $\r{Ind}_{P_0}^{\r{GL}(V)}\tau\otimes\varsigma$)
 where $P_0^*$ (resp. $P_0$) is some parabolic subgroup whose Levi is
 $M^*\times\r{GL}(Y_0^*)$ (resp. $\r{GL}(Y_0)\times M$) and
 $\varsigma^*$ (resp. $\varsigma$) is some smooth representation of $M^*$
 (resp. $M$). By our assumption on $\tau$ and $\tau^*$, the second and the fifth (last) terms
 in the above exact sequence are both zero. Hence
 $\r{Hom}_{\r{GL}(V)}(\Pi,\widetilde{\pi})\cong\r{Hom}_{\r{GL}(V)}(\Pi_1,\widetilde{\pi})$.

 Applying the same functor again, we have:
  \begin{align*}
  &\xymatrix{
    0 \ar[r] & \r{Hom}_{\r{GL}(V)}\(\bigoplus_{i=1}^2\r{cInd}_{Q_i}^{\r{GL}(V)}\pi_i,\widetilde{\pi}\)
    \ar[r] & \r{Hom}_{\r{GL}(V)}\(\Pi_1,\widetilde{\pi}\)
    \ar[r] & \r{Hom}_{\r{GL}(V)}\(\Pi_0,\widetilde{\pi}\) }\\
  &\xymatrix{
    \;\; \ar[r] & \r{Ext}^1_{\r{GL}(V)}\(\bigoplus_{i=1}^2\r{cInd}_{Q_i}^{\r{GL}(V)}\pi_i,\widetilde{\pi}\)
    }.
  \end{align*}
 For $i=1$ (resp. $2$), $\r{cInd}_{Q_i}^{\r{GL}(V)}\pi_i$ is of the form
 $\r{Ind}_{P_0^*}^{\r{GL}(V)}\varsigma^*\otimes\tau^*$ (resp. $\r{Ind}_{P_0}^{\r{GL}(V)}\tau\otimes\varsigma$).
 Again by our assumption on $\tau$ and $\tau^*$, the second and the fifth (last) terms
 in the above exact sequence are both zero. Hence
  \begin{align*}
  \r{Hom}_{\r{GL}(V)}(\Pi,\widetilde{\pi})\cong\r{Hom}_{\r{GL}(V)}(\Pi_1,\widetilde{\pi})
  \cong\r{Hom}_{\r{GL}(V)}(\Pi_0,\widetilde{\pi})
  \cong\r{Hom}_{\r{GL}(V)}(\r{cInd}_{Q_0}^{\r{GL}(V)}\pi_0,\widetilde{\pi}).
  \end{align*}

 Recall that $P_{r,m+1,r}$ is the parabolic subgroup of
 $\r{GL}(V)$ stabilizing $0\subset X\subset X\oplus W\oplus E\subset
 V$. Hence it contains $Q_0$ as a subgroup and moreover, is equal to
 $U_{r,m+1,r^*}\rtimes\(\r{GL}(X)\times\r{GL}(W)\times\r{GL}(X^*)\)$.
 The natural projection:
 $P_0\R\r{GL}(Y)\times\r{GL}(W)\times\r{GL}(Y^*)$ induces the
 following commutative diagram with exact rows and injective
 vertical arrows:
  \begin{align*}
  \xymatrix{
    0 \ar[r] & N_0 \ar[r] & P_0
    \ar[r] & \r{GL}(Y)\times\r{GL}(W)\times\r{GL}(Y^*) \ar[r] & 0 \\
    0 \ar[r] & N_0\cap Q_0 \ar[u] \ar[r] & Q_0 \ar[u]
    \ar[r] & P_Y\times\r{GL}(W)\times P_{Y^*}  \ar[u] \ar[r] & 0\\
    0 \ar[r] & N_0\cap Q_0 \ar@{=}[u] \ar[r] & U_{r,m+1,r^*} \ar[u] \ar[r]
    & \r{Hom}(\LA v_0\RA, X)\times\r{Hom}(X^*,\LA v_0^*\RA) \ar[u] \ar[r] & 0   }
  \end{align*}
 where $N_0$ is the maximal unipotent subgroup of $P_0$, $P_Y\subset\r{GL}(Y)$ is
 the mirabolic subgroup stabilizing $X$ and $P_{Y^*}\subset\r{GL}(Y^*)$ is the
 mirabolic subgroup fixing $v^*_0$. The proof of this is similar to
 \cite[Lemma 15.2]{GGP}. By the diagram, we have
  \begin{align*}
  \(\tau\otimes\sigma\otimes\tau^*\)|_{Q_0}=\tau|_{P_Y}\otimes\sigma\otimes\tau^*|_{P_{Y^*}}.
  \end{align*}
 By a result of Gelfand-Kazhdan, we have
  \begin{align*}
  \tau|_{P_Y}\cong\r{cInd}_{U_Y}^{P_Y}\lambda;\qquad\tau^*|_{P_{Y^*}}\cong\r{cInd}_{U_{Y^*}}^{P_{Y^*}}\lambda^*
  \end{align*}
 where $U_Y$ (resp. $U_{Y^*}$) is the unipotent radical of a Borel subgroup of $\r{GL}(Y)$ (resp.
 $\r{GL}(Y^*)$) satisfying $U_X\subset U_Y\subset P_Y$ (resp. $U_{X^*}\subset U_{Y^*}\subset
 P_{Y^*}$), and $\lambda$ (resp. $\lambda^*$) is a generic character
 of $U_Y$ (resp. $U_{Y^*}$).

 By our choice of unipotent radical, it is clear that, the pre-image
 of $U_Y\times\r{GL}(W)\times U_{Y^*}$ in $Q_0$ is the subgroup
  \begin{align*}
  H=U_{r,m+1,r^*}\rtimes\(U_X\times\r{GL}(W)\times U_{X^*}\)
  \end{align*}
 and the pull-back of the representation
 $\lambda\otimes\sigma\otimes\lambda^*$ is just
 $\sigma\delta_W^{-\frac{1}{2}}\otimes\widetilde{\nu}$. Hence, by induction in stages, we
 have
  \begin{align*}
  \r{cInd}_{Q_0}^{\r{GL}(V)}\pi_0\cong\r{cInd}_H^{\r{GL}(V)}\sigma\delta_W^{-\frac{1}{2}}\otimes\widetilde{\nu}
  \end{align*}
 and by Frobenius reciprocity, we conclude that
  \begin{align*}
  \r{Hom}_{\r{GL}(V)}\(\r{I}\(\tau,\sigma,\tau^*\)\otimes\pi,\d{C}\)
  \cong\r{Hom}_H\(\pi\otimes\sigma\delta_W^{-\frac{1}{2}},\nu\).
  \end{align*}
 \end{proof}

 \begin{proof}[Proof of Theorem \ref{theo2bm} for the uniqueness part when $k$ is non-archimedean]
 We choose an irreducible supercuspidal representation $\tau$ (resp. $\tau^*$) of $\r{GL}(Y)$ (resp. $\r{GL}(Y^*)$)
 satisfying the assumption in the above proposition. After twisting unramified
 characters (which still satisfy the assumption), we may
 assume that the induced representation $\r{I}\(\tau,\sigma\delta_W^{\frac{1}{2}},\tau^*\)$
 is irreducible. Then by the above proposition and \cite{AGRS10}, we have
  \begin{align*}
  \r{dim}_{\d{C}}\r{Hom}_H\(\pi\otimes\sigma,\nu\)
  =\r{dim}_{\d{C}}\r{Hom}_{\r{GL}(V)}\(\r{I}\(\tau,\sigma\delta_W^{\frac{1}{2}},\tau^*\)\otimes\pi,\d{C}\)\leq 1.
  \end{align*}
 \end{proof}

\begin{defn}
A nontrivial element in the space $\r{Hom}_H(\pi\otimes\sigma,\nu)$
is called an \emph{$(r,r^*)$-Bessel model} of $\pi\otimes\sigma$.
When $r=r^*$,
hence $n-m$ is odd, it is just the one defined in \cite{GGP}.\\
\end{defn}

\subsection{Bessel integrals, functional equations and $L$-functions}
\label{sec2bi}

In this section, we consider the global situation. Hence $k$ will be
a number field and $\d{A}$ is its ring of ad\`{e}les. Let
$|\;|_{\d{A}}=\prod_{v\in\d{M}_k}|\;|_v$. For any $v\in\d{M}_k$, we
denote $k_v$ the completion of $k$ at $v$. We denote $\f{o}$ (resp.
$\f{o}_v$) the ring of integers of $k$ (resp. $k_v$ for $v$ finite).
For any algebraic group $G$ over $k$, we denote $G_v=G(k_v)$ the
local Lie group for $v\in\d{M}_k$. If $G$ is reductive, we denote
$\c{A}(G)$ (resp. $\c{A}_0(G)$) the space of automorphic forms
(resp. cusp forms) of $G$ which is a
representation of $G(\d{A})$ by right translation $\rho$.\\

We define the pair $(H,\nu)$ as in the local case. Hence $\ell_X$,
$\ell_{X^*}$ and $\ell_W$ are defined over the number field $k$;
$\psi$ is a nontrivial character of $k\B\d{A}$; $\lambda$ is a
generic character $(U_{X}\times U_{X^*})(k)\B(U_{X}\times
U_{X^*})(\d{A})\R\d{C}^1$. Let $\pi$ (resp. $\sigma$) be an
irreducible cuspidal automorphic representation of
$\r{GL}(V)(\d{A})$ (resp. $\r{GL}(W)(\d{A})$). Then $\pi$ (resp.
$\sigma$) is isomorphic to a unique irreducible sub-representation
$\c{A}_{\pi}$ (resp. $\c{A}_{\sigma}$) of $\c{A}_0(\r{GL}(V))$
(resp. $\c{A}_0(\r{GL}(W))$).

\begin{defn}
The following absolutely convergent integral is called an
\emph{$(r,r^*)$-Bessel period} of $\pi\otimes\sigma$ (for a pair
$(H,\nu)$):
 \begin{align*}
 \c{B}^{\nu}_{r,r^*}(\varphi_{\pi},\varphi_{\sigma}):=\intl_{H(k)\B
 H(\d{A})}\varphi_{\pi}(\vep(h))\varphi_{\sigma}(\kappa(h))\nu(h)^{-1}\r{d}h,\qquad
 \varphi_\pi\in\c{A}_{\pi},\varphi_{\sigma}\in\c{A}_{\sigma}
 \end{align*}
where $\r{d}h$ is the Tamagawa measure on $H(\d{A})$. If there exist
$\varphi_{\pi}$, $\varphi_{\sigma}$ such that
$\c{B}^{\nu}_{r,r^*}(\varphi_{\pi},\varphi_{\sigma})\neq0$, then we
say $\pi\otimes\sigma$ has a nontrivial $(r,r^*)$-Bessel period.
\end{defn}

It is obvious that $\c{B}^{\nu}_{r,r^*}$ defines an element in
 \begin{align*}
 \r{Hom}_{H(\d{A})}(\pi\otimes\sigma,\nu)=\bigotimes_{v\in\d{M}_k}\r{Hom}_{H_v}(\pi_v\otimes\sigma_v,\nu_v).
 \end{align*}
Since the later space has the multiplicity one property, we expect
that the Bessel period is Eulerian. We now show that this is true.\\

We can choose a basis $\1 v_1,...,v_r\2$ of $X$ such that
\begin{itemize}
  \item The homomorphism $\ell_X:X\R k$ is given by the coefficient of $v_r$
  under the above basis;
  \item $U_X$ is the maximal unipotent subgroup of the parabolic
  subgroup $P_X$ stabilizing the complete flag
   $0\subset \LA v_1\RA\subset \LA v_1,v_2\RA\subset\cdots\subset\LA
   v_1,...,v_r\RA=X$;
  \item The generic character $\lambda|_{U_X}$ is given by
   \begin{align*}
   \lambda(u)=\psi\(u_{1,2}+u_{2,3}+\cdots+u_{r-1,r}\)
   \end{align*}
  where
   \begin{align}\label{2ux}
   u=\[
       \begin{array}{cccccc}
         1 & u_{1,2} & u_{1,3} & \cdots & u_{1,r-1} & u_{1,r} \\
          & 1 & u_{2,3} & \cdots & u_{2,r-1} & u_{2,r} \\
          &  &  1 & \cdots & u_{3,r-1} & u_{3,r}\\
          &  &    & \ddots & \vdots & \vdots \\
          &  &   & & 1 & u_{r-1,r} \\
          &  &   & &  & 1 \\
       \end{array}
     \]\in U_X(\d{A})
   \end{align}
  under the above basis.
\end{itemize}

Similarly, we can also choose a basis $\1 v^*_{r^*},...,v^*_1\2$ of
$X^*$ such that
\begin{itemize}
  \item The homomorphism $\ell_{X^*}:k\R X^*$ is given by $x\T
  cxv^*_{r^*}$ for some $c\neq0$ determined later;
  \item $U_{X^*}$ is the maximal unipotent subgroup of the parabolic
  subgroup $P_{X^*}$ stabilizing the complete flag
   $0\subset \LA v^*_{r^*}\RA\subset \LA v^*_{r^*},v^*_{r^*-1}\RA\subset\cdots\subset\LA
   v^*_{r^*},...,v^*_1\RA=X^*$;
  \item The generic character $\lambda|_{U_{X^*}}$ is given by
   \begin{align*}
   \lambda(u^*)=\psi\(u^*_{r^*,r^*-1}+u^*_{r^*-1,r^*-2}+\cdots+u^*_{2,1}\)
   \end{align*}
  where
   \begin{align}\label{2uxstar}
   u^*=\[
       \begin{array}{cccccc}
         1 & u^*_{r^*,r^*-1} & u^*_{r^*,r^*-2} & \cdots & u^*_{r^*,2} & u^*_{r^*,1} \\
          & 1 & u^*_{r^*-1,r^*-2} & \cdots &u^*_{r^*-1,2} & u^*_{r^*-1,1} \\
          &  & 1 & \cdots & u^*_{r^*-2,2} & u^*_{r^*-2,1}\\
          &  & & \ddots & \vdots & \vdots \\
          &  & &  & 1 & u^*_{2,1} \\
          &  &  & &  & 1 \\
       \end{array}
     \]\in U_{X^*}(\d{A})
   \end{align}
  under the above basis.
\end{itemize}

Moreover, we can choose a basis $\1 w_1,...,w_m\2$ of $W$ and $\1
w_0\2$ of $E$ such that the bilinear form $\ell_W:(W\oplus
E)+(W^{\vee}\oplus E^{\vee})\R k$ is given by
$\ell_W(w_i)=\ell_W(w_i^{\vee})=0$ ($1\leq i\leq m$) and
$\ell_W(w_0^{\vee})=1$, where $\1
w_1^{\vee},...,w_m^{\vee},w_0^{\vee}\2$ is the dual basis. Let
$c=\ell_W(w_0)^{-1}$.

We write elements in $\r{GL}(V)$ in the matrix form under the basis:
 \begin{align}\label{2basis}
 \1 w_1,...,w_m,v_1,...,v_r,w_0,v^*_{r^*},...,v^*_1\2
 \end{align}
and view $\r{GL}(W)$ as a subgroup of $\r{GL}(V)$. Then the image of
$H(\d{A})$ in $\r{GL}(V)(\d{A})$ consists of following matrices:
 \begin{align}\label{2huge}
 h=h\(n,n^*,b;u,u^*;g\)=
  \[
   \begin{array}{c|c|c|c}
     \text{\Large{$g$}} & \multicolumn{2}{c|}{} &
             \begin{array}{ccc}
               n^*_{1,r^*} & \cdots & n^*_{1,1} \\
               \vdots &  & \vdots \\
               n^*_{m,r^*} & \cdots & n^*_{m,1}
             \end{array} \\ \cline{1-4}
     \begin{array}{ccc}
      n_{1,1} & \cdots & n_{1,m} \\
      \vdots &  & \vdots \\
      n_{r,1} & \cdots & n_{r,m}
     \end{array} & \begin{array}{ccc}
                      & \text{\Large{$u$}} &
                   \end{array}
      &
           \begin{array}{c}
             n_{1,0} \\
             \vdots \\
             n_{r,0}
           \end{array} & \text{\Large{$b$}} \\ \cline{1-4}
     \multicolumn{2}{c|}{} & \begin{array}{c}
                                \\
                               1 \\
                                \\
                             \end{array}
      &
        \begin{array}{ccc}
          n^*_{0,r^*} & \cdots & n^*_{0,1}
        \end{array} \\ \cline{3-4}
     \multicolumn{3}{c|}{} & \begin{array}{c}
                               \\
                               \\
                              \text{\Large{$u^*$}} \\
                              \\
                              \\
                            \end{array}
   \end{array}
  \]
 \end{align}
where
 \begin{align*}
 n&=\[
 \begin{array}{cccc}
   n_{1,1} & \cdots & n_{1,m} & n_{1,0} \\
   \vdots &  & \vdots & \vdots \\
   n_{r,1} & \cdots & n_{r,m} & n_{r,0}
 \end{array}
 \]\in\r{Hom}(W\oplus E,X)(\d{A}),\\
 n^*&=\[
 \begin{array}{ccc}
   n^*_{1,r^*} & \cdots & n^*_{1,1} \\
   \vdots &  & \vdots \\
   n^*_{m,r^*} & \cdots & n^*_{m,1} \\
   n^*_{0,r^*} & \cdots & n^*_{0,1}
 \end{array}
 \]\in\r{Hom}(X^*,W\oplus E)(\d{A}),
 \end{align*}
$b\in\r{Hom}(X^*,X)(\d{A})$, $u\in U_X(\d{A})$, $u^*\in
U_{X^*}(\d{A})$ and $g\in\r{GL}(W)(\d{A})$. Hence $u$ and $u^*$ are
upper triangular matrices as in \eqref{2ux} and \eqref{2uxstar}; the
character $\nu$ on $H(\d{A})$ is given by
 \begin{align*}
 \nu(h)&=\nu\(h\(n,n^*,b;u,u^*;g\)\)\\
 &=|\det g|_{\d{A}}^\frac{r-r^*}{2}\psi\(u_{1,2}+\cdots+u_{r-1,r}+n_{r,0}+n^*_{0,r^*}+u^*_{r^*,r^*-1}+\cdots+u^*_{2,1}\).
 \end{align*}
Let $U_{1^r,m+1,1^{r^*}}=U_{r,m+1,r^*}\rtimes(U_X\times U_{X^*})$ be
a unipotent subgroup of $\r{GL}(V)$, then
 \begin{align*}
 U_{1^r,m+1,1^{r^*}}(\d{A})=\1
 \underline{u}=\underline{u}\(n,n^*,b;u,u^*\):=h\(n,n^*,l;u,u^*;\b{1}_m\)\2
 \end{align*}
and we denote $\r{d}\underline{u}$ the product measure. Then we have
$\r{d}h=|\det g|_{\d{A}}^{r-r^*}\r{d}\underline{u}\r{d}g$. Let us
simply write
 \begin{align*}
 \underline{\psi}(\underline{u})=\psi\(u_{1,2}+\cdots+u_{r-1,r}+n_{r,0}+n^*_{0,r^*}+u^*_{r^*,r^*-1}+\cdots+u^*_{2,1}\)
 \end{align*}
and identify $\r{GL}(V)$ (resp. $\r{GL}(W)$) with $\r{GL}_{n,k}$
(resp. $\r{GL}_{m,k}$) under the basis \eqref{2basis}. Then
 \begin{align}\label{2bs1}
 \c{B}^{\nu}_{r,r^*}(\varphi_{\pi},\varphi_{\sigma})=\intl_{\r{GL}_m(k)\B\r{GL}_m(\d{A})}
 \intl_{U_{1^r,m+1,1^{r^*}}(k)\B
 U_{1^r,m+1,1^{r^*}}(\d{A})}\varphi_{\pi}\(\underline{u}g\)\varphi_{\sigma}(g)|\det
 g|_{\d{A}}^{\frac{r-r^*}{2}}\overline{\underline{\psi}(\underline{u})}\r{d}\underline{u}\r{d}g.
 \end{align}
We insert an $s$-variable as
 \begin{align*}
 \c{B}^{\nu}_{r,r^*}(s;\varphi_{\pi},\varphi_{\sigma})=\intl_g\intl_{\underline{u}}
 \varphi_{\pi}\(\underline{u}g\)\varphi_{\sigma}(g)|\det
 g|_{\d{A}}^{s-\frac{1}{2}+\frac{r-r^*}{2}}\overline{\underline{\psi}(\underline{u})}\r{d}\underline{u}\r{d}g
 \end{align*}
and call it the \emph{$(r,r^*)$-Bessel integral}. We are going to
use the Fourier transform. Let
 \begin{align*}
 L_{r+1}=\1\underline{u}(n_0,n^*,b;u,u^*)\left|n_0=\[\begin{array}{cccc}
                                                     0 & \cdots & 0 & n_{1,0} \\
                                                     \vdots &  & \vdots & \vdots \\
                                                     0 & \cdots & 0 &
                                                     n_{r,0}
                                                   \end{array}\]
 \right.\2
 \end{align*}
be a subgroup of $U_{1^r,m+1,1^{r^*}}$. Let
 \begin{align*}
 L_r=\1\underline{l}=\underline{l}(l_r;n_0,n^*,b;u,u^*)|l_r=\:^{\r{t}}\[l_{1,r},...,l_{m,r}\]\2
 \end{align*}
where the matrix $\underline{l}(l_r;n_0,n^*,b;u,u^*)$ is the one
obtained from $\underline{u}(n_0,n^*,b;u,u^*)$ by adding the column
$l_r$ above the entry $n_{1,0}$ as in \eqref{2huge}. It is clear
that $L_r/L_{r+1}$ is isomorphic to $k^m$ which may be identified
with the set of column vector $l_r$. By the Fourier inverse formula
for $(k\B\d{A})^m$, we have
 \begin{align}\label{2bs2}
 \eqref{2bs1}=\intl_g\intl_{\r{M}_{r,m}(k\B\d{A})}\sum_{\epsilon_i\in k}\intl_{L_r(k)\B L_r(\d{A})}
 \varphi_{\pi}(\underline{l}\:\underline{n}g)\varphi_{\sigma}(g)|\det
 g|_{\d{A}}^{s-\frac{1}{2}+\frac{r-r^*}{2}}\overline{\psi\(\sum_{i=1}^m\epsilon_il_{i,r}\)}
 \overline{\underline{\psi}(\underline{l})}\r{d}\underline{l}\r{d}\underline{n}\r{d}g
 \end{align}
where $\underline{n}$ represents the element
$\underline{u}\(n,0,0;\b{1}_r,\b{1}_{r^*}\)$ with
 \begin{align*}
 n=\[\begin{array}{cccc}
       n_{1,1} & \cdots & n_{1,m} & 0 \\
       \vdots &  & \vdots & \vdots \\
       n_{r,1} & \cdots & n_{r,m} & 0
     \end{array}
 \]
 \end{align*}
and we view $\underline{\psi}$ as a character on $L_r(\d{A})$
through the natural quotient $L_{r+1}(\d{A})$. We also extend the
measure on $L_{r+1}(\d{A})$ to $L_r(\d{A})$ by the self-dual measure
on $\d{A}^m$. Let $\underline{\epsilon}$ be an element like
$\underline{u}(\epsilon,0,0;\b{1}_r,\b{1}_{r^*})$ where
 \begin{align*}
 \epsilon=\[\begin{array}{cccc}
              0 & \cdots & 0 & 0 \\
              \vdots &  & \vdots & \vdots \\
              0 & \cdots & 0 & 0\\
              \epsilon_1 & \cdots & \epsilon_m & 0
            \end{array}
 \],
 \end{align*}
then if we conjugate $\underline{l}$ by $\underline{\epsilon}$ from
left to right, we can incorporate $\psi\(\sum\epsilon_il_{i,r}\)$
into $\underline{\psi}(\underline{l})$ and collapse the summation
over $\epsilon_i\in k$. The result is
 \begin{align}\label{2bs3}
 \eqref{2bs2}=\intl_g\intl_{\substack{\r{M}_{r-1,m}(k\B\d{A})\\ \times\r{M}_{1,m}(\d{A})}}
 \intl_{L_r(k)\B L_r(\d{A})}\varphi_{\pi}(\underline{l}\:\underline{n}g)\varphi_{\sigma}(g)|\det
 g|_{\d{A}}^{s-\frac{1}{2}+\frac{r-r^*}{2}}
 \overline{\underline{\psi}(\underline{l})}\r{d}\underline{l}\r{d}\underline{n}\r{d}g.
 \end{align}
Similarly, we can introduce the subgroup $L_{r-1},...,L_1$ and
repeat the process $r-1$ more times. Then we finally get
 \begin{align}\label{2bs4}
 \eqref{2bs3}=\intl_g\intl_{\r{M}_{r,m}(\d{A})}\intl_{L_1(k)\B L_1(\d{A})}
 \varphi_{\pi}(\underline{l}\:\underline{n}g)\varphi_{\sigma}(g)|\det
 g|_{\d{A}}^{s-\frac{1}{2}+\frac{r-r^*}{2}}
 \overline{\underline{\psi}(\underline{l})}\r{d}\underline{l}\r{d}\underline{n}\r{d}g.
 \end{align}
Here,
 \begin{align*}
 L_1=\1\underline{l}=\underline{l}(l_1;n_0,n^*,b;u,u^*)\left|l_1=\[\begin{array}{cccc}
         l_{1,1} & \cdots & l_{1,r-1} & l_{1,r} \\
         \vdots &  & \vdots & \vdots\\
         l_{m,1} & \cdots & l_{m,r-1} & l_{m,r}
       \end{array}
 \]\right.\2
 \end{align*}
where $\underline{l}(l_1;n_0,n^*,b;u,u^*)$ is the one obtained from
$\underline{u}(n_1,n^*,b;u,u^*)$ by adding $l_1$ on the sub-row
$[u_{1,2},...,u_{1,r},n_{1,0}]$. In fact, $L_1$ is the maximal
unipotent subgroup of the (standard) parabolic subgroup stabilizing
the flag
 \begin{align*}
 0\subset W\oplus\LA v_1\RA\subset W\oplus\LA
 v_1,v_2\RA\subset\cdots\subset V.
 \end{align*}
Interchanging the order of $g$ and $\underline{n}$, we have
 \begin{align}\label{2bs5}
 \eqref{2bs4}&=\intl_{\r{M}_{r,m}(\d{A})}\intl_{g}\intl_{L_1(k)\B L_1(\d{A})}
 \varphi_{\pi}(\underline{l}g\underline{n})\varphi_{\sigma}(g)|\det
 g|_{\d{A}}^{s-\frac{1}{2}-\frac{r+r^*}{2}}
 \overline{\underline{\psi}(\underline{l})}\r{d}\underline{l}\r{d}g\r{d}\underline{n}\notag\\
 &=\intl_{\r{M}_{r,m}(\d{A})}\intl_{g}|\det
 g|_{\d{A}}^{s-\frac{1}{2}-\frac{n-m-1}{2}}\intl_{L_1(k)\B L_1(\d{A})}
 \(\rho(\underline{n})\varphi_{\pi}\)(\underline{l}g)\varphi_{\sigma}(g)
 \overline{\underline{\psi}(\underline{l})}\r{d}\underline{l}\r{d}g\r{d}\underline{n}.
 \end{align}
Now the inner double integral is just the usual Rankin-Selberg
convolution (cf. \cite{JPSS83}). The following calculation is
well-known:
 \begin{align}\label{2rs1}
 &\intl_{\r{GL}_m(k)\B\r{GL}_m(\d{A})}|\det g|_{\d{A}}^{s-\frac{n-m}{2}}\intl_{L_1(k)\B L_1(\d{A})}
 \(\rho(\underline{n})\varphi_{\pi}\)(\underline{l}g)\varphi_{\sigma}(g)
 \overline{\underline{\psi}(\underline{l})}\r{d}\underline{l}\r{d}g \notag\\
 =&\intl_{g}\sum_{\gamma\in
 U_{1^m}(k)\B\r{GL}_m(k)}W^{\psi}_{\rho(\underline{n})\varphi_{\pi}}\(\[\begin{array}{cc}
                                                            \gamma & 0 \\
                                                            0 &
                                                            \b{1}_{n-m}
                                                          \end{array}
 \]\[\begin{array}{cc}
       g & 0 \\
       0 & \b{1}_{n-m}
     \end{array}
 \]\)\varphi_{\sigma}(g)|\det g|_{\d{A}}^{s-\frac{n-m}{2}}\r{d}g \notag\\
 =&\intl_{U_{1^m}(k)\B\r{GL}_m(\d{A})}W^{\psi}_{\rho(\underline{n})\varphi_{\pi}}\(\[\begin{array}{cc}
                                                            \gamma & 0 \\
                                                            0 &
                                                            \b{1}_{n-m}
                                                          \end{array}
 \]\[\begin{array}{cc}
       g & 0 \\
       0 & \b{1}_{n-m}
     \end{array}
 \]\)\varphi_{\sigma}(g)|\det g|_{\d{A}}^{s-\frac{n-m}{2}}\r{d}g,
 \end{align}
where $U_{1^m}$ is the standard maximal unipotent subgroup of
$\r{GL}_m$ and $W^{\psi}_{?}$ is $\psi$-Whittaker function.
Factorize the integral over $U_{1^m}(\d{A})$, we have
 \begin{align}\label{2rs2}
 \eqref{2rs1}=&\intl_{U_{1^m}(\d{A})\B\r{GL}_m(\d{A})}\intl_{U_{1^m}(k)\B
 U_{1^m}(\d{A})}\notag\\
 &W^{\psi}_{\rho(\underline{n})\varphi_{\pi}}\(\[\begin{array}{cc}
                                                            u & 0 \\
                                                            0 &
                                                            \b{1}_{n-m}
                                                          \end{array}
 \]\[\begin{array}{cc}
       g & 0 \\
       0 & \b{1}_{n-m}
     \end{array}
 \]\)\varphi_{\sigma}(g)\r{d}u\:|\det
 g|_{\d{A}}^{s-\frac{n-m}{2}}\r{d}g \notag\\
 =&\intl_{U_{1^m}(\d{A})\B\r{GL}_m(\d{A})}W^{\psi}_{\rho(\underline{n})\varphi_{\pi}}\(\[\begin{array}{cc}
       g & 0 \\
       0 & \b{1}_{n-m}
     \end{array}
 \]\)\intl_{U_{1^m}(k)\B U_{1^m}(\d{A})}\psi(u)\varphi_{\sigma}(g)\r{d}u\:|\det
 g|_{\d{A}}^{s-\frac{n-m}{2}}\r{d}g \notag\\
 =&\intl_{U_{1^m}(\d{A})\B\r{GL}_m(\d{A})}W^{\psi}_{\rho(\underline{n})\varphi_{\pi}}\(\[\begin{array}{cc}
       g & 0 \\
       0 & \b{1}_{n-m}
     \end{array}
 \]\)W^{\overline{\psi}}_{\varphi_{\sigma}}(g)|\det
 g|_{\d{A}}^{s-\frac{n-m}{2}}\r{d}g \notag\\
 =&\intl_{U_{1^m}(\d{A})\B\r{GL}_m(\d{A})}W^{\psi}_{\varphi_{\pi}}\(\[\begin{array}{cc}
       g & 0 \\
       0 & \b{1}_{n-m}
     \end{array}
 \]\underline{n}\)W^{\overline{\psi}}_{\varphi_{\sigma}}(g)|\det
 g|_{\d{A}}^{s-\frac{n-m}{2}}\r{d}g.
 \end{align}
Plug \eqref{2rs2} into \eqref{2bs5}, we have
 \begin{align}\label{2bs6}
 \eqref{2bs5}=\intl_{U_{1^m}(\d{A})\B\r{GL}_m(\d{A})}\intl_{\r{M}_{r,m}(\d{A})}
 W^{\psi}_{\varphi_{\pi}}\(\[\begin{array}{ccc}
                            g & 0 & 0 \\
                            x & \b{1}_r & 0 \\
                            0 & 0 & \b{1}_{n-m-r}
                          \end{array}
 \]\)W^{\overline{\psi}}_{\varphi_{\sigma}}(g)|\det
 g|_{\d{A}}^{s-\frac{n-m}{2}}\r{d}x\r{d}g.
 \end{align}
We denote the above integral by
$\Psi_r(s;W^{\psi}_{\varphi_{\pi}},W^{\overline{\psi}}_{\varphi_{\sigma}})$
which is absolutely convergent when $\Re(s)\gg0$. If we assume that
$W^{\psi}_{\varphi_{\pi}}=\otimes_vW_v$ and
$W^{\overline{\psi}}_{\varphi_{\sigma}}=\otimes_vW^-_v$ are
factorizable and let
 \begin{align*}
 \Psi_{v,r}(s;W_v,W^-_v)
 =\intl_{U_{1^m,v}\B\r{GL}_{m,v}}\intl_{\r{M}_{r,m,v}}W_v\(\[\begin{array}{ccc}
                            g_v & 0 & 0 \\
                            x_v & \b{1}_r & 0 \\
                            0 & 0 & \b{1}_{n-m-r}
                          \end{array}
 \]\)W^-_v(g_v)|\det
 g_v|_v^{s-\frac{n-m}{2}}\r{d}x_v\r{d}g_v.
 \end{align*}
Then in summary, we have for $\Re(s)$ large,
 \begin{align*}
 \c{B}^{\nu}_{r,r^*}(s;\varphi_{\pi},\varphi_{\sigma})=
 \Psi_r(s;W^{\psi}_{\varphi_{\pi}},W^{\overline{\psi}}_{\varphi_{\sigma}})=\prod_{v\in\d{M}_k}\Psi_{v,r}(s;W_v,W^-_v)
 \end{align*}
where $W_v\in\c{W}(\pi_v,\psi_v)$ (see Section \ref{sec6} for
notations); $W^-_v\in\c{W}(\sigma_v,\overline{\psi_v})$ and for
almost all finite places $v$, $W_v$, $W^-_v$ are unramified
satisfying $W_v(\b{1}_n)=W^-_v(\b{1}_m)=1$.\\

Now we discuss the functional equation of the Bessel integrals.
First, let us introduce some Weyl elements:
 \begin{align*}
 w_1=\b{1}_1;\qquad w_n=\[\begin{array}{cc}
                             & 1 \\
                            w_{n-1} &
                          \end{array}
 \]\;\qquad
 w_{n,m}=\[\begin{array}{cc}
               \b{1}_m &  \\
                & w_{n-m}
             \end{array}
 \].
 \end{align*}
We define $\iota$ the outer automorphism of $\r{GL}_n$ and
$\r{GL}_m$ by $\iota(g)=g^{\iota}:=\:^{\r{t}}g^{-1}$. Let
$\widetilde{\varphi_{\pi}}(g)=\varphi_{\pi}(g^{\iota})=\varphi_{\pi}(w_ng^{\iota})$
and
$\widetilde{\varphi_{\sigma}}(g)=\varphi_{\sigma}(g^{\iota})=\varphi_{\sigma}(w_mg^{\iota})$.
Then
 \begin{align}\label{2fe1}
 &\c{B}^{\nu}_{r,r^*}(s;\varphi_{\pi},\varphi_{\sigma}) \notag\\
 =&\intl_{\r{GL}_m(k)\B\r{GL}_m(\d{A})}
 \intl_{U_{1^r,m+1,1^{r^*}}(k)\B
 U_{1^r,m+1,1^{r^*}}(\d{A})}\varphi_{\pi}(\underline{u}g^{\iota})\varphi_{\sigma}(g^{\iota})|\det
 g^{\iota}|_{\d{A}}^{s-\frac{1}{2}+\frac{r-r^*}{2}}
 \overline{\underline{\psi}(\underline{u})}\r{d}\underline{u}\r{d}g \notag\\
 =&\intl_{\r{GL}_m(k)\B\r{GL}_m(\d{A})}\intl_{U_{1^r,m+1,1^{r^*}}(k)\B
 U_{1^r,m+1,1^{r^*}}(\d{A})}\widetilde{\varphi_{\pi}}(\underline{u}^{\iota}g)\widetilde{\varphi_{\sigma}}(g)|\det
 g|_{\d{A}}^{-s+\frac{1}{2}+\frac{r^*-r}{2}}\overline{\underline{\psi}(\underline{u})}\r{d}\underline{u}\r{d}g\notag\\
 =&\intl_g\intl_{U_{1^r,m+1,1^{r^*}}(k)\B
 U_{1^r,m+1,1^{r^*}}(\d{A})}\(\rho(w_{n,m})\widetilde{\varphi_{\pi}}\)(w_{n,m}\underline{u}^{\iota}w_{n,m}g)
 \widetilde{\varphi_{\sigma}}(g)|\det
 g|_{\d{A}}^{-s+\frac{1}{2}+\frac{r^*-r}{2}}\overline{\underline{\psi}(\underline{u})}\r{d}\underline{u}\r{d}g.
 \end{align}
Since we have
$w_{n,m}U_{1^r,m+1,1^{r^*}}w_{n,m}=U_{1^{r^*},m+1,1^r}$ and
$\overline{\underline{\psi}(\underline{u})}=\underline{\psi}\(w_{n,m}\underline{u}^{\iota}w_{n,m}\)$,
we have
 \begin{align*}
 \eqref{2fe1}&=\intl_g\intl_{U_{1^{r^*},m+1,1^r}(k)\B
 U_{1^{r^*},m+1,1^r}(\d{A})}\(\rho(w_{n,m})\widetilde{\varphi_{\pi}}\)
 (\underline{u}g)\widetilde{\varphi_{\sigma}}(g)|\det
 g|_{\d{A}}^{1-s-\frac{1}{2}+\frac{r^*-r}{2}}\underline{\psi}(\underline{u})\r{d}\underline{u}\r{d}g\\
 &=\c{B}^{\overline{\nu}}_{r^*,r}(1-s;\rho(w_{n,m})\widetilde{\varphi_{\pi}},\widetilde{\varphi_{\sigma}}).
 \end{align*}
In summary, we have the following

 \begin{theo}
 The Bessel integrals are holomorphic in $s$ and satisfy the
 following functional equation
  \begin{align*}
  \c{B}^{\nu}_{r,r^*}(s;\varphi_{\pi},\varphi_{\sigma})
  =\c{B}^{\overline{\nu}}_{r^*,r}(1-s;\rho(w_{n,m})\widetilde{\varphi_{\pi}},\widetilde{\varphi_{\sigma}}).
  \end{align*}
 Let $\widetilde{W^{\psi}_{?}}(g)=W^{\psi}_{?}(w_ng^{\iota})\in\c{W}(\widetilde{\pi},\overline{\psi})$
 and similar for the one on $\r{GL}_m$.
 If $W^{\psi}_{\varphi_{\pi}}=\otimes_vW_v$ and $W^{\overline{\psi}}_{\varphi_{\sigma}}=\otimes_vW^-_v$ are
 factorizable, then $\widetilde{W^{\psi}_{\varphi_{\pi}}}=\otimes_v\widetilde{W_v}$ and
 $\widetilde{W^{\overline{\psi}}_{\varphi_{\sigma}}}=\otimes_v\widetilde{W^-_v}$
 are also factorizable with $\widetilde{W_v}(g)=W_v(w_ng^{\iota})$ and
 similar for $\widetilde{W^-_v}$. Then for $\Re(s)$ large,
  \begin{align*}
  \c{B}^{\nu}_{r,r^*}(s;\varphi_{\pi},\varphi_{\sigma})&=
  \Psi_r(s;W^{\psi}_{\varphi_{\pi}},W^{\overline{\psi}}_{\varphi_{\sigma}})
  =\prod_{v\in\d{M}_k}\Psi_{v,r}(s;W_v,W^-_v);\\
  \c{B}^{\overline{\nu}}_{r^*,r}(s;\rho(w_{n,m})\widetilde{\varphi_{\pi}},\widetilde{\varphi_{\sigma}})&=
  \Psi_{r^*}(s;\rho(w_{n,m})\widetilde{W^{\psi}_{\varphi_{\pi}}},\widetilde{W^{\overline{\psi}}_{\varphi_{\sigma}}})
  =\prod_{v\in\d{M}_k}\Psi_{v,r^*}(s;\rho(w_{n,m})\widetilde{W_v},\widetilde{W^-_v}).
  \end{align*}
 In particular, when $r=r^*$ (hence $n-m$ is odd), the $(r,r)$-Bessel
 integral itself has a functional equation.\\
 \end{theo}

By Proposition \eqref{prop6jpss} and \eqref{prop6unr}, we have the
following theorem, which confirms \cite[Conjecture 24.1]{GGP} for
the Bessel periods of split unitary groups, i.e., general linear
groups.

\begin{theo}
(1) Let $\pi$ (resp. $\sigma$) be an irreducible cuspidal
automorphic representation of $\r{GL}(V)(\d{A})$ (resp.
$\r{GL}(W)(\d{A})$). For any $(r,r^*)$ such that $r+r^*=n-m-1$ and
the automorphic representation $\nu$ introduced above, we have, for
$\varphi_{\pi}\in\c{A}_{\pi}$ and
$\varphi_{\sigma}\in\c{A}_{\sigma}$ such that
$W^{\psi}_{\varphi_{\pi}}=\otimes_vW_v$ and
$W^{\overline{\psi}}_{\varphi_{\sigma}}=\otimes_vW^-_v$ are
factorizable,
 \begin{align*}
 \c{B}^{\nu}_{r,r^*}(\varphi_{\pi},\varphi_{\sigma})=L(\frac{1}{2},\pi\times\sigma)
 \prod_{v\in\d{M}_k}\left.\frac{\Psi_{v,r}(s;W_v,W^-_v)}{L_v(s,\pi_v\times\sigma_v)}\right|_{s=\frac{1}{2}}
 \end{align*}
where in the last product almost all factors are $1$, and the
$L$-functions are the ones defined by Rankin-Selberg convolutions
(cf. \cite{JPSS83}).\\
(2) There is a nontrivial Bessel period of $\pi\otimes\sigma$ if
and only if $L(\frac{1}{2},\pi\times\sigma)\neq 0$.\\
\end{theo}

\section{Fourier-Jacobi periods of $\r{GL}_n\times\r{GL}_m$}
\label{sec3}

\subsection{Fourier-Jacobi models}
\label{sec3fm}

Let $k$ be a local field and $V$ a $k$-vector space of dimension
$n>0$. Suppose that $V$ has a decomposition $V=X\oplus W\oplus X^*$
where $W$, $X$ and $X^*$ have dimension $m$, $r$ and $r^*$
respectively. Then $n=m+r+r^*$. Let $P_{r,m,r^*}$ be the parabolic
subgroup of $\r{GL}(V)$ stabilizing the flag $0\subset X\subset
X\oplus W\subset V$ and $U_{r,m,r^*}$ its maximal unipotent
subgroup. Then $U_{r,m,r^*}$ fits into the following exact sequence:
 \begin{align*}
 \xymatrix@C=0.5cm{
   0 \ar[r] & \r{Hom}(X^*,X) \ar[r]& U_{r,m,r^*} \ar[r]& \r{Hom}(X^*,W)+\r{Hom}(W,X) \ar[r] & 0
   }.
 \end{align*}
We may write the above sequence as:
 \begin{align*}
 \xymatrix@C=0.5cm{
   0 \ar[r] & \(X^*\)^{\vee}\otimes X \ar[r]& U_{r,m,r^*} \ar[r]&
   \(X^*\)^{\vee}\otimes W+W^{\vee}\otimes X \ar[r] & 0}.
 \end{align*}
Let $\ell_X:X\R k$ (resp. $\ell_{X^*}:k\R X^*$) be any nontrivial
homomorphism (if exists) and let $U_{X}$ (resp. $U_{X^*}$) be a
maximal unipotent subgroup of $\r{GL}(X)$ (resp. $\r{GL}(X^*)$)
stabilizing $\ell_X$ (resp. $\ell_{X^*}$). Then the above exact
sequence fits into the following commutative diagram:
 \begin{align*}
 \xymatrix{
   0  \ar[r] & \(X^*\)^{\vee}\otimes X \ar[d]_{\ell_{X^*}^{\vee}\otimes\ell_X} \ar[r] & U_{r,m,r^*} \ar[d]
   \ar[r] & W^{\vee}\otimes X+\(X^*\)^{\vee}\otimes W \ar[d]_{\ell_X+\ell_{X^*}^{\vee}} \ar[r] & 0 \\
   0 \ar[r] & k \ar[r] & \r{H}(W^{\vee}+W) \ar[r] & W^{\vee}+W \ar[r] & 0   }
 \end{align*}
which is equivariant under the action of $U_X\times
U_{X^*}\times\r{GL}(W)$, where $\r{H}(W^{\vee}+W)=k+ W^{\vee}+W$ is
the Heisenberg group of $W^{\vee}+W$ whose multiplication is given
by
 \begin{align*}
 (t_1,w_1^{\vee},w_1)(t_2,w_2^{\vee},w_2)=
 \(t_1+t_2+\frac{w_1^{\vee}(w_2)-w_2^{\vee}(w_1)}{2},w_1^{\vee}+w_2^{\vee},w_1+w_2\).
 \end{align*}
Given a nontrivial character $\psi:k\R\d{C}^1$, there is a unique
infinite dimensional irreducible smooth representation
$\omega_{\psi}$ of $\r{H}(W^{\vee}+W)$ with central character
$\psi$. We choose the following model. Let $\c{S}(W^{\vee})$ be the
space of Bruhat-Schwartz functions on $W^{\vee}$. For
$\Phi\in\c{S}(W^{\vee})$, let
 \begin{align*}
 \(\omega_{\psi}(t,w^{\vee},w)\Phi\)(w^{\flat})=\psi\(t+w^{\flat}(w)+\frac{w^{\vee}(w)}{2}\)\Phi(w^{\flat}+w^{\vee})
 \end{align*}
for all $(t,w^{\vee},w)\in\r{H}(W^{\vee}+W)$. Moreover, if we choose
a character $\mu:k^{\times}\R\d{C}^{\times}$, we have a Weil
representation $\omega_{\mu}$ of $\r{GL}(W)$ on $\c{S}(W^{\vee})$ by
 \begin{align*}
 \(\omega_{\mu}(g)\Phi\)(w^{\flat})=\mu(\det g)|\det g|_k^{\frac{1}{2}}\Phi(w^{\flat}\cdot g)
 \end{align*}
where $g\in\r{GL}(W)$ acts on $W^{\vee}$ by $\(w^{\flat}\cdot
g\)w=w^{\flat}(g\cdot w)$ for all $w\in W$. They together form a
representation $\omega_{\psi,\mu}$ of $U_{r,m,r^*}\rtimes\r{GL}(W)$
through the projection $U_{r,m,r^*}\R\r{H}(W^{\vee}+W)$ and hence a
representation of $H:=U_{r,m,r^*}\rtimes(U_X\times
U_{X^*}\times\r{GL}(W))$ by extending trivially to $U_X\times
U_{X^*}$. As in the Bessel model, we choose a generic character
$\lambda:U_X\times U_{X^*}\R\d{C}^1$. Then we define the
representation
 \begin{align*}
 \nu_{\mu}=\omega_{\psi,\mu}\otimes\lambda\otimes\delta_W^{-\frac{1}{2}}
 \end{align*}
of $H$ which has the Gelfand-Kirillov dimension $m$. We also define
 \begin{align*}
 \overline{\nu_{\mu}}=\widetilde{\nu_{\mu}}\delta_W^{-1}
 =\omega_{\overline{\psi},\mu^{-1}}\otimes\overline{\lambda}\otimes\delta_W^{-\frac{1}{2}}.
 \end{align*}

As in the Bessel model, we have an injective morphism
$(\vep,\kappa):H\HR\r{GL}(V)\times\r{GL}(W)$. Then the pair
$(H,\nu_{\mu})$ is uniquely determined up to conjugacy in the group
$\r{GL}(V)\times\r{GL}(W)$ by the pair $W\subset V$, $(r,r^*)$ and
$\mu$. The following theorem generalizes the result in \cite{GGP}.

\begin{theo}\label{theo3fm}
Let $k$ be of characteristic zero and non-archimedean. Let $\pi$
(resp. $\sigma$) be an irreducible admissible representation of
$\r{GL}(V)$ (resp. $\r{GL}(W)$). Then
$\r{dim}_{\d{C}}\r{Hom}_H(\pi\otimes\sigma\otimes\widetilde{\nu_{\mu}},\d{C})\leq1$.
Moreover, if $\pi$ and $\sigma$ are generic, then
$\r{dim}_{\d{C}}\r{Hom}_H(\pi\otimes\sigma\otimes\widetilde{\nu_{\mu}},\d{C})=1$.\\
\end{theo}

The existence part is due to Corollary \ref{6cor} (2). We consider
the uniqueness part whose proof is similar to that in \cite[Section
16]{GGP} with mild modifications for general linear groups and a
general pair $(r,r^*)$.

Recall that we have $V=X\oplus W\oplus X^*$. Let $\tau$ (resp.
$\tau^*$) be a supercuspidal representation of $\r{GL}(X)$ (resp.
$\r{GL}(X^*)$) and let
 \begin{align*}
 \r{I}\(\tau,\sigma,\tau^*\):=\r{Ind}_{P_{r,m,r^*}}^{\r{GL}(V)}\(\tau\otimes\sigma\otimes\tau^*\)
 \end{align*}
be the induced representation. We have the following proposition
which is similar to \cite[Theorem 16.1]{GGP}.

\begin{prop}
With $\psi$ a fixed additive character of $k$ which is
non-archimedean, we have
 \begin{align*}
 \r{Hom}_{\r{GL}(V)}\(\r{I}\(\tau,\sigma\mu^{-1}\delta_W^{\frac{1}{2}},\tau^*\)\otimes\pi,\omega_{V,\psi}\)
 =\r{Hom}_H(\pi\otimes\sigma,\nu_{\mu})
 \end{align*}
as long as  $\widetilde{\pi}$ does not belong to the Bernstein
component of $\r{GL}(V)$ associated to the data $\(\r{GL}(X)\times
M^*,\tau\otimes\varsigma^*\)$ and
$\(M\times\r{GL}(X^*),\varsigma\otimes\tau^*\)$ where $M$ (resp.
$M^*$) is any Levi subgroup of $\r{GL}(X+W)$ (resp. $\r{GL}(W+X^*)$)
and $\varsigma$ (resp. $\varsigma^*$) is any irreducible
supercuspidal representation of $M$ (resp. $M^*$). Here,
$\omega_{V,\psi}$ is the Weil representation of $\r{GL}(V)$.
\end{prop}
\begin{proof}
Let $\r{Mp}(W+W^{\vee})$ be the $\d{C}^{\times}$-metaplectic cover
of the symplectic group of the symplectic space $W+W^{\vee}$. We
choose the Lagrangian $W^{\vee}$ and an additive character $\psi$,
and hence get the a model of Weil representation $\omega_{W,\psi}$.
We fix a homomorphism $\r{GL}(W)\HR\r{Mp}(W+W^{\vee})$ (by choosing
$\mu$ to be trivial) lifting the embedding
$\r{GL}(W)\HR\r{Sp}(W+W^{\vee})$ such that
 \begin{align*}
 \(\omega_{W,\psi}(g)\Phi\)(w^{\flat})=|\det
 g|_k^{\frac{1}{2}}\Phi(w^{\flat}g).
 \end{align*}
We consider another symplectic space $V_1+V_1^{\vee}$, where
$V_1=X+\(X^*\)^{\vee}+W$ and $V_1^{\vee}=W^{\vee}+X^*+X^{\vee}$. We
choose the Lagrangian $V_1^{\vee}\subset W^{\vee}$ and consider the
mixed model of the Weil representation $\omega_{V_1,\psi}$ of
$\r{Mp}(V_1+V_1^{\vee})$. This model has a realization on the space
$\c{S}\(\(X+\(X^*\)^{\vee}\)^{\vee}\)\otimes\c{S}(W^{\vee})=\c{S}(W^{\vee}+X^*+X^{\vee})$.
We choose a lifting of the embedding
$\r{GL}(V)\HR\r{Sp}(V_1+V_1^{\vee})$ such that the following diagram
commutes:
 \begin{align*}
 \xymatrix{
   \r{Mp}(W+W^{\vee}) \ar[dr] \ar@{^{(}->}[rr] && \r{Mp}(V_1+V_1^{\vee}) \ar[dr] \\
   & \r{Sp}(W+W^{\vee})  \ar@{^{(}->}[rr] && \r{Sp}(V_1+V_1^{\vee}) \\
   \r{GL}(W) \ar@{^{(}->}[uu]\ar@{^{(}->}[ur]\ar@{^{(}->}[rr] && \r{GL}(V) \ar@{^{(}->}[uu]\ar@{^{(}->}[ur]   }
 \end{align*}
The parabolic subgroup $P_{r,m,r^*}$ of $\r{GL}(V)$ is contained in
the parabolic subgroup $P(V_1,W)$ of $\r{Sp}(V_1+V_1^{\vee})$
stabilizing the flag $0\subset X+\(X^*\)^{\vee}\subset
X+\(X^*\)^{\vee}+W+W^{\vee}$. Elements in $P_{r,m,r^*}$ can be
written as $p=p(n,n^*,b;h,h^*;g)$ with $n\in\r{Hom}(W,X)$,
$n^*\in\r{Hom}(X^*,W)$, $b\in\r{Hom}(X^*,X)$, $h\in\r{GL}(X)$,
$h^*\in\r{GL}(X^*)$ and $g\in\r{GL}(W)$. We have the following
formula for the mixed model:
 \begin{align*}
 \(\omega_{V_1,\psi}(p(0,0,0;h,h^*;g))\Phi\)(w^{\flat},x^{\sharp},x^{\flat})&=
 |\det
 hgh^{*-1}|_k^{\frac{1}{2}}\Phi(w^{\flat}g,h^{*-1}x^{\sharp},x^{\flat}h);\\
 \(\omega_{V_1,\psi}(p(n,n^*,b;\b{1}_r,\b{1}_{r^*};\b{1}_m))\Phi\)(w^{\flat},x^{\sharp},x^{\flat})&=
 \psi\(x^{\flat}\(b\(x^{\sharp}\)\)+w^{\flat}(n^*(x^{\sharp}))\)
 \Phi\(w^{\flat}+n^{\vee}(x^{\flat}),x^{\sharp},x^{\flat}\)
 \end{align*}
for $(w^{\flat},x^{\sharp},x^{\flat})\in W^{\vee}+X^*+X^{\vee}$. We
denote by $\omega_{V,\psi}=\omega_{V_1,\psi}|_{\r{GL}(V)}$. Then
there is a $P_{r,m,r^*}$-equivariant map:
 \begin{align*}
 \r{ev}:\omega_{V,\psi}&\LR|\r{det}_X|_k^{\frac{1}{2}}\otimes\omega_{W+X^*,\psi}
 \bigoplus\omega_{X+W,\psi}\otimes|\r{det}_{X^*}|_k^{-\frac{1}{2}}\\
 \Phi(w^{\flat},x^{\sharp},x^{\flat})&\T\(\Phi(w^{\flat},x^{\sharp},0),\Phi(w^{\flat},0,x^{\flat})\).
 \end{align*}
The kernel of $\r{ev}$ is generated by the Schwartz functions
$\Phi_1\otimes\Phi_2\otimes\Phi_3\in\c{S}(W^{\vee})\otimes\c{S}(X^*)\otimes\c{S}(X^{\vee})$
such that $\Phi_2$ (resp. $\Phi_3$) is compactly supported on
$X^*-\{0\}$ (resp. $X^{\vee}-\1 0\2$). Since
$\widetilde{\omega_{?,\psi}}\cong\omega_{?,\overline{\psi}}$, we
have the following exact sequence of smooth
$P_{r,m,r^*}$-representations:
 \begin{align*}
 \xymatrix{
   0 \ar[r] &
   \r{cInd}_{N_{r,m,r^*}\rtimes\(P_X\times\r{GL}(W)\times
   P_{X^*}\)}^{\r{GL}(V)}|\r{det}_X|_k^{\frac{1}{2}}\otimes\widetilde{\omega_{W,\psi}}
   \otimes|\r{det}_{X^*}|_k^{-\frac{1}{2}} \ar[d]\\
     & \widetilde{\omega_{V,\psi}} \ar[d]\\
     &
     |\r{det}_X|_k^{\frac{1}{2}}\otimes\widetilde{\omega_{W+X^*,\psi}}
 \bigoplus\widetilde{\omega_{X+W,\psi}}\otimes|\r{det}_{X^*}|_k^{-\frac{1}{2}} \ar[r]&0}
 \end{align*}
where $P_X$ (resp. $P_{X^*}$) is the mirabolic subgroup of
$\r{GL}(X)$ (resp. $\r{GL}(X^*)$) stabilizing $\ell_X$ (resp.
$\ell_{X^*}^{\vee}$) and $\omega_{W,\psi}$ is a representation of
$N_{r,m,r^*}\rtimes\r{GL}(W)$ which coincides with
$\nu_{\mu}\mu^{-1}\delta_W^{\frac{1}{2}}$.

Tensoring the above sequence with
$\tau\otimes\sigma\mu^{-1}\delta_W^{\frac{1}{2}}\otimes\tau^{*}$ and
inducing to $\r{GL}(V)$, we get an exact sequence of
$\r{GL}(V)$-representations:
 \begin{align*}
 \xymatrix{
   0 \ar[r] & A \ar[r]
   & \r{I}\(\tau,\sigma\mu^{-1}\delta_W^{\frac{1}{2}},\tau^*\) \ar[r]
    & B\oplus B^*
   \ar[r] & 0 }
 \end{align*}
where
 \begin{align*}
 A&=\r{cInd}_{N_{r,m,r^*}\rtimes\(P_X\times\r{GL}(W)\times
  P_{X^*}\)}^{\r{GL}(V)}\(\tau|_{P_X}\)|\r{det}_X|_k^{\frac{1}{2}}\otimes\widetilde{\omega_{W,\psi}}
   \otimes\(\tau^*|_{P_{X^*}}\)|\r{det}_{X^*}|_k^{-\frac{1}{2}};\\
 B&=\r{Ind}_{P_{r,m+r^*}}^{\r{GL}(V)}\tau|\r{det}_X|^{\frac{1}{2}}\otimes
   \(\widetilde{\omega_{W+X^*,\psi}}\otimes\sigma\mu^{-1}\delta_W^{\frac{1}{2}}\otimes\tau^*\);\\
 B^*&=\r{Ind}_{P_{r+m,r^*}}^{\r{GL}(V)}
 \(\widetilde{\omega_{X+W,\psi}}\otimes\tau\otimes\sigma\mu^{-1}\delta_W^{\frac{1}{2}}\)
 \otimes\tau^*|\r{det}_{X^*}|^{-\frac{1}{2}}.
 \end{align*}
By our assumption on $\pi$, we have
 \begin{align*}
 \r{Hom}_{\r{GL}(V)}(B,\widetilde{\pi})=\r{Hom}_{\r{GL}(V)}(B^*,\widetilde{\pi})
 =\r{Ext}_{\r{GL}(V)}^1(B,\widetilde{\pi})=\r{Ext}_{\r{GL}(V)}^1(B^*,\widetilde{\pi})=0.
 \end{align*}
Hence
 \begin{align*}
 &\r{Hom}_{\r{GL}(V)}\(\r{I}\(\tau,\sigma\mu^{-1}\delta_W^{\frac{1}{2}},\tau^*\),\widetilde{\pi}\)\\
 =&\r{Hom}_{\r{GL}(V)}\(\r{cInd}_{N_{r,m,r^*}\rtimes\(P_X\times\r{GL}(W)\times
   P_{X^*}\)}^{\r{GL}(V)}\(\tau|_{P_X}\)|\r{det}_X|_k^{\frac{1}{2}}\otimes\widetilde{\omega_{W,\psi}}
   \otimes\(\tau^*|_{P_{X^*}}\)|\r{det}_{X^*}|_k^{-\frac{1}{2}},\widetilde{\pi}\).
 \end{align*}
Again since $\tau$ and $\tau^*$ are irreducible supercuspidal, we
have
 \begin{align*}
 \left.\tau|\r{det}_X|_k^{\frac{1}{2}}\right|_{P_X}\cong\r{cInd}_{U_X}^{\r{GL}(X)}\lambda;
 \qquad\left.\tau^*|\r{det}_{X^*}|_k^{-\frac{1}{2}}\right|_{P_{X^*}}
 \cong\r{cInd}_{U_{X^*}}^{\r{GL}(X^*)}\lambda^*
 \end{align*}
and by induction in stages, we have
 \begin{align*}
 \r{cInd}_{N_{r,m,r^*}\rtimes\(P_X\times\r{GL}(W)\times
   P_{X^*}\)}^{\r{GL}(V)}\(\tau|_{P_X}\)|\r{det}_X|_k^{\frac{1}{2}}\otimes\widetilde{\omega_{W,\psi}}
   \otimes\(\tau^*|_{P_{X^*}}\)|\r{det}_{X^*}|_k^{-\frac{1}{2}}
    =\r{cInd}_H^{\r{GL}(V)}\sigma\otimes\widetilde{\nu_{\mu}}.
 \end{align*}
Thus the proposition follows by the Frobenius reciprocity.\\
\end{proof}

 \begin{proof}[Proof of Theorem \ref{theo3fm} for the uniqueness part]
 We choose an irreducible supercuspidal representation $\tau$ (resp. $\tau^*$) of $\r{GL}(X)$ (resp. $\r{GL}(X^*)$)
 satisfying the assumption in the above proposition. After twisting unramified
 characters, we may assume that the induced representation
 $\r{I}\(\tau,\sigma\mu^{-1}\delta_W^{\frac{1}{2}},\tau^*\)$
 is irreducible. Then by the above proposition and \cite[Theorem 14.1(iv)]{GGP}, we have
  \begin{align*}
  \r{dim}_{\d{C}}\r{Hom}_H\(\pi\otimes\sigma\otimes\widetilde{\nu_{\mu}},\d{C}\)
  =\r{dim}_{\d{C}}\r{Hom}_{\r{GL}(V)}\(\r{I}\(\tau,\sigma\mu^{-1}\delta_W^{\frac{1}{2}},\tau^*\)
  \otimes\pi,\d{C}\)\leq 1.
  \end{align*}
 \end{proof}

\begin{defn}
A nontrivial element in the space
$\r{Hom}(\pi\otimes\sigma\otimes\widetilde{\nu_{\mu}},\d{C})$ is
called an \emph{$(r,r^*)$-Fourier-Jacobi model} of
$\pi\otimes\sigma$. When $r=r^*$, hence $n-m$ is even, it is just
the one defined in \cite{GGP}.\\
\end{defn}

\subsection{Fourier-Jacobi integrals, functional equations and $L$-functions}
\label{sec3fi}

Let $k$ be a number field, $(H,\nu_{\mu})$ be a pair associated with
nontrivial $\psi:k\B\d{A}\R\d{C}^1$,
$\mu:k^{\times}\B\d{A}^{\times}\R\d{C}^{\times}$, $\lambda$ a
generic character $(U_{X}\times U_{X^*})(k)\B(U_{X}\times
U_{X^*})(\d{A})\R\d{C}^1$ and
$\nu_{\mu}$ realizing on the space $\c{S}(W^{\vee}(\d{A}))$.\\

For any $\Phi\in\c{S}(W^{\vee}(\d{A}))$, we define the theta series
 \begin{align*}
 \theta_{\psi,\lambda,\mu}(h,\Phi)=\sum_{w^{\flat}\in
 W^{\vee}(k)}\lambda(h)\(\omega_{\psi,\mu}(h)\Phi\)(w^{\flat})
 \end{align*}
which is an automorphic form of $H$. Let $\pi$ (resp. $\sigma$) be
an irreducible cuspidal automorphic representation of
$\r{GL}(V)(\d{A})$ (resp. $\r{GL}(W)(\d{A})$).

\begin{defn}
If $n>m$, the following absolutely convergent integral is called an
\emph{$(r,r^*)$-Fourier-Jacobi period} of $\pi\otimes\sigma$ (for a
pair $(H,\nu_{\mu})$):
 \begin{align*}
 \FJ^{\nu_{\mu}}_{r,r^*}(\varphi_{\pi},\varphi_{\sigma};\Phi):=\intl_{H(k)\B
 H(\d{A})}\varphi_{\pi}(\vep(h))\varphi_{\sigma}(\kappa(h))
 \theta_{\overline{\psi},\overline{\lambda},\mu^{-1}}(h,\Phi)
 |\det h|_{\d{A}}^{\frac{r^*-r}{2}}\r{d}h,\quad\varphi_\pi\in\c{A}_{\pi},\varphi_{\sigma}\in\c{A}_{\sigma}.
 \end{align*}
If there exist $\varphi_{\pi}$, $\varphi_{\sigma}$ and
$\Phi\in\c{S}(W^{\vee}(\d{A}))$ such that
$\FJ^{\nu_{\mu}}_{r,r^*}(\varphi_{\pi},\varphi_{\sigma};\Phi)\neq0$,
then we say $\pi\otimes\sigma$ has a nontrivial
$(r,r^*)$-Fourier-Jacobi period. The case $m=n$ will be discussed in
Remark \ref{rem3nm}.
\end{defn}

It is obvious that $\FJ^{\nu_{\mu}}_{r,r^*}$ defines an element in
 \begin{align*}
 \r{Hom}_{H(\d{A})}(\pi\otimes\sigma\otimes\widetilde{\nu_{\mu}},\d{C})
 =\bigotimes_{v\in\d{M}_k}\r{Hom}_{H_v}(\pi_v\otimes\sigma_v\otimes\widetilde{\nu_{\mu_v}},\d{C}).
 \end{align*}
Since the later space has the multiplicity one property, we expect
that the Fourier-Jacobi period is also Eulerian. We now show that
this is true.\\

We choose basis $\1 v_1,...,v_r\2$ (resp. $\1
v^*_1,...,v^*_{r^*}\2$) for $X$ (resp. $X^*$) in the same way as in
Section \ref{sec2bi} (with $c=1$). We also choose a basis $\1
w_1,...,w_m\2$ of $W$ with dual basis $\1
w_1^{\vee},...,w_m^{\vee}\2$ of $W^{\vee}$. We identity $\r{GL}(V)$
with $\r{GL}_{n,k}$ and hence $\r{GL}(W)$ with $\r{GL}_{m,k}$ via
the basis:
 \begin{align}\label{3basis}
 \1 w_1,...,w_m,v_1,...,v_r,v^*_{r^*},...,v^*_1\2.
 \end{align}
Then the image of $H(\d{A})$ in $\r{GL}_n(\d{A})$ consists of
following matrices:
 \begin{align}\label{3huge}
 h=h(n,n^*,b;u,u^*;g)=\[\begin{array}{c|c|c}
                          \Large{\text{$g$}} &  & \begin{array}{ccc}
                                                    n^*_{1,r^*} & \cdots & n^*_{1,1} \\
                                                    \vdots &  & \vdots \\
                                                    n^*_{m,r^*} & \cdots &
                                                    n^*_{m,1}
                                                  \end{array}
                           \\ \hline
                          \begin{array}{ccc}
                            n_{1,1} & \cdots & n_{1,m} \\
                            \vdots &  & \vdots \\
                            n_{r,1} & \cdots & n_{r,m}
                          \end{array}
                           & \begin{array}{ccc}
                                & \Large{\text{$u$}} &
                             \end{array}
                            & \Large{\text{$b$}}
                           \\ \hline
                          \multicolumn{2}{c|}{} & \begin{array}{c}
                                                     \\
                                                    \Large{\text{$u^*$}} \\
                                                    \\
                                                  \end{array}
                        \end{array}
 \]
 \end{align}
where $n$, $n^*$, $b$, $u$, $u^*$ and $g$ are similar to those in
Section \ref{sec2bi}, but without entries related to $w_0$. Let
$U_{1^r,m,1^{r^*}}=U_{r,m,r^*}\rtimes(U_X\times U_{X^*})$ be a
unipotent subgroup of $\r{GL}_n$, then
 \begin{align*}
 U_{1^r,m,1^{r^*}}(\d{A})=\1
 \underline{u}=\underline{u}\(n,n^*,b;u,u^*\):=h\(n,n^*,l;u,u^*;\b{1}_m\)\2.
 \end{align*}
For $\Phi\in\c{S}(W^{\vee}(\d{A}))$, we have
 \begin{align*}
 \(\nu_{\mu}(\underline{u})\Phi\)(w^{\flat})=&\psi\(u_{1,2}+\cdots+u_{r-1,r}+b_{r,r^*}+
 u^*_{r^*,r^*-1}+\cdots+u^*_{2,1}+w^{\flat}\:^{\r{t}}[n^*_{1,r^*},...,n^*_{m,r^*}]\)\\
 &\Phi(w^{\flat}+[n_{r,1},...,n_{r,m}]).
 \end{align*}

For simplicity, we denote
$\underline{\psi}(\underline{u})=\psi(u_{1,2}+\cdots+u_{r-1,r}+b_{r,r^*}+
u^*_{r^*,r^*-1}+\cdots+u^*_{2,1})$. Then
 \begin{align*}
 \FJ^{\nu_{\mu}}_{r,r^*}(\varphi_{\pi},\varphi_{\sigma};\Phi)
 =\FJ^{\nu_{\mu}}_{r,r^*}(\frac{1}{2};\varphi_{\pi},\varphi_{\sigma};\Phi)
 \end{align*}
where
 \begin{align}\label{3fs1}
 \FJ^{\nu_{\mu}}_{r,r^*}(s;\varphi_{\pi},\varphi_{\sigma};\Phi)=&\intl_{\r{GL}_m(k)\B\r{GL}_m(\d{A})}
 \intl_{U_{1^r,m,1^{r^*}}(k)\B U_{1^r,m,1^{r^*}}(\d{A})}\notag\\
 &\varphi_{\pi}(\underline{u}g)\varphi_{\sigma}(g)
 \theta_{\overline{\psi},\overline{\lambda},\mu^{-1}}(\underline{u}g,\Phi)|\det
 g|_{\d{A}}^{s-\frac{1}{2}+\frac{r-r^*}{2}}\r{d}\underline{u}\r{d}g.
 \end{align}

There are two cases.\\

\textbf{Case 1: $r>0$.} In this case, we have
 \begin{align*}
 \theta_{\psi,\lambda,\mu}(\underline{u}g,\Phi)=\sum_{w^{\flat}\in
 W^{\vee}(k)}\lambda(\underline{u})\(\omega_{\psi,\mu}(\underline{u}g)\Phi\)(w^{\flat})=
 \sum_{n_r\in k^m}\lambda(\underline{u})\(\omega_{\psi,\mu}
 (\underline{n}_r\underline{u}g)\Phi\)(0)
 \end{align*}
where $\underline{n}_r=\underline{u}(n_r,0,0;\b{1}_r,\b{1}_{r^*})$
and
 \begin{align*}
 n_r=\[\begin{array}{ccc}
         0 & \cdots & 0 \\
         \vdots &  & \vdots \\
         0 & \cdots & 0 \\
         n_{r,1} & \cdots & n_{r,m}
       \end{array}
 \].
 \end{align*}
We define the subgroup $L_{r+1}$ of $U_{1^r,m,1^{r^*}}$ in the
similar way as before just by taking the matrix $n=0$, then we
separate $\underline{n}$ from $\underline{u}$ and summate the last
row in $\underline{n}$ over $k^m$. We get
 \begin{align}\label{3fs2}
  \eqref{3fs1}=&\intl_g\intl_{\r{M}_{r-1,m}(k\B\d{A})}\intl_{L_{r+1}(k)\B
  L_{r+1}(\d{A})}\intl_{\d{A}^m}\notag\\
  &\varphi_{\pi}(\underline{n}_r\underline{l}\:\underline{n}g)\varphi_{\sigma}(g)
  \overline{\lambda(\underline{l})}
  \(\omega_{\overline{\psi},\overline{\lambda},\mu^{-1}}(\underline{l}g)\Phi\)(\underline{n}_r)|\det
  g|_{\d{A}}^{s-\frac{1}{2}+\frac{r-r^*}{2}}\r{d}\underline{n}_r\r{d}\underline{l}\r{d}\underline{n}\r{d}g\notag\\
  =&\intl_g\intl_{\substack{\r{M}_{r-1,m}(k\B\d{A})\\
  \times\r{M}_{1,m}(\d{A})}}\intl_{L_{r+1}(k)\B
  L_{r+1}(\d{A})}\varphi_{\pi}(\underline{l}\:\underline{n}g)\varphi_{\sigma}(g)\Phi(\underline{n}_rg)
  \overline{\underline{\psi}(\underline{l})}\mu(\det g)^{-1}|\det
  g|_{\d{A}}^{s+\frac{r-r^*}{2}}\r{d}\underline{l}\r{d}\underline{n}\r{d}g
 \end{align}
where in the last line, $\underline{n}_r$ is the last row of
$\underline{n}$. If we repeat the process \eqref{2bs2},
\eqref{2bs3}, \eqref{2bs4} and \eqref{2bs5}, then
 \begin{align}\label{3fs3}
 \eqref{3fs2}&=\intl_{\r{M}_{r,m}(\d{A})}\intl_g\mu(\det g)^{-1}|\det
 g|_{\d{A}}^{s-\frac{n-m}{2}}\intl_{L_1(k)\B
 L_1(\d{A})}\(\rho(\underline{n})\varphi_{\pi}\)(\underline{l}g)\varphi_{\sigma}(g)\Phi(\underline{n}_r)
 \overline{\underline{\psi}(\underline{l})}\r{d}\underline{l}\r{d}g\r{d}\underline{n}.
 \end{align}
By the classical argument for the Rankin-Selberg convolution,
 \begin{align}\label{3fs4}
 \eqref{3fs3}=&\intl_{U_{1^m}(\d{A})\B\r{GL}_m(\d{A})}\intl_{\r{M}_{r-1,m}(\d{A})}\intl_{\r{M}_{1,m}(\d{A})}\notag\\
 &W^{\psi}_{\varphi_{\pi}}\(\[\begin{array}{cccc}
                            g & 0 & 0 & 0 \\
                            x & \b{1}_{r-1} & 0 & 0 \\
                            y & 0 & 1 & 0 \\
                            0 & 0 & 0 & \b{1}_{r^*}
                          \end{array}
 \]\)W^{\overline{\psi}}_{\varphi_{\sigma}}(g)\Phi(y) \mu(\det
 g)^{-1}|\det g|_{\d{A}}^{s-\frac{n-m}{2}}\r{d}y\r{d}x\r{d}g.
 \end{align}
If we denote \eqref{3fs4} by
$\Psi_r(s;W^{\psi}_{\varphi_{\pi}},W^{\overline{\psi}}_{\varphi_{\sigma}}\otimes\mu^{-1};\Phi)$,
then it is absolutely convergent for $\Re(s)\gg0$. Moreover, if
$W^{\psi}_{\varphi_{\pi}}=\otimes_vW_v$,
$W^{\overline{\psi}}_{\varphi_{\sigma}}=\otimes_vW^-_v$ and
$\Phi=\otimes_v\Phi_v$ are factorizable, we have
 \begin{align*}
 \FJ^{\nu_{\mu}}_{r,r^*}(s;\varphi_{\pi},\varphi_{\sigma};\Phi)=
 \Psi_r(s;W^{\psi}_{\varphi_{\pi}},W^{\overline{\psi}}_{\varphi_{\sigma}}\otimes\mu^{-1};\Phi)=
 \prod_{v\in\d{M}_k}\Psi_{v,r}(s;W_v,W^-_v\otimes\mu^{-1}_v;\Phi_v)
 \end{align*}
where
 \begin{align*}
 &\Psi_{v,r}(s;W_v,W^-_v\otimes\mu^{-1}_v;\Phi_v)=\intl_{U_{1^m,v}\B\r{GL}_{m,v}}
 \intl_{\r{M}_{r-1,m,v}}\intl_{\r{M}_{1,m,v}}\notag\\
 &W_v\(\[\begin{array}{cccc}
                            g_v & 0 & 0 & 0 \\
                            x_v & \b{1}_{r-1} & 0 & 0 \\
                            y_v & 0 & 1 & 0 \\
                            0 & 0 & 0 & \b{1}_{r^*}
                          \end{array}
 \]\)W^-_v(g_v)\Phi_v(y_v)\mu_v(\det
 g_v)^{-1}|\det g_v|_v^{s-\frac{n-m}{2}}\r{d}y_v\r{d}x_v\r{d}g_v.
 \end{align*}\\

\textbf{Case 2: $r=0$ but $r^*>0$.} Let $P_m$ be the standard
mirabolic subgroup of $\r{GL}_m$ consisting of (invertible) matrices
whose last row is $e_m=[0,...,0,1]\in k^m$. Then
 \begin{align*}
 &\theta_{\psi,\lambda,\mu}(g\underline{u},\Phi)=\sum_{w^{\flat}\in
 W^{\vee}(k)}\lambda(\underline{u})\(\omega_{\psi,\mu}(g\underline{u})\Phi\)(w^{\flat})\\
 =&\lambda(\underline{u})\mu(\det g)|\det g|_{\d{A}}^{\frac{1}{2}}\Phi(0)
 +\mu(\det g)|\det g|_{\d{A}}^{\frac{1}{2}}
 \sum_{\gamma\in P_m(k)\B\r{GL}_m(k)}\lambda(\underline{u})\(\omega_{\mu,\psi}(\underline{u})\Phi\)(e_m\gamma g)
 \end{align*}
where the first term (for
$\theta_{\overline{\psi},\overline{\lambda},\mu^{-1}}$) contributes
$0$ to
$\FJ^{\nu_{\mu}}_{0,n-m}(s;\varphi_{\pi},\varphi_{\sigma};\Phi)$
since $\varphi_{\pi}$ is a cusp form. Hence,
 \begin{align}\label{3fs5}
 &\FJ^{\nu_{\mu}}_{0,n-m}(s;\varphi_{\pi},\varphi_{\sigma};\Phi)=
 \intl_{U_{m,1^{n-m}}(k)\B
 U_{m,1^{n-m}}(\d{A})}\intl_{P_m(k)\B\r{GL}_m(\d{A})}\notag\\
 &\varphi_{\pi}(g\underline{u})\varphi_{\sigma}(g)
 \Phi(e_mg)\overline{\psi(e_mgn^*e_m^*)}\overline{\underline{\psi}(\underline{u})}
 \mu(\det g)^{-1}|\det g|_{\d{A}}^{s+\frac{n-m}{2}}\r{d}g\r{d}\underline{u}
 \end{align}
where $\underline{u}=\underline{u}(-,n^*,-;-,u^*)$ since $r=0$ and
$e^*_m=\:^{\r{t}}[1,0,...,0]$. Applying the Fourier inverse formula
to $\varphi_{\sigma}$, we have
 \begin{align}\label{3fs6}
 \eqref{3fs5}=\intl_{\underline{u}}\intl_{U_{1^m}(k)\B\r{GL}_m(\d{A})}\varphi_{\pi}(g\underline{u})
 W^{\overline{\psi}}_{\varphi_{\sigma}}(g)\Phi(e_mg)
 \overline{\psi(e_mgn^*e_m^*)}\:\overline{\underline{\psi}(\underline{u})}
 \mu(\det g)^{-1}|\det
 g|_{\d{A}}^{s+\frac{n-m}{2}}\r{d}g\r{d}\underline{u}.
 \end{align}
Factoring the inner integral through $U_{1^m}(k)\B U_{1^m}(\d{A})$
and incorporating this unipotent part into $\underline{u}$, what we
get is the integral over $U_{1^n}(k)\B
 U_{1^n}(\d{A})$ where $U_{1^n}$ is the standard maximal unipotent subgroup
of $\r{GL}_n$. Moreover, if we change the order of $g$ and $u\in
U_{1^n}(k)\B U_{1^n}(\d{A})$, all terms involving $\psi$ will form a
generic character $\underline{\psi}$ of $U_{1^n}(k)\B
U_{1^n}(\d{A})$:
 \begin{align*}
 \underline{\psi}(u)=\psi\(u_{1,2}+\cdots+
 u_{n-1,n}\).
 \end{align*}
In all,
 \begin{align}\label{3fs7}
 \eqref{3fs6}&=\intl_{U_{1^m}(\d{A})\B\r{GL}_m(\d{A})}\intl_{U_{1^n}(k)\B
 U_{1^n}(\d{A})}\varphi_{\pi}(ug)\overline{\underline{\psi}(u)}
 W^{\overline{\psi}}_{\varphi_{\sigma}}(g)\Phi(e_mg)\mu(\det g)^{-1}|\det
 g|_{\d{A}}^{s-\frac{n-m}{2}}\r{d}u\r{d}g \notag\\
 &=\intl_{U_{1^m}(\d{A})\B\r{GL}_m(\d{A})}W^{\psi}_{\varphi_{\pi}}\(\[\begin{array}{cc}
                                                                     g & 0 \\
                                                                     0 &
                                                                     \b{1}_{n-m}
                                                                   \end{array}
 \]\)
 W^{\overline{\psi}}_{\varphi_{\sigma}}(g)\Phi(e_mg)\mu(\det g)^{-1}|\det
 g|_{\d{A}}^{s-\frac{n-m}{2}}\r{d}g.
 \end{align}
If we denote \eqref{3fs7} by
$\Psi_0(s;W^{\psi}_{\varphi_{\pi}},W^{\overline{\psi}}_{\varphi_{\sigma}}\otimes\mu^{-1};\Phi)$,
then it is absolutely convergent for $\Re(s)\gg0$. Moreover, if
$W^{\psi}_{\varphi_{\pi}}=\otimes_vW_v$,
$W^{\overline{\psi}}_{\varphi_{\sigma}}=\otimes_vW^-_v$ and
$\Phi=\otimes_v\Phi_v$ are factorizable, we have
 \begin{align*}
 \FJ^{\nu_{\mu}}_{0,n-m}(s;\varphi_{\pi},\varphi_{\sigma};\Phi)
 =\Psi_0(s;W^{\psi}_{\varphi_{\pi}},W^{\overline{\psi}}_{\varphi_{\sigma}}\otimes\mu^{-1};\Phi)=
 \prod_{v\in\d{M}_k}\Psi_{v,0}(s;W_v,W^-_v\otimes\mu^{-1}_v;\Phi_v)
 \end{align*}
where
 \begin{align*}
 &\Psi_{v,0}(s;W_v,W^-_v\otimes\mu^{-1}_v;\Phi_v)\\
 =&\intl_{U_{1^m,v}\B\r{GL}_{m,v}}W_v\(\[\begin{array}{cc}
                                                                     g_v & 0 \\
                                                                     0 &
                                                                     \b{1}_{n-m}
                                                                   \end{array}
 \]\)W^-_v(g_v)\Phi_v(e_mg_v)\mu_v(\det g_v)^{-1}|\det
 g_v|_v^{s-\frac{n-m}{2}}\r{d}g_v.
 \end{align*}\\

Now we discuss the functional equations of the
\emph{$(r,r^*)$-Fourier-Jacobi integrals}
$\FJ^{\nu_{\mu}}_{r,r^*}(s;\varphi_{\pi},\varphi_{\sigma};\Phi)$.
First, there is a linear map
$\widehat{-}:\c{S}(W^{\vee}(\d{A}))\R\c{S}(W(\d{A}))$ given by
 \begin{align*}
 \widehat{\Phi}(w^{\sharp})=\intl_{W^{\vee}(\d{A})}\Phi(w^{\flat})\psi\(w^{\flat}(w^{\sharp})\)\r{d}w^{\flat}.
 \end{align*}
If we identify $W$ with $W^{\vee}$ through the basis $\1
w_1,...,w_m\2$, then $\widehat{-}$ is an endomorphism of
$\c{S}(W^{\vee}(\d{A}))$. Consider the group isomorphism
$\iota_{r^*,r}:H\R H$ given by
$\iota_{r^*,r}(\underline{u}g)=w_{n,m}\underline{u}^{\iota}w_{n,m}g^{\iota}$.
Then for any $h\in H(\d{A})$, we have the following commutative
diagram which can be checked directly:
 \begin{align}\label{3fe1}
 \xymatrix{
   \c{S}(W^{\vee}(\d{A})) \ar[d]_{\overline{\lambda}\cdot
   \omega_{\overline{\psi},\mu^{-1}}\(\iota_{r^*,r}(h)\)} \ar[r]^{\widehat{-}}
   & \c{S}(W^{\vee}(\d{A})) \ar[d]^{\lambda\cdot\omega_{\psi,\mu}(h)} \\
   \c{S}(W^{\vee}(\d{A})) \ar[r]^{\widehat{-}} & \c{S}(W^{\vee}(\d{A}))  }
 \end{align}
Then
 \begin{align}\label{3fe2}
 &\FJ^{\nu_{\mu}}_{r,r^*}(s;\varphi_{\pi},\varphi_{\sigma};\Phi) \notag\\
 =&\intl_g\intl_{\underline{u}}\varphi_{\pi}(\underline{u}g^{\iota})\varphi_{\sigma}(g^{\iota})
 \theta_{\overline{\psi},\overline{\lambda},\mu^{-1}}(\underline{u}g^{\iota},\Phi)|\det
 g|_{\d{A}}^{-s+\frac{1}{2}+\frac{r^*-r}{2}}\r{d}\underline{u}\r{d}g
 \notag\\
 =&\intl_g\intl_{\underline{u}}\(\rho(w_{n,m})\widetilde{\varphi_{\pi}}\)(w_{n,m}\underline{u}^{\iota}w_{n,m}g)
 \widetilde{\varphi_{\sigma}}(g)\theta_{\overline{\psi},\overline{\lambda},\mu^{-1}}(\underline{u}g^{\iota},\Phi)|\det
 g|_{\d{A}}^{-s+\frac{1}{2}+\frac{r^*-r}{2}}\r{d}\underline{u}\r{d}g
 \notag\\
 =&\intl_g\intl_{U_{1^{r^*},m,1^r}(k)\B U_{1^{r^*},m,1^r}(\d{A})}
 \(\rho(w_{n,m})\widetilde{\varphi_{\pi}}\)(\underline{u}g)
 \widetilde{\varphi_{\sigma}}(g)\theta_{\overline{\psi},\overline{\lambda},\mu^{-1}}
 \(\iota_{r^*,r}(\underline{u}g),\Phi\)|\det
 g|_{\d{A}}^{-s+\frac{1}{2}+\frac{r^*-r}{2}}\r{d}\underline{u}\r{d}g
 \end{align}

But by the Poisson summation formula and \eqref{3fe1},
 \begin{align*}
 \theta_{\overline{\psi},\overline{\lambda},\mu^{-1}}\(\iota_{r^*,r}(\underline{u}g),\Phi\)&=
 \sum_{w^{\flat}\in
 W^{\vee}(k)}\overline{\lambda(\underline{u})}
 \(\omega_{\overline{\psi},\mu^{-1}}\(\iota_{r^*,r}(\underline{u}g)\)\Phi\)(w^{\flat})\\
 &=\sum_{w^{\sharp}\in
 W^{\vee}(k)}{\lambda{\underline{u}}}
 \(\omega_{\overline{\psi},\mu^{-1}}\(\iota_{r^*,r}(\underline{u}g)\)\Phi\)^{\wedge}(w^{\sharp})\\
 &=\sum_{w^{\sharp}\in
 W^{\vee}(k)}\lambda{\underline{u}}
 \(\omega_{\psi,\mu}(\underline{u}g)\widehat{\Phi}\)(w^{\sharp})\\
 &=\theta_{\psi,\lambda,\mu}(\underline{u}g,\widehat{\Phi}).
 \end{align*}
Hence,
 \begin{align*}
 \eqref{3fe2}&=\intl_g\intl_{U_{1^{r^*},m,1^r}(k)\B U_{1^{r^*},m,1^r}(\d{A})}
 \(\rho(w_{n,m})\widetilde{\varphi_{\pi}}\)(\underline{u}g)
 \widetilde{\varphi_{\sigma}}(g)\theta_{\psi,\lambda,\mu}(\underline{u}g,\widehat{\Phi})|\det
 g|_{\d{A}}^{1-s-\frac{1}{2}+\frac{r^*-r}{2}}\r{d}\underline{u}\r{d}g\\
 &=\FJ^{\overline{\nu_{\mu}}}_{r^*,r}
 (1-s;\rho(w_{n,m})\widetilde{\varphi_{\pi}},\widetilde{\varphi_{\sigma}};\widehat{\Phi}).
 \end{align*}
In summary, we have the following

 \begin{theo}
 Let $n>m$, the Fourier-Jacobi integrals are holomorphic in $s$ and satisfy the
 following functional equation
  \begin{align*}
  \FJ^{\nu_{\mu}}_{r,r^*}(s;\varphi_{\pi},\varphi_{\sigma};\Phi)
  =\FJ^{\overline{\nu_{\mu}}}_{r^*,r}(1-s;\rho(w_{n,m})\widetilde{\varphi_{\pi}},\widetilde{\varphi_{\sigma}};
  \widehat{\Phi}).
  \end{align*}
 Let $\widetilde{W^{\psi}_{?}}(g)=W^{\psi}_{?}(w_ng^{\iota})\in\c{W}(\widetilde{\pi},\overline{\psi})$
 and similar for the one on $\r{GL}_m$.
 If $W^{\psi}_{\varphi_{\pi}}=\otimes_vW_v$, $W^{\overline{\psi}}_{\varphi_{\sigma}}=\otimes_vW^-_v$
 and $\Phi=\otimes_v\Phi_v$ are
 factorizable, then
 $\widetilde{W^{\psi}_{\varphi_{\pi}}}=\otimes_v\widetilde{W_v}$,
 $\widetilde{W^{\overline{\psi}}_{\varphi_{\sigma}}}=\otimes_v\widetilde{W^-_v}$
 and $\widehat{\Phi}=\otimes_v\widehat{\Phi}_v$
 are also factorizable with $\widetilde{W_v}(g)=W_v(w_ng^{\iota})$, $\widetilde{W^-_v}(g)=W^-_v(w_mg^{\iota})$
 and $\widehat{\Phi}_v=\widehat{\Phi_v}$. Then for $\Re(s)$ large,
  \begin{align*}
  \FJ^{\nu_{\mu}}_{r,r^*}(s;\varphi_{\pi},\varphi_{\sigma};\Phi)
  &=\Psi_r(s;W^{\psi}_{\varphi_{\pi}},W^{\overline{\psi}}_{\varphi_{\sigma}}\otimes\mu^{-1};\Phi)
  =\prod_{v\in\d{M}_k}\Psi_{v,r}(s;W_v,W^-_v\otimes\mu_v^{-1};\Phi_v);\\
  \FJ^{\overline{\nu_{\mu}}}_{r^*,r}
  (s;\rho(w_{n,m})\widetilde{\varphi_{\pi}},\widetilde{\varphi_{\sigma}};\widehat{\Phi})&=
  \Psi_{r^*}(s;\rho(w_{n,m})\widetilde{W^{\psi}_{\varphi_{\pi}}},
  \widetilde{W^{\overline{\psi}}_{\varphi_{\sigma}}}\otimes\mu;\widehat{\Phi})
  =\prod_{v\in\d{M}_k}\Psi_{v,r^*}(s;\rho(w_{n,m})\widetilde{W_v},\widetilde{W^-_v}\otimes\mu_v;\widehat{\Phi_v}).
  \end{align*}
 In particular, when $r=r^*$ (hence $n-m$ is even), the
 $(r,r)$-Fourier-Jacobi integral itself has a functional equation.\\
 \end{theo}

\begin{rem}[\textbf{The case $n=m$}]\label{rem3nm}
In this case, $V=W$ and $H=\r{GL}_n$. We will see that this is
exactly the case of Rankin-Selberg convolution for
$\r{GL}_n\times\r{GL}_n$. For simplicity, we assume that
$\pi\boxtimes\sigma\otimes\mu^{-1}$ is unitary. We fix a basis $\1
v_1,...,v_n\2$ for $V$ and identify $V^{\vee}$ as the set of row
vectors of length $n$. In this case, there is no choice of $\lambda$
anymore.

For any $\Phi\in\c{S}(V^{\vee}(\d{A}))$ and any character
$\chi:k^{\times}\B\d{A}^{\times}\R\d{C}^{\times}$ such that
$\mu\cdot\chi$ is unitary, we define
 \begin{align*}
 \theta^*_{\psi,\mu}(s;g,\Phi,\chi)=|\det
 g|_{\d{A}}^{s-\frac{1}{2}}\intl_{k^{\times}\B\d{A}^{\times}}\sum_{v^{\flat}\in V^{\vee}(k)-\1
 0\2}\(\omega_{\psi,\mu}(ag)\Phi\)(v^{\flat})|a|_{\d{A}}^{n\(
 s-\frac{1}{2}\)}\chi(a)\r{d}a
 \end{align*}
which is absolutely convergent when $\Re(s)>1$. It has a meromorphic
continuation to the entire complex plane which is holomorphic at
$s=\frac{1}{2}$ (cf. \cite{JS81}). For any holomorphic point $s$, it
is in $\c{A}\(\r{GL}_{n,k}\)$ with central character $\chi^{-1}$.
Moreover,
 \begin{align*}
 \theta^*_{\psi,\mu}(\frac{1}{2};g,-,\chi):\(\nu_{\mu},\c{S}(V^{\vee}(\d{A}))\)\LR\c{A}\(\r{GL}_{n,k}\)
 \end{align*}
is $\r{GL}_n(\d{A})$-equivariant. We denote $Z_n$ the center of
$\r{GL}_n$, $\chi_{\pi}$ (resp. $\chi_{\sigma}$) the central
character of $\pi$ (resp. $\sigma$) and let
 \begin{align}\label{3rs1}
 \FJ^{\nu_{\mu}}_{0,0}(s;\varphi_{\pi},\varphi_{\sigma};\Phi)=\intl_{Z_n(\d{A})\r{GL}_n(k)\B\r{GL}_n(\d{A})}
 \varphi_{\pi}(g)\varphi_{\sigma}(g)\theta^*_{\overline{\psi},\mu^{-1}}(s;g,\Phi,\chi_{\pi}\cdot\chi_{\sigma})\r{d}g
 \end{align}
be the usual Rankin-Selberg integral (cf. \cite{JS81,JPSS83}) which
is absolutely convergent at any holomorphic point $s$. The
\emph{Fourier-Jacobi period}
$\FJ^{\nu_{\mu}}_{0,0}(\varphi_{\pi},\varphi_{\sigma};\Phi):=
\FJ^{\nu_{\mu}}_{0,0}(\frac{1}{2};\varphi_{\pi},\varphi_{\sigma};\Phi)$
defines an element in
 \begin{align*}
 \r{Hom}_{\r{GL}_n(\d{A})}(\pi\otimes\sigma\otimes\widetilde{\nu_{\mu}},\d{C})
 =\bigotimes_{v\in\d{M}_k}\r{Hom}_{H_v}(\pi_v\otimes\sigma_v\otimes\widetilde{\nu_{\mu_v}},\d{C}).
 \end{align*}
Unfolding $\theta^*_{\overline{\psi},\mu^{-1}}$, we see that
 \begin{align}\label{3rs2}
 \eqref{3rs1}&=\intl_{P_n(k)\B\r{GL}_n(\d{A})}\varphi_{\pi}(g)\varphi_{\sigma}(g)\Phi(e_ng)\mu(\det
 g)^{-1}|\det g|_{\d{A}}^s\r{d}g \notag\\
 &=\intl_{U_{1^n}(\d{A})\B\r{GL}_n(\d{A})}W^{\psi}_{\varphi_{\pi}}(g)W^{\overline{\psi}}_{\varphi_{\sigma}}(g)\Phi(e_ng)
 \mu(\det g)^{-1}|\det g|_{\d{A}}^s\r{d}g.
 \end{align}
If we denote \eqref{3rs2} by
$\Psi_0(s;W^{\psi}_{\varphi_{\pi}},W^{\overline{\psi}}_{\varphi_{\sigma}}\otimes\mu^{-1};\Phi)$,
then it is absolutely convergent for $\Re(s)\gg0$. Moreover, if
$W^{\psi}_{\varphi_{\pi}}=\otimes_vW_v$,
$W^{\overline{\psi}}_{\varphi_{\sigma}}=\otimes_vW^-_v$ and
$\Phi=\otimes_v\Phi_v$ are factorizable, we have
 \begin{align*}
 \FJ^{\nu_{\mu}}_{0,0}(s;\varphi_{\pi},\varphi_{\sigma};\Phi)
 =\Psi_0(s;W^{\psi}_{\varphi_{\pi}},W^{\overline{\psi}}_{\varphi_{\sigma}}\otimes\mu^{-1};\Phi)=
 \prod_{v\in\d{M}_k}\Psi_{v,0}(s;W_v,W^-_v\otimes\mu^{-1}_v;\Phi_v)
 \end{align*}
where
 \begin{align*}
 \Psi_{v,0}(s;W_v,W^-_v\otimes\mu^{-1}_v;\Phi_v)=
 \intl_{U_{1^n,v}\B\r{GL}_{n,v}}W_v(g_v)W^-_v(g_v)\Phi_v(e_ng_v)\mu_v(\det g_v)^{-1}|\det g_v|^s_v\r{d}g_v.
 \end{align*}
Moreover, we have the following well-known functional equation
 \begin{align*}
 \FJ^{\nu_{\mu}}_{0,0}(s;\varphi_{\pi},\varphi_{\sigma};\Phi)=
 \FJ^{\overline{\nu_{\mu}}}_{0,0}(1-s;\widetilde{\varphi_{\pi}},\widetilde{\varphi_{\sigma}};\widehat{\Phi})
 \end{align*}
and $\FJ^{\nu_{\mu}}_{0,0}(s;\varphi_{\pi},\varphi_{\sigma};\Phi)$
will have possible simple poles at $s=-\r{i}\sigma$ and
$s=1-\r{i}\sigma$ with $\sigma$ real only if
$\pi\cong\widetilde{\sigma}\otimes\mu|\det|_{\d{A}}^{\r{i}\sigma}$.\\
\end{rem}

By Proposition \eqref{prop6jpss} and \eqref{prop6unr}, we have the
following theorem, which confirms \cite[Conjecture 24.1]{GGP} for
the Fourier-Jacobi periods of split unitary groups, i.e., general
linear groups.

\begin{theo}
(1) Let $\pi$ (resp. $\sigma$) be an irreducible cuspidal
automorphic representation of $\r{GL}(V)(\d{A})$ (resp.
$\r{GL}(W)(\d{A})$) \footnote[1]{When $n=m$, to prevent the
occurrence of a pole at $s=\frac{1}{2}$, we assume that the
character $\chi_{\pi}\otimes\chi_{\sigma}\otimes\mu^{-1}$ is unitary
for simplicity.}. For any $(r,r^*)$ such that $r+r^*=n-m$ and the
representation $\nu_{\mu}$ introduced above, we have, for
$\varphi_{\pi}\in\c{A}_{\pi}$, $\varphi_{\sigma}\in\c{A}_{\sigma}$
and $\Phi\in\c{S}(W^{\vee}(\d{A}))$ such that
$W^{\psi}_{\varphi_{\pi}}=\otimes_vW_v$,
$W^{\overline{\psi}}_{\varphi_{\sigma}}=\otimes_vW^-_v$ and
$\Phi=\otimes\Phi_v$ are factorizable,
 \begin{align*}
 \FJ^{\nu_{\mu}}_{r,r^*}(\varphi_{\pi},\varphi_{\sigma};\Phi)=L(\frac{1}{2},\pi\times\sigma\otimes\mu^{-1})
 \prod_{v\in\d{M}_k}\left.\frac{\Psi_{v,r}(s;W_v,W^-_v\otimes\mu_v^{-1};\Phi_v)}
 {L_v(s,\pi_v\times\sigma_v\otimes\mu_v^{-1})}\right|_{s=\frac{1}{2}}
 \end{align*}
where in the last product almost all factors are $1$, and the
$L$-functions are the ones defined by Rankin-Selberg convolutions
(cf. \cite{JPSS83}.\\
(2) There is a nontrivial Fourier-Jacobi period of
$\pi\otimes\sigma$ for $\nu_{\mu}$ if
and only if $L(\frac{1}{2},\pi\times\sigma\otimes\mu^{-1})\neq 0$.\\
\end{theo}

\section{A relative trace formula for $\r{U}_n\times\r{U}_m$: Bessel periods}
\label{sec4}

\subsection{Bessel models and periods}
\label{sec4bmp}

Let $k'$ be a field and $k/k'$ be a separable quadratic extension
which may be split. Let $\tau$ be the unique nontrivial involution
of $k$ fixing $k'$. We denote by $\r{Tr}$ and $\r{Nm}$ the trace and
norm of $k/k'$, respectively. Let $k^-=\{x\in k\;|\;x^{\tau}=-x\}$.
We fix a nonzero element $\jmath\in k^-$ once for all and define
$\widetilde{\r{Tr}}(x)=\jmath(x-x^{\tau})\in k'$.

If $k'$ is a local field or a number field, we denote $\eta$ the
character associated with $k/k'$ via the class field theory. Let
$\r{Her}_n(k/k')$ be the set of all $n\times n$ matrices satisfying
$\:^{\r{t}}g^{\tau}=g$ and
$\r{Her}_n(k/k')^{\times}=\r{Her}_n(k/k')\cap\r{GL}_n(k)$;
$\r{Her}^-_n(k/k')$ be the set of all $n\times n$ matrices
satisfying $\:^{\r{t}}g^{\tau}=-g$ and
$\r{Her}^-_n(k/k')^{\times}=\r{Her}^-_n(k/k')\cap\r{GL}_n(k)$. For a
vector space $X$ over $k$, we denote $X_{\tau}$ the vector space
over $k$ with the same underlying set of $X$ but the action of $k$
twisted by $\tau$. We also simply denote $X^{\vee}_{\tau}$ by
$\check{X}$. All hermitian or skew-hermitian spaces over $k$ are
defined with respect to $\tau$ and are assumed to be
non-degenerate.\\

Let us briefly recall the definition of Bessel models and periods
for unitary groups in \cite{GGP}. First, let us consider the local
situation, hence $k'$ is a local field. Let $V$ be a hermitian space
over $k$ of dimension $n$ with the hermitian form $(-,-)$ and
$W\subset V$ a subspace of dimension $m$ such that the restricted
hermitian form $(-,-)|_W$ is non-degenerate. We assume that the
orthogonal complement $W^{\perp}=E\oplus X\oplus X^*$ such that $E$
is a non-degenerate line and $X$, $X^*$ are isotropic, perpendicular
to $E$ and of dimension $r$. Hence $n=m+2r+1$. The hermitian form
restricted on $W$ (resp. $X\oplus X^*$) identifies $W$ (resp. $X^*$)
with $\check{W}$ (resp. $\check{X}$). We denote $\r{U}(V)$ (resp.
$\r{U}(W)$) the unitary group of $V$ (resp. $W$) which is a
reductive group over $k'$. Let $P'_{r,m+1}$ be the parabolic
subgroup of $\r{U}(V)$ stabilizing $X$ and $U'_{r,m+1}$ its maximal
unipotent subgroup. Then $U'_{r,m+1}$ fits into the following exact
sequence:
 \begin{align*}
 \xymatrix@C=0.5cm{
   0 \ar[r] & \bigwedge^2_{\tau}X \ar[r]& U'_{r,m+1} \ar[r]& \r{Hom}_k(W\oplus E,X) \ar[r] & 0 }
 \end{align*}
where $\bigwedge^2_{\tau}X\subset X_{\tau}\otimes
X=\r{Hom}_k(\check{X},X)$ consists of homomorphisms $b$ such that
$b^{\vee}_{\tau}=-b$. Here $b^{\vee}_{\tau}$ is just
$b^{\vee}:X^{\vee}\R X_{\tau}$ but viewed as an element in
$\r{Hom}_k(X^{\vee}_{\tau},X)$.

Let $\ell'_X:X\R k$ be any nontrivial homomorphism (if exists) and
let $U'_X$ be a maximal unipotent subgroup of $\r{GL}(X)$
stabilizing $\ell'_X$. Let $\ell'_W:k\R W\oplus E$ be a nontrivial
homomorphism such that its image is contained in $E$. Hence we have
a homomorphism
 \begin{align*}
 \ell':U'_{r,m+1}\LR \r{Hom}_k(W\oplus
 E,X)\overset{\(\ell'_W\)^{\vee}\otimes\ell'_X}{\LR}k
 \end{align*}
which is fixed by $U'_X\times\r{U}(W)$. Hence we can extend $\ell'$
trivially to it and define a homomorphism from
$H'=U'_{r,m+1}\rtimes\(U'_X\times\r{U}(W)\)$ to $k$. Let
$\psi':k'\R\d{C}^1$ be a nontrivial character and
$\lambda':U'_X\R\d{C}^{\times}$ be a generic character. We define
$\nu'=\(\psi'\circ\r{Tr}\circ\ell'\)\otimes\lambda'$ which is a
character of $H'$. We have an embedding
$(\vep,\kappa):H'\HR\r{U}(V)\times\r{U}(W)$. Then up to
$\r{U}(V)\times\r{U}(W)$-conjugacy, the pair $(H',\nu')$ is uniquely
determined by $W\subset V$.

Let $\pi$ (resp. $\sigma$) be an irreducible admissible
representation of $\r{U}(V)$ (resp. $\r{U}(W)$). A nontrivial
element in $\r{Hom}_{H'}(\pi\otimes\sigma,\nu')$ is called a
\emph{Bessel model} of $\pi\otimes\sigma$. In particular, when
$k/k'$ is split, then the Bessel model is just the $(r,r)$-Bessel
model for general linear groups introduced in Section \ref{sec2bm}.
We have the following multiplicity one result.

\begin{theo}
Let $k$ be of characteristic zero and $\pi$, $\sigma$ as above. Then
$\r{dim}_{\d{C}}\r{Hom}_{H'}(\pi\otimes\sigma,\nu')\leq1$.
\end{theo}
\begin{proof}
If $k$ is non-archimedean, this is due to
Aizenbud-Gourevitch-Rallis-Schiffmann \cite{AGRS10} for $m=n-1$ and
to \cite[Section 15]{GGP} for general $n,m$. If $k$ is archimedean,
this is due to Sun-Zhu \cite{SZ09} and Aizenbud-Gourevitch
\cite{AG09} for $m=n-1$ and to Jiang-Sun-Zhu
\cite{JSZ10} for general $n,m$.\\
\end{proof}

Now, we discuss the global case. Let $k/k'$ be a quadratic extension
of number fields. We have the notions, $\f{o}'$, $k'_{v'}$,
$\f{o}'_{v'}$ for $v'\in\d{M}_{k'}$, and $\d{A}'$, $\psi'$ similar
to those for $k$. Let $\lambda':U'_X(k)\B U'_X(\d{A})\R\d{C}^1$ be a
generic character. Then we have the pair $(H',\nu')$ in the global
situation.

Let $\pi$ (resp. $\sigma$) be an irreducible tempered representation
of $\r{U}(V)(\d{A}')$ (resp. $\r{U}(W)(\d{A}')$) which occurs with
multiplicity one in the space $\c{A}_0(\r{U}(V))$ (resp.
$\c{A}_0(\r{U}(W))$). We denote by $\c{A}_{\pi}$ (resp.
$\c{A}_{\sigma}$) the unique irreducible $\pi$ (resp.
$\sigma$)-isotypic subspace in $\c{A}_0(\r{U}(V))$ (resp.
$\c{A}_0(\r{U}(W))$).

\begin{defn}
The following absolutely convergent integral is called a
\emph{Bessel period} of $\pi\otimes\sigma$ (for a pair $(H',\nu')$):
 \begin{align*}
 \c{B}^{\nu'}_r(\varphi_{\pi},\varphi_{\sigma}):=\intl_{H'(k')\B
 H'(\d{A}')}\varphi_{\pi}(\vep(h'))\varphi_{\sigma}(\kappa(h'))\nu'(h')^{-1}\r{d}h',\qquad
 \varphi_\pi\in\c{A}_{\pi},\varphi_{\sigma}\in\c{A}_{\sigma}.
 \end{align*}
If there exist $\varphi_{\pi}$, $\varphi_{\sigma}$ such that
$\c{B}^{\nu'}_r(\varphi_{\pi},\varphi_{\sigma})\neq0$, then we say
$\pi\otimes\sigma$ has a nontrivial Bessel period.
\end{defn}

It is obvious that $\c{B}^{\nu'}_r$ defines an element in
 \begin{align*}
 \r{Hom}_{H'(\d{A})}(\pi\otimes\sigma,\nu')=\bigotimes_{v'\in\d{M}_{k'}}
 \r{Hom}_{H'_{v'}}(\pi_{v'}\otimes\sigma_{v'},\nu'_{v'}).
 \end{align*}\\

In order to decompose the distribution $\c{J}^{\psi'}$ in the next
section. We choose a basis $\1 v_1,...,v_r\2$ of $X$ such that
\begin{itemize}
  \item The homomorphism $\ell'_X:X\R k$ is given by the coefficient of $v_r$
  under the above basis;
  \item $U'_X$ is the maximal unipotent subgroup of the parabolic
  subgroup $P'_X$ stabilizing the complete flag
  $0\subset \LA v_1\RA\subset \LA v_1,v_2\RA\subset\cdots\subset\LA
   v_1,...,v_r\RA=X$;
  \item The generic character $\lambda'$ is given by
   \begin{align*}
   \lambda'(u')=\psi'\(\widetilde{\r{Tr}}\(u'_{1,2}+u'_{2,3}+\cdots+u'_{r-1,r}\)\)
   \end{align*}
   under the above basis.
\end{itemize}
We denote by $\1\check{v}_1,...,\check{v}_r\2$ the dual basis of
$\check{X}$. We also choose a basis $\1 w_1,...,w_m\2$ of $W$ and
$\1 w_0\2$ of $E$ such that the homomorphism $\ell'_W:k\R W\oplus E$
is given by $a\T aw_0$. Let
$\beta=\[(w_i,w_j)\]_{i,j=1}^m\in\r{Her}_m(k/k')^{\times}$,
$\beta_0=(w_0,w_0)\in k'^{\times}$ and
 \begin{align}\label{4beta}
 \beta'=\[\begin{array}{cc}
           \beta &  \\
            & \beta_0
         \end{array}
 \]\in\r{Her}_{m+1}(k/k')^{\times}.
 \end{align}
We identify $\r{U}(V)$ (resp. $\r{U}(W)$) with a unitary group of
$n$ (resp. $m$) variables $\r{U}_n$ (resp. $\r{U}_m$) under the
basis $\{v_1,...,v_r,w_1,...,w_m,w_0,\check{v}_r,...,\check{v}_1\}$
and view $\r{U}_m$ as a subgroup of $\r{U}_n$. Let
$U'_{1^r,m+1}=U'_{r,m+1}\rtimes U'_X$ be a unipotent subgroup of
$\r{U}_n$. Then the image of $H'(\d{A}')$ in $\r{U}_n(\d{A}')$
consists of the matrices
$h'=h'(n',b';u';g')=\underline{u}'(n',b';u')\cdot g'$ where
$g'\in\r{U}_m(\d{A}')$. Here,
 \begin{align*}
 \underline{u}'=\underline{u}'(n',b';u')=\[\begin{array}{ccc}
                                             \b{1}_r & n' &  w_r\(b'+\frac{n'n'_{\beta'}}{2}\) \\
                                              & \b{1}_{m+1} & n'_{\beta'}\\
                                              &  & \b{1}_r
                                           \end{array}
 \]
 \[\begin{array}{ccc}
     u' &  &  \\
      &  \b{1}_{m+1} &  \\
      &  & \check{u}'
   \end{array}
 \]\in U'_{1^r,m+1}(\d{A}')
 \end{align*}
where $n'\in\r{M}_{r,m+1}(\d{A})$, $b'\in\r{Her}_r^-(\d{A}/\d{A}')$
(similarly defined as for $k/k'$), $u'\in U'_X(\d{A})$;
$n'_{\beta'}=-\beta'^{-1}\:^{\r{t}}n'^{\tau}w_r$ and
$\check{u}'=w_r\:^{\r{t}}u'^{\tau,-1}w_r$. The character $\nu'$ on
$H'(\d{A}')$ is given by
 \begin{align}\label{4nupsi}
 \nu'(h')=\nu'(h'(n',b';u';g'))=\nu'(\underline{u}')
 =\underline{\psi'}(\underline{u}'):=\psi'\(\widetilde{\r{Tr}}\(u'_{1,2}+\cdots+u'_{r-1,r}+n'_{r,m+1}\)\)
 \end{align}
Then the Bessel period
 \begin{align*}
 \c{B}^{\nu'}_r(\varphi_{\pi},\varphi_{\sigma})=\intl_{\r{U}_m(k')\B\r{U}_m(\d{A}')}
 \intl_{U'_{1^r,m+1}(k')\B
 U'_{1^r,m+1}(\d{A}')}\varphi_{\pi}(\underline{u}'g')\varphi_{\sigma}(g')\overline{\underline{\psi'}(\underline{u}')}
 \r{d}\underline{u}'\r{d}g'.
 \end{align*}\\

\subsection{Decomposition of distributions}
\label{sec4dd}

We describe the relative trace formula on the unitary pair relating
the Bessel periods. Recall that $\pi$ (resp. $\sigma$) is an
irreducible tempered representation of $\r{U}_n(\d{A}')$ (resp.
$\r{U}_m(\d{A}')$) which occurs with multiplicity one in the space
$\c{A}_0(\r{U}_n)$ (resp. $\c{A}_0(\r{U}_m)$). For simplicity, we
further assume that the standard base change $\Pi$ (resp. $\Sigma$)
of $\pi$ (resp. $\sigma$) is an irreducible cuspidal automorphic
representation of $\r{GL}_n(\d{A})$ (resp. $\r{GL}_m(\d{A})$).

Let $f_n\in\c{H}(\r{U}_n(\d{A}'))$ (resp.
$f_m\in\c{H}(\r{U}_m(\d{A}'))$) be a smooth function on
$\r{U}_n(\d{A}')$ (resp. $\r{U}_m(\d{A}')$) with compact support. We
introduce a distribution
 \begin{align*}
 \c{J}_{\pi,\sigma}(f_n\otimes f_m):=\sum\c{B}^{\nu'}_r
 (\rho(f_n)\varphi_{\pi},\rho(f_m)\varphi_{\sigma})
 \overline{\c{B}^{\nu'}_r(\varphi_{\pi},\varphi_{\sigma})}
 \end{align*}
where the sum is taken over orthonormal bases of $\c{A}_{\pi}$ and
$\c{A}_{\sigma}$. Here we use a model of $\nu'$ as \eqref{4nupsi}.

As usual, we associate to $f_n\otimes f_m$ a kernel function on
$\(\r{U}_n(k')\B\r{U}_n(\d{A}')\times\r{U}_m(k')\B\r{U}_m(\d{A}')\)^2$:
 \begin{align}\label{4ker}
 \c{K}_{f_n\otimes
 f_m}(g'_1,g'_2;g'_3,g'_4)=\sum_{\zeta'\in\r{U}_n(k')}f_n(g'^{-1}_1\zeta'
 g'_3)\sum_{\xi'\in\r{U}_m(k')}f_m(g'^{-1}_2\xi' g'_4)
 \end{align}
and consider the following distribution:
 \begin{align}\label{4dis1}
 \c{J}(f_n\otimes f_m):=\iintl_{\(H'(k')\B
 H'(\d{A}')\)^2}\c{K}_{f_n\otimes
 f_m}\(\vep(h'_1),\kappa(h'_1);\vep(h'_2),\kappa(h'_2)\)
 \nu'(h'^{-1}_1h'_2)\r{d}h'_1\r{d}h'_2.
 \end{align}
The above integral is not absolutely convergent in general and need
regularization. In order to see what we should expect for these
distributions, We avoid this problem in this paper. Plug in
\eqref{4ker}, we have
 \begin{align}\label{4dis2}
 \eqref{4dis1}=&\iintl_{h'_1,h'_2}\sum_{\zeta'\in\r{U}_n(k')}f_n(\vep(h'_1)^{-1}\zeta'
 \vep(h'_2))\sum_{\xi'\in\r{U}_m(k')}f_m(\kappa(h'_1)^{-1}\xi'\kappa(h'_2))
 \nu'(h'^{-1}_1h'_2)\r{d}h'_1\r{d}h'_2 \notag\\
 =&\iintl_{h'_1,h'_2}\sum_{\xi'\in\r{U}_m(k')}\sum_{\zeta'\in\r{U}_n(k')}f_n(\vep(h'_1)^{-1}\xi'\zeta'
 \vep(h'_2))f_m(\kappa(h'_1)^{-1}\xi'\kappa(h'_2))
 \nu'(h'^{-1}_1h'_2)\r{d}h'_1\r{d}h'_2
 \notag\\
 =&\iintl_{h'_1,h'_2}\sum_{\xi'\in H'(k')}\sum_{\zeta'\in
 U'_{1^r,m+1}(k')\B\r{U}_n(k')}f_n(\vep(\xi'h'_1)^{-1}\zeta'
 \vep(h'_2))f_m(\kappa(\xi'h'_1)^{-1}\kappa(h'_2))\nu'(h'^{-1}_1h'_2)\r{d}h'_1\r{d}h'_2\notag\\
 =&\intl_{H'(k')\B H'(\d{A}')}\intl_{H'(\d{A}')}\sum_{\zeta'\in
 U'_{1^r,m+1}(k')\B\r{U}_n(k')}f_n(\vep(h'_1)^{-1}\zeta'
 \vep(h'_2))f_m(\kappa(h'_1)^{-1}\kappa(h'_2))
 \nu'(h'^{-1}_1h'_2)\r{d}h'_1\r{d}h'_2.
 \end{align}
If we write $h'_i=\underline{u}'_ig'_i$, then
 \begin{align}\label{4dis3}
 \eqref{4dis2}=&\intl_{H'(k')\B H'(\d{A}')}\intl_{U'_{1^r,m+1}(\d{A}')}\intl_{\r{U}_m(\d{A}')}
 \sum_{\zeta'\in U'_{1^r,m+1}(k')\B\r{U}_n(k')}\notag\\
 &f_n(g'^{-1}_1\underline{u}'^{-1}_1\zeta'\vep(h'_2))f_m(g'^{-1}_1\kappa(h'_2))
 \underline{\psi'}(\underline{u}'^{-1}_1)\nu'(h'_2)\r{d}g'_1\r{d}\underline{u}'_1\r{d}h'_2
 \notag\\
 =&\intl_{H'(k')\B H'(\d{A}')}\intl_{U'_{1^r,m+1}(\d{A}')}\intl_{\r{U}_m(\d{A}')}
 \sum_{\zeta'\in U'_{1^r,m+1}(k')\B\r{U}_n(k')}\notag\\
 &f_n(g'^{-1}_1\kappa(h'_2)^{-1}\underline{u}'^{-1}_1\zeta'\vep(h'_2))f_m(g'^{-1}_1)
 \underline{\psi'}(\underline{u}'^{-1}_1)\nu'(h'_2)\r{d}g'_1\r{d}\underline{u}'_1\r{d}h'_2.
 \end{align}

Define a function $f\in\c{H}(\r{U}_n(\d{A}'))$ by
 \begin{align*}
 f(g')=\intl_{\r{U}_m(\d{A}')}f_n(g'_1g')f_m(g'_1)\r{d}g'_1.
 \end{align*}
Then
 \begin{align}\label{4dis4}
 \eqref{4dis3}=&\sum_{\zeta'\in U'_{1^r,m+1}(k')\B\r{U}_n(k')}
 \intl_{H'(k')\B H'(\d{A}')}\intl_{U'_{1^r,m+1}(\d{A}')}f(\kappa(h'_2)^{-1}\underline{u}'^{-1}_1\zeta'\vep(h'_2))
 \underline{\psi'}(\underline{u}'^{-1}_1)\nu'(h'_2)\r{d}\underline{u}'_1\r{d}h'_2 \notag\\
 =&\sum_{\zeta'\in\(U'_{1^r,m+1}(k')\B\r{U}_n(k')\)/\!/H'(k')}
 \intl_{\r{Stab}_{\zeta'}^{H'}(k')
 \B H'(\d{A}')}\intl_{U'_{1^r,m+1}(\d{A}')}\notag\\
 &f(\kappa(h'_2)^{-1}\underline{u}'^{-1}_1\zeta'\vep(h'_2))
 \underline{\psi'}(\underline{u}'^{-1}_1)\nu'(h'_2)\r{d}\underline{u}'_1\r{d}h'_2
 \end{align}
where $H'$ acts on $U'_{1^r,m+1}\B\r{U}_n$ by conjugation and
$\(U'_{1^r,m+1}(k')\B\r{U}_n(k')\)/\!/H'(k')$ is the set of
conjugacy classes of $k'$-points. We introduce a $k'$-algebraic
group
 \begin{align*}
 \b{H}':=H'\underset{\r{U}_m}{\times}H'
 \end{align*}
which acts on $\r{U}_n$ in the following way: for any $k'$-algebra
$R$,
$\b{h}'=\b{h}'(\underline{u}'_1,\underline{u}'_2;g')\in\b{H}'(R)$
with $\underline{u}'_i\in U'_{1^r,m+1}(R)$, $g'\in\r{U}_m(R)$ and
$g\in\r{U}_n(R)$, we define the right action
$[g]\b{h}':=g'^{-1}\underline{u}'^{-1}_1g\underline{u}'_2g'$. We
also denote by $[\r{U}_n(k')]/\b{H}'(k')$ the set of $k'$-orbits
under this action. We define a character (also denoted by)
$\underline{\psi'}$ of $\b{H}'(\d{A}')$ by
 \begin{align*}
 \underline{\psi'}(\b{h}')=\underline{\psi'}\(\b{h}'(\underline{u}'_1,\underline{u}'_2;g')\)
 :=\underline{\psi'}(\underline{u}'^{-1}_1\underline{u}'_2).
 \end{align*}
Then
 \begin{align}\label{4dis5}
 \eqref{4dis4}&=\sum_{\zeta'\in[\r{U}_n(k')]/\b{H}'(k')}\intl_{\r{Stab}_{\zeta'}^{\b{H}'}(k')\B\b{H}'(\d{A}')}
 f([\zeta']\b{h}')\underline{\psi'}(\b{h}')\r{d}\b{h}'\notag\\
 &=:\sum_{\zeta'\in[\r{U}_n(k')]/\b{H}'(k')}\c{J}_{\zeta'}(f)
 =:\c{J}(f)
 \end{align}
which is a decomposition of the distribution $\c{J}$ according to
the orbits.

We denote by $\r{U}_n(k')_{\r{reg}}$ the set of all regular
$k'$-elements which will be defined in Section \ref{sec4co},
Definition \ref{def4reg}. In particular, the $\b{H}'$-stabilizer
$\r{Stab}_{\zeta'}^{\b{H}'}$ is trivial for
$\zeta'\in\r{U}_n(k')_{\r{reg}}$ by Proposition \ref{prop4orbit} and
the corresponding term $\c{O}(f,\zeta'):=\c{J}_{\zeta'}(f)$ is a
weighted orbital integral. If $f=\otimes_{v'}f_{v'}$ is
factorizable, then
 \begin{align*}
 \c{O}(f,\zeta')=\prod_{v'\in\d{M}_{k'}}\c{O}(f_{v'},\zeta')
 \end{align*}
where
 \begin{align*}
 \c{O}(f_{v'},\zeta')=\intl_{\b{H}'_{v'}}f_{v'}([\zeta']\b{h}'_{v'})
 \underline{\psi'}_{v'}(\b{h}'_{v'})\r{d}\b{h}'_{v'}.
 \end{align*}
In summary,
 \begin{align*}
 \c{J}(f)=&\c{J}_{\r{reg}}(f)+\c{J}_{\r{irr}}(f)\\
 :=&\sum_{\zeta'\in[\r{U}_n(k')_{\r{reg}}]/\b{H}'(k')}\c{J}_{\zeta'}(f)+
 \c{J}_{\r{irr}}(f)\\
 =&\sum_{\zeta'\in[\r{U}_n(k')_{\r{reg}}]/\b{H}'(k')}\prod_{v'\in\d{M}_{k'}}\c{O}(f_{v'},\zeta')
 +\c{J}_{\r{irr}}(f).
 \end{align*}\\

Now we discuss the relative trace formula on the product of two
general linear groups. We identify $\r{GL}_{m,k}\subset\r{GL}_{n,k}$
with $\r{GL}(W)\subset\r{GL}(V)$ and view
$\r{GL}_{n,k'}\subset\r{Res}_{k/k'}(\r{GL}_{n,k})$ (res.
$\r{GL}_{m,k'}\subset\r{Res}_{k/k'}(\r{GL}_{m,k})$) through the
basis $\1
v_1,...,v_r,w_1,...,w_m,w_0,\check{v}_r,...,\check{v}_1\2$. Let
$Z'_n$ (resp. $Z'_m$) be the center of $\r{GL}_{n,k'}$ (resp.
$\r{GL}_{m,k'}$). We denote $\psi=\psi'\circ\widetilde{\r{Tr}}$. Let
$\Pi$ and $\Sigma$ be as above which are cuspidal. Since $\Pi$ and
$\Sigma$ are the standard base change of representations of unitary
groups, we need to introduce a period integral to single out such
representations.

Until the end of this section, we assume that $n$ is odd, hence $m$
is even. Since the other case is similar and will lead to the same
fundamental lemma, we omit it in the following discussion. As
pointed out in \cite{Fl91}, \cite{Fl92} and \cite{GJR01}, the
central character of $\Pi$ (resp. $\Sigma$) should be trivial on
$\d{A}'$ and $\Pi$ (resp. $\Sigma\otimes\eta$) should be
distinguished by $\r{GL}_{n,k'}$ (resp. $\r{GL}_{m,k'}$). Hence we
consider the following integrals as in \cite{GJR01} and \cite{JR}:
 \begin{align*}
 \c{P}_n(\varphi_{\Pi})&=\intl_{Z'_n(\d{A'})\r{GL}_n(k')\B\r{GL}_n(\d{A}')}\varphi_{\Pi}(g_1)\r{d}g_1\\
 \c{P}_m(\varphi_{\Sigma})&=\intl_{Z'_m(\d{A'})\r{GL}_m(k')\B\r{GL}_m(\d{A}')}\varphi_{\Pi}(g_2)\eta(\det g_2)\r{d}g_2
 \end{align*}
where we recall that $\eta$ is the quadratic character associated to
$k/k'$.

Let $F_n\in\c{H}(\r{GL}_n(\d{A}))$ and
$F_m\in\c{H}(\r{GL}_m(\d{A}))$. We introduce another distribution
 \begin{align*}
 \c{J}_{\Pi,\Sigma}(s;F_n\otimes F_m):=\sum\c{B}^{\nu}_{r,r}(s;\rho(F_n)\varphi_{\Pi},\rho(F_m)\varphi_{\Sigma})
 \overline{\c{P}_n(\varphi_{\Pi})\c{P}_m(\varphi_{\Sigma})}
 \end{align*}
where the sum is taken over orthonormal bases of $\c{A}_{\Pi}$ and
$\c{A}_{\Sigma}$. Here, we take a model of $\nu$ in the following
way (be cautious that since we change the coordinates, the matrix
form of the following element changes from \eqref{2huge}):
 \begin{align*}
 \nu(h)=\nu\(\underline{u}g\)=\nu\(h\(n,n^*,b;u,u^*;g\)\)=\underline{\psi}(\underline{u})
 \end{align*}
where
 \begin{align}\label{4psi}
 \underline{\psi}(\underline{u})=\psi\(u_{1,2}+\cdots+u_{r-1,r}+n_{r,0}
 +\beta_0n^*_{0,r}+u^*_{r,r-1}+\cdots+u^*_{2,1}\)
 \end{align}

We associate to $F_n\otimes F_m$ a kernel function on
$\(\r{GL}_n(k)\B\r{GL}_n(\d{A})\times\r{GL}_m(k)\B\r{GL}_m(\d{A})\)^2$
(averaged by $Z'_n\times Z'_m$):
 \begin{align*}
 \c{K}_{F_n\otimes F_m}(g_1,g_2;g_3,g_4)
 =\intl_{Z'_n(k')\B Z'_n(\d{A}')}\sum_{\zeta\in\r{GL}_n(k)}F_n(g_1^{-1}z_1\zeta
 g_3)\r{d}z_1\intl_{Z'_m(k')\B Z'_m(\d{A}')}\sum_{\xi\in\r{GL}_m(k)}F_m(g_2^{-1}z_2\xi
 g_4)\r{d}z_2
 \end{align*}
and consider the following distribution:
 \begin{align}\label{4di1}
 &\c{J}(s;F_n\otimes F_m) \notag\\
 =&\intl_{Z'_m(\d{A}')\r{GL}_m(k')\B\r{GL}_m(\d{A}')}\intl_{Z'_n(\d{A}')\r{GL}_n(k')\B\r{GL}_n(\d{A}')}
 \intl_{H(k)\B H(\d{A})}\notag\\
 &\c{K}_{F_n\otimes F_m}(\vep(h),\kappa(h);g_1,g_2)\nu(h^{-1})
 |\det h|_{\d{A}}^{s-\frac{1}{2}}\eta(\det g_2)\r{d}h\r{d}g_1\r{d}g_2 \notag\\
 =&\intl_{Z'_m(\d{A}')\r{GL}_m(k')\B\r{GL}_m(\d{A}')}\intl_{Z'_n(\d{A}')\r{GL}_n(k')\B\r{GL}_n(\d{A}')}
 \intl_{Z'_m(k')\B Z'_m(\d{A}')}\intl_{Z'_n(k')\B Z'_n(\d{A}')}
 \intl_{H(k)\B H(\d{A})} \notag\\
 &\sum_{\zeta\in\r{GL}_n(k)}F_n(\vep(h)^{-1}z_1\zeta
 g_1)\sum_{\xi\in\r{GL}_m(k)}F_m(\kappa(h)^{-1}z_2\xi g_2)\nu(h^{-1})|\det
 h|_{\d{A}}^{s-\frac{1}{2}}\eta(\det g_2)\r{d}h\r{d}z_1\r{d}z_2\r{d}g_1\r{d}g_2 \notag\\
 =&\intl_{\r{GL}_m(k')\B\r{GL}_m(\d{A}')}\intl_{\r{GL}_n(k')\B\r{GL}_n(\d{A}')}
 \intl_{U_{1^r,m+1,1^r}(k)\B H(\d{A})}\notag\\
 &\sum_{\zeta\in\r{GL}_n(k)}F_n(\vep(h)^{-1}\zeta
 g_1)F_m(\kappa(h)^{-1}g_2)\nu(h^{-1})|\det
 h|_{\d{A}}^{s-\frac{1}{2}}
 \eta(\det g_2)\r{d}h\r{d}g_1\r{d}g_2.
 \end{align}
We decompose $h=\underline{u}g$ and notice that the group $H(\d{A})$
is unimodular. Moreover, we change a variable $g\T g_2^{-1}g$. Then
 \begin{align}\label{4di2}
 \eqref{4di1}=&\intl_{g_2}\intl_{g_1}\intl_{U_{1^r,m+1,1^r}(k)\B
 U_{1^r,m+1,1^r}(\d{A})}\intl_{\r{GL}_m(\d{A})}\sum_{\zeta\in
 \r{GL}_n(k)}\notag\\
 &F_n(g^{-1}g_2^{-1}\underline{u}^{-1}\zeta
 g_1)F_m(g^{-1})\underline{\psi}(\underline{u}^{-1})
 |\det g|_{\d{A}}^{s-\frac{1}{2}}|\det g_2|_{\d{A}}^{s-\frac{1}{2}}\eta(\det
 g_2)\r{d}g\r{d}\underline{u}\r{d}g_1\r{d}g_2.
 \end{align}

Define a function $\tilde{F}_s$ on $\r{GL}_n(\d{A})$, which is
holomorphic in $s$, by
 \begin{align*}
 \tilde{F}_s(\tilde{g})=\intl_{\r{GL}_m(\d{A})}\tilde{F}_n(g^{-1}\tilde{g})F_m(g^{-1})|\det
 g|_{\d{A}}^{s-\frac{1}{2}}\r{d}g.
 \end{align*}

We also introduce the symmetric space
$\r{S}_n\subset\r{Res}_{k/k'}(\r{GL}_{n,k})$ defined by the equation
$ss^{\tau}=\b{1}_n$, hence
 \begin{align*}
 \r{S}_n(k')=\1 s\in\r{GL}_n(k)\;|\;ss^{\tau}=\b{1}_n \2.
 \end{align*}
We have an isomorphism $\r{GL}_{n,k}/\r{GL}_{n,k'}\cong\r{S}_n$
given by $g\T gg^{\tau,-1}$. Define a linear map
$\sigma:\c{H}(\r{GL}_n(\d{A}))\R\c{H}(\r{S}_n(\d{A}'))$ by
 \begin{align}\label{4sigma}
 \(\sigma(F)\)(gg^{\tau,-1})=\intl_{\r{GL}_n(\d{A}')}F(g\tilde{g})\r{d}\tilde{g}
 \end{align}
and let $F_s=\sigma(\tilde{F}_s)$.

Combining these two operations together, we get
 \begin{align}\label{4di3}
 \eqref{4di2}=&\sum_{\zeta\in
 \r{S}_n(k')}\intl_{\r{GL}_{m}(k')\B\r{GL}_m(\d{A}')}\intl_{U_{1^r,m+1,1^r}(k)\B
 U_{1^r,m+1,1^r}(\d{A})}
 F_s(g_2^{-1}\underline{u}^{-1}\zeta
 \underline{u}^{\tau}g_2)\underline{\psi}(\underline{u}^{-1})
 |\det g_2|_{\d{A}}^{s-\frac{1}{2}}\eta(\det
 g_2)\r{d}\underline{u}\r{d}g_2
 \end{align}

Similar to the unitary case, we introduce the following
$k'$-algebraic group
 \begin{align*}
 \b{H}:=\r{Res}_{k/k'}(U_{1^r,m+1,1^r})\rtimes\r{GL}_{m,k'}
 \end{align*}
which acts on $\r{S}_n$ in the following way: for any $k'$-algebra
$R$, $\b{h}=\b{h}(\underline{u};g)$ with $\underline{u}\in
U_{1^r,m+1,1^r}(R\otimes k)$, $g\in\r{GL}_m(R)$ and
$s\in\r{S}_n(R)$, we define a right action
$[s]\b{h}:=g^{-1}\underline{u}^{-1}s\underline{u}^{\tau}g$. We also
denote by $[\r{S}_n(k')]/\b{H}(k')$ the set of $k'$-orbits under
this action. We define a character (also denoted by
$\underline{\psi}$) of $\b{H}(\d{A}')$ by
 \begin{align*}
 \underline{\psi}(\b{h})=\underline{\psi}(\b{h}(\underline{u}g))=\underline{\psi}(\underline{u}^{-1})
 \end{align*}
and $\det\b{h}:=\det g$. Then
 \begin{align}\label{4di4}
 \eqref{4di3}=&\sum_{\zeta\in
 [\r{S}_n(k')]/\b{H}(k')}\intl_{\b{H}(\d{A}')/\r{Stab}_{\zeta}^{\b{H}}(k')}
 F_s([\zeta]\b{h})
 \underline{\psi}(\b{h})
 |\det\b{h}|_{\d{A}}^{s-\frac{1}{2}}\eta(\det
 \b{h})\r{d}\b{h} \notag\\
 =&:\sum_{\zeta\in
 [\r{S}_n(k')]/\b{H}(k')}\c{J}_{\zeta}(s;F_s)=:\c{J}(s;F_s).
 \end{align}

We denote by $\r{S}_n(k')_{\r{reg}}$ the set of all regular
$k'$-elements which will be defined in Section \ref{sec4co},
Definition \ref{def4reg}. In particular, the $\b{H}$-stabilizer
$\r{Stab}_{\zeta}^{\b{H}}$ is trivial for
$\zeta\in\r{S}_n(k')_{\r{reg}}$ by Proposition \ref{prop4orbit} and
the corresponding term $\c{O}(s;F_s,\zeta):=\c{J}_{\zeta}(s;F_s)$ is
a weighted orbital integral. If $F_s=\otimes_{v'}F_{s,v'}$ is
factorizable, then
 \begin{align*}
 \c{O}(s;F_s,\zeta)=\prod_{v'\in\d{M}_{k'}}\c{O}(s;F_{s,v'},\zeta)
 \end{align*}
where
 \begin{align*}
 \c{O}(s;F_{s,v'},\zeta)=\intl_{\b{H}_{v'}}
 F_{s,v'}([\zeta]\b{h}_{v'})\underline{\psi}_{v'}(\b{h}_{v'})
 \eta_{v'}(\det\b{h}_{v'})|\det\b{h}_{v'}|_{v'}^{s-\frac{1}{2}}\r{d}\b{h}_{v'}.
 \end{align*}

In summary,
 \begin{align*}
 \c{J}(s;F_s)=&\c{J}_{\r{reg}}(s;F_s)+\c{J}_{\r{irr}}(s;F_s)\\
 :=&\sum_{\zeta\in[\r{S}_n(k')_{\r{reg}}]/\b{H}(k')}\c{J}_{\zeta}(s;F_s)+
 \c{J}_{\r{irr}}(s,F_s)\\
 =&\sum_{\zeta\in[\r{S}_n(k')_{\r{reg}}]/\b{H}(k')}\prod_{v'\in\d{M}_{k'}}\c{O}(s;F_{s,v'},\zeta)
 +\c{J}_{\r{irr}}(s;F_s).
 \end{align*}
When $s=\frac{1}{2}$, the terms involving $|\det|$ disappear and we
discard $s$ in all notations in this case. In particular,
 \begin{align*}
 \c{J}_{\zeta}(F)=\c{O}(F,\zeta)=\intl_{\b{H}(\d{A}')}
 F([\zeta]\b{h})\underline{\psi}(\b{h})\eta(\det\b{h})\r{d}\b{h}
 \end{align*}
which should be compared with
 \begin{align*}
 \c{J}_{\zeta'}(f)=\c{O}(f,\zeta')=\intl_{\b{H}'(\d{A}')}
 f([\zeta']\b{h}')\underline{\psi'}(\b{h}')\r{d}\b{h}'
 \end{align*}
assuming that $\zeta$ and $\zeta'$ are both regular.\\

\begin{rem}
If the original functions $F_n=\otimes_vF_{n,v}$ and $F_m=\otimes_v
F_{m,v}$ are factorizable, then
$\tilde{F}_s=\otimes_v\tilde{F}_{s,v}$ and
$F_s=\otimes_{v'}F_{s,v'}$ are also factorizable. If for some
(finite) place $v'$, $F_{m,v'}$ has the property that $\1 |\det
g|_{v'}\;|\;F_{m,v'}(g)\neq0\2$ is a singleton, then
$F_{s,v'}=F_{v'}$ are independent of $s$. In particular, this is the
case for almost all $v'$.\\
\end{rem}

We will introduce the notion of matching orbits in the next section.
It is expected that we have enough pairs $(F,f)$ such that if
$\zeta\leftrightarrow\zeta'$ are regular and match, then
 \begin{align*}
 \c{O}(F,\zeta)=\c{O}(f,\zeta').
 \end{align*}
The above relation should also be true place by place. In
particular, let us consider the case where $v'\in\d{M}_{k'}$ splits
into two places $v_{\bullet},v_{\circ}\in\d{M}_k$. Then we may
identify $\r{S}_{n,v'}$ with the set of pairs
$(g_{\bullet},g_{\circ})\in\r{GL}_{n,v_{\bullet}}\times\r{GL}_{n,v_{\circ}}$
with $g_{\bullet}g_{\circ}=\b{1}_n$, hence with $\r{GL}_{n,v'}$ by
$(g_{\bullet},g_{\circ})\T g_{\bullet}$. Then $F_{v'}$ becomes a
function on $\r{GL}_{n,v'}$ and
 \begin{align*}
 \c{O}(F_{v'},\zeta)=\intl_{\r{GL}_{m,v'}}\iintl_{\(U_{1^r,m+1,1^r,v'}\)^2}F_{v'}
 (g^{-1}\underline{u}_{\bullet}^{-1}\zeta\underline{u}_{\circ}g)
 \underline{\underline{\psi}}'(\underline{u}_{\bullet}^{-1}\underline{u}_{\circ})
 \r{d}\underline{u}_{\bullet}\r{d}\underline{u}_{\circ}\r{d}g
 \end{align*}
for the generic character
\begin{align}\label{4psi'}
 \underline{\underline{\psi}}'(\underline{u})=\psi'\(j\(u_{1,2}+\cdots+u_{r-1,r}+n_{r,0}
 +\beta_0n^*_{0,r}+u^*_{r,r-1}+\cdots+u^*_{2,1}\)\)
 \end{align}
where $\jmath=(j,-j)$.

On the other hand, we may identify $\r{U}_{n,v'}$ with the pairs
$(g_{\bullet},g_{\circ})$ such that
$g_{\circ}=w_{\beta'_{\circ}}^{-1}\:^{\r{t}}g_{\bullet}^{-1}w_{\beta'_{\circ}}$,
and with $\r{GL}_{n,v'}$. Here
 \begin{align*}
 w_{\beta'_{\circ}}=\[\begin{array}{ccc}
                 &  & w_r \\
                 & \beta'_{\circ} &  \\
                w_r &  &
              \end{array}
 \]
 \end{align*}
where $\beta'=(\beta'_{\bullet},\beta'_{\circ})$. Then $f_{v'}$
becomes a function on $\r{GL}_{n,v'}$ and
 \begin{align*}
 \c{O}(f_{v'},\zeta')=\intl_{\r{GL}_{m,v'}}\iintl_{\(U_{1^r,m+1,1^r,v'}\)^2}f_{v'}
 (g'^{-1}\underline{u}'^{-1}_{\bullet}\zeta'\underline{u}'_{\circ}g')
 \underline{\underline{\psi}}'(\underline{u}'^{-1}_{\bullet}\underline{u}'_{\circ})
 \r{d}\underline{u}'_{\bullet}\r{d}\underline{u}'_{\circ}\r{d}g'.
 \end{align*}

Moreover, in this case, that $\zeta$ and $\zeta'$ match exactly
means that $\zeta=\zeta'\in\r{GL}_{n,v'}$. Hence for
 \begin{align*}
 f_{v'}(g)=F_{v'}(g)=\intl_{\r{GL}_{n,v'}}\tilde{F}_{v_{\bullet}}(g\tilde{g})
 \tilde{F}_{v_{\circ}}(\tilde{g})\r{d}\tilde{g},
 \end{align*}
we have
 \begin{align}\label{4sms}
 \c{O}(F_{v'},\zeta)=\c{O}(f_{v'},\zeta)
 \end{align}
for all $\zeta$ regular.\\

\subsection{Matching of orbits and functions}
\label{sec4co}

Let $k'$ be a number field or a local field and $k/k'$ possibly
split. We say two elements
$\beta_1,\beta_2\in\r{Her}_m(k/k')^{\times}$ are \emph{similar},
denoted by $\beta_1\sim\beta_2$ if $\exists g\in\r{GL}_{m}(k)$ such
that $\beta_2=\:^{\r{t}}g^{\tau}\beta_1g$. We denote
$[\r{Her}_m(k/k')^{\times}]$ the set of similarity classes. We write
$W=W^{\beta}$, $V=V^{\beta}$ if the representing matrix of $W$ is in
the class $\beta$, and $\r{U}_m^{\beta}$ (resp. $\r{U}_n^{\beta}$,
$\b{H}^{\beta}$) for $\r{U}(W^{\beta})$ (resp. $\r{U}(V^{\beta})$,
$\b{H}'$). We define
 \begin{align*}
 \epsilon(\beta)=\eta\((-1)^{\frac{m(m-1)}{2}}\det\beta\)\in\{\pm1\}
 \end{align*}
to be the $\epsilon$-factor.\\

We first define the notion of pre-regular orbits. Recall that we
have the action of $\b{H}$ (resp. $\b{H}'$), and hence its maximal
unipotent subgroup $\r{Res}_{k/k'}U_{1^r,m+1,1^r}$ (resp.
$\(U'_{1^r,m+1}\)^2$), on $\r{S}_n$ (resp. $\r{U}^{\beta}_n$).

\begin{defn}
An element $\zeta\in\r{S}_n(k')$ (resp.
$\zeta^{\beta}\in\r{U}^{\beta}_n(k')$) is called \emph{pre-regular}
if its stabilizer under the action of
$\r{Res}_{k/k'}U_{1^r,m+1,1^r}$ (resp. $\(U'_{1^r,m+1}\)^2$) is
trivial.\\
\end{defn}

Let us start with the symmetric space $\r{S}_n$. Let $\b{B}$ be the
Borel subgroup of $\r{GL}_n$ consisting of upper-triangular matrices
and $\b{A}\cong(\r{GL}_1)^n$ be the maximal torus consisting of
diagonal matrices. Let $\b{W}_n$ be the Weyl group of $\r{GL}_n$
which is isomorphic to the group of $n$-permutations. We identity
elements in $\b{W}_n$ with permutation matrices, and hence identity
$\b{W}_n$ with a subgroup of $\r{GL}_n(k)$. Moreover, let
$\b{W}_n^{\r{S}}\subset\b{W}_n$ be the subgroup consisting of
elements whose square is $\b{1}_n$. Let $\b{P}$ be a standard
parabolic subgroup of $\r{GL}_{n,k}$ whose unipotent radical is
$\b{U}$. We also choose a standard Levi subgroup $\b{M}$ of it
consisting of matrices with diagonal blocks. The group
$\r{Res}_{k/k'}\b{P}$ acts on $\r{S}_n$ from right by
$[s]p=p^{-1}sp^{\tau}$. First, we have the following lemma.

 \begin{lem}\label{lem4weyl}
 An element $\zeta\in\r{S}_n(k')$ has trivial stabilizer under the action of $\b{U}(k)\subset\b{P}(k)$
 if and only if its orbit intersects $[\b{w}]\b{M}(k)$, where $\b{w}=w_n$ is the longest element
 in $\b{W}_n$. Moreover, the intersection contains at most $1$
 element.
 \end{lem}
 \begin{proof}
 By \cite[Proposition 3]{Fl92}, we have a Bruhat decomposition for
 $\r{S}_n(k')$:
  \begin{align*}
  \r{S}_n(k')=\coprod_{w\in\b{W}^{\r{S}}_n}[w]\b{B}(k).
  \end{align*}
 It implies that for general $\b{P}$, we have
  \begin{align*}
  \r{S}_n(k')=\bigcup_{w\in\b{W}^{\r{S}}_n}[w]\b{P}(k)=
  \bigcup_{w\in\b{W}^{\r{S}}_n}[w]\b{M}(k)\b{U}(k).
  \end{align*}
 Hence in any
 $\b{U}(k)$-orbit, there is a representative of the form $[w]m$. We
 assume that $\zeta=[w]m=m^{-1}wm^{\tau}$. Then its stabilizer is
 trivial if and only if
  \begin{align}\label{4weyl}
  \1 u^{-1}wu^{\tau}=w\;|\;u\in\b{U}(k)\2=\1\b{1}_n\2.
  \end{align}
 But $u^{-1}wu^{\tau}=w$ is equivalent to $wuw=u^{\tau}$, hence if
 $w=\b{w}$ is the longest Weyl element, it will force $u=\b{1}_n$,
 and the $[w]m$ is the only point where its orbit and $[w]\b{M}(k)$
 intersect.

 Conversely, we need to show that if \eqref{4weyl} holds, then
 $w\in[\b{w}]\b{W}_{\b{M}}$, where $\b{W}_{\b{M}}\subset\b{W}_n\cap\b{M}(k)$ is the Weyl group of $\b{M}$.
 We observe that any
 $w\in\b{W}^{\r{S}}_n$ is a disjoint union of transpositions. We use
 induction on $n$. The case $n=1$ is trivial and assume this for
 $<n$. If the transposition $(1,n)$ appears in $w$, then we reduce
 to the case of $n-2$ and we are done. Otherwise $(1,a)$ will appear
 in $w$ with $1\leq a<n$. Suppose that
 $\b{M}=\r{GL}_{n_1}\times\cdots\times\r{GL}_{n_t}$ (arranged from upper-left to
 lower-right) with $n=n_1+\cdots+n_t$, $n_i>0$ and $t>1$ (otherwise, it is
 trivial). If $n-a<n_t$, then $w'=(a,n)$ is an element in
 $\b{W}_{\b{M}}\subset\b{M}(k)$. The conjugation $w'^{-1}ww'\in\b{W}^{\r{S}}_n$ will contain the
 transportation $(1,n)$ and we are done. Otherwise, $n-a\geq n_t$
 and consider the transportation $(b,n)$ in $w'$ with $1<b\leq n$.
 If $b-1<n_1$, then we can conjugate $w$ by $(1,b)\in\b{W}_{\b{M}}$
 and we are again done. The rest case is that $b-1\geq n_1$. Then we
 define an element $u\in\b{U}(k)$ whose entries are $1$ at diagonals
 and positions $(1,b),(a,n)$, $0$ elsewhere. Then $wuw=u=u^{\tau}$
 which contradicts \eqref{4weyl}.\\
 \end{proof}

Applying the above lemma to $\b{P}=P_{1^r,m+1,1^r}$ stabilizing the
flag $0\subset\1 v_1\2\subset\cdots\subset X\subset X\oplus W\oplus
E\subset X\oplus W\oplus E\oplus\1\check{v}_r\2\subset\cdots\subset
V$, which is standard according to our basis
$\{v_1,...,v_r,w_1,...,w_m,w_0,\check{v}_r,...,\check{v}_1\}$ used
in this Chapter. Since $\b{w}$ normalizes $P_{1^r,m+1,1^r}$, the
$U_{1^r,m+1,1^r}(k)$-orbit of a pre-regular element $\zeta$ must
contain a unique element of the form
 \begin{align}\label{4normal}
 \[\begin{array}{ccccccc}
      &  &  &  &  &  & t_1(\zeta) \\
      &  &  &  &  & \iddots  &  \\
      &  &  &  & t_r(\zeta) &  &  \\
      &  &  & \r{Pr}(\zeta) &  &  &  \\
      &  & t_r(\zeta)^{\tau,-1} &  &  &  &  \\
      & \iddots &  &  &  &  &  \\
     t_1(\zeta)^{\tau,-1} &  &  &  &  &  &
   \end{array}
 \]
 \end{align}
with $t_i(\zeta)\in k^{\times}$ and
$\r{Pr}(\zeta)\in\r{S}_{m+1}(k')$. We call it the \emph{normal form}
of $\zeta$.\\

Now we consider the unitary group $\r{U}_n^{\beta}$. We fix a
minimal parabolic subgroup $\b{P}^{\beta}_0$ such that its maximal
unipotent subgroup $\b{U}^{\beta}_0$ contains $U'_{1^r,m+1}$. Let
$\b{A}_0^{\beta}$ be a maximal torus inside $\b{P}^{\beta}_0$ and
$\b{W}^{\beta}_n$ the Weyl group. Let $\b{P}'$ be a general standard
parabolic subgroup of $\r{U}^{\beta}_n$ whose maximal unipotent
subgroup is $\b{U}'$ and $\b{M}'\supset\b{A}_0^{\beta}$ a Levi
subgroup. The group $\(\b{P}'\)^2$ acts on $\r{U}^{\beta}_n$ from
right by $[g](p_1,p_2)=p_1^{-1}gp_2$. We have the following lemma
similar to the case of symmetric spaces.

 \begin{lem}\label{lem4weylu}
 An element $\zeta'\in\r{U}^{\beta}_n(k')$ has trivial stabilizer
 under the action of $\(\b{U}'(k')\)^2\subset\(\b{P}'(k')\)^2$ if and only if
 its orbit intersects
 $[\b{w}^{\beta}]\(\b{M}'(k')\)^2=\b{M}'(k')\b{w}^{\beta}\b{M}'(k')$,
 where $\b{w}^{\beta}$ is the longest element in $\b{W}^{\beta}_n$.
 Moreover, the intersection contains at most one element.
 \end{lem}
 \begin{proof}
 We have the usual Bruhat decomposition:
  \begin{align*}
  \r{U}_n^{\beta}(k')=\coprod_{w\in\b{W}_{\b{M}'}\B\b{W}^{\beta}_n/\b{W}_{\b{M}'}}\b{P}'(k')w\b{P}'(k')
  =\coprod_{w\in\b{W}_{\b{M}'}\B\b{W}^{\beta}_n/\b{W}_{\b{M}'}}\b{U}'(k')\b{M}'(k')w\b{M}'(k')\b{U}'(k').
  \end{align*}
 Hence in any $\(\b{U}'(k')\)^2$-orbit, there is a representative of
 the form $m_1wm_2$. We assume that $\zeta'=m_1wm_2$. Then its
 stabilizer is trivial if and only if
  \begin{align}\label{4uweyl}
  w\b{U}'(k')w^{-1}\cap\b{U}'(k')=\{\b{1}_n\}.
  \end{align}
 Let $R^+(\b{A}^{\beta}_0,\r{U}^{\beta}_n)$ (resp. $R^+(\b{A}^{\beta}_0,\b{M}')$) be the set of positive roots of
 $\b{A}^{\beta}_0$ (resp. in $\b{M}'$), then any double coset of
 $\b{W}_{\b{M}'}\B\b{W}^{\beta}_n/\b{W}_{\b{M}'}$ has a unique
 representative $w$ satisfying $w(\alpha)<0$ and $w^{-1}(\alpha)<0$
 for all $\alpha\in R^+(\b{A}^{\beta}_0,\b{M}')$. Assuming that $w$
 satisfies \eqref{4uweyl} and the above condition, then
 $w(\alpha)<0$ for any $\alpha\in
 R^+(\b{A}^{\beta}_0,\r{U}^{\beta}_n)$. Hence $w=\b{w}^{\beta}$.
 Conversely, if $w=\b{w}^{\beta}$, then \eqref{4uweyl} holds and the
 intersection is a singleton.\\
 \end{proof}

Applying the above lemma to $\b{P}'=P^{\beta}_{1^r,m+1}$, the
standard parabolic subgroup stabilizing the flag $0\subset\1
v_1\2\subset\cdots\subset X\subset V$. Since $\b{w}^{\beta}$
normalizes $P^{\beta}_{1^r,m+1}$, the $\(U'_{1^r,m+1}(k')\)^2$-orbit
of a pre-regular element $\zeta^{\beta}$ must contain a unique
element of the form
 \begin{align}\label{4unormal}
 \[\begin{array}{ccccccc}
      &  &  &  &  &  & t_1(\zeta^{\beta}) \\
      &  &  &  &  & \iddots  &  \\
      &  &  &  & t_r(\zeta^{\beta}) &  &  \\
      &  &  & \r{Pr}(\zeta^{\beta}) &  &  &  \\
      &  & t_r(\zeta^{\beta})^{\tau,-1} &  &  &  &  \\
      & \iddots &  &  &  &  &  \\
     t_1(\zeta^{\beta})^{\tau,-1} &  &  &  &  &  &
   \end{array}
 \]
 \end{align}
with $t_i(\zeta^{\beta})\in k^{\times}$ and
$\r{Pr}(\zeta^{\beta})\in\r{U}_{m+1}^{\beta}(k')$, where
$\r{U}_{m+1}^{\beta}=\r{U}(W^{\beta}\oplus E)$. We call it the
\emph{normal form} of $\zeta^{\beta}$.\\

\begin{rem}
There is a more intrinsic way to define the invariants $t_i$. For
$\zeta\in\r{M}_n(k)$, for $i=1,...,r$, let $\zeta[i]$ be the matrix
leaving the first $i$ columns and the last $i$ rows of $\zeta$ and
$s_i(\zeta)=\det\zeta[i]$ which is invariant under the action
$\zeta\T u\zeta u'$ for $u,u'\in U_{1^r,m+1,1^r}(k)$. Then
$\zeta\in\r{S}_n(k')$ (resp. $\zeta^{\beta}\in\r{U}^{\beta}_n(k')$)
is pre-regular if and only if $s_i(\zeta)\in k^{\times}$ (resp.
$s_i(\zeta^{\beta})\in k^{\times}$), where we view $\zeta$ (resp.
$\zeta^{\beta}$) as elements in $\r{M}_n(k)$ through the natural
inclusion $\r{S}_n\subset\r{Res}_{k/k'}\r{M}_{n,k}=\r{End}(V)$ and
$\r{U}^{\beta}_n\subset\r{Res}_{k/k'}\r{M}_{n,k}$. Moreover, the
invariants $t_i$ and $s_i$ are related by
$t_i(\zeta)=s_{i-1}(\zeta)^{\tau}s_i(\zeta)^{\tau,-1}$ ($s_0=1$) and
similar for $\zeta^{\beta}$.\\
\end{rem}

Now we are at the point to introduce regular orbits. We have natural
inclusions
$\r{S}_{m+1}\subset\r{Res}_{k/k'}\r{M}_{m+1,k}=\r{End}(W\oplus E)$
and $\r{U}^{\beta}_{m+1}\subset\r{Res}_{k/k'}\r{M}_{m+1,k}$.

 \begin{defn}\label{def4reg}
 An element $\xi\in\r{M}_{m+1}(k)$ is called \emph{regular} if it satisfies:
  \begin{itemize}
    \item $\xi$ is regular semisimple as an element of
    $\r{M}_{m+1}(k)$;
    \item the vectors $\1 w_0, \xi w_0,...,\xi^m w_0\2$ span $W\oplus
    E$;
    \item the vectors $\1 w^{\vee}_0,
    w^{\vee}_0\xi,...,w^{\vee}_0\xi^m\2$ span $W^{\vee}\oplus
    E^{\vee}$.
  \end{itemize}
 An element $\zeta\in\r{S}_n(k')$ (resp.
 $\zeta^{\beta}\in\r{U}_n^{\beta}(k')$) is called \emph{regular}, if
 it is pre-regular and the uniquely determined element
 $\r{Pr}(\zeta)\in\r{M}_{m+1}(k)$ (resp.
 $\r{Pr}(\zeta^{\beta})\in\r{U}^{\beta}_{m+1}(k)$) is regular.
 An $\b{H}$-orbit $\zeta\in[\r{S}_n(k')]/\b{H}(k')$ (resp. $\b{H}^{\beta}$-orbit
 $\zeta^{\beta}\in[\r{U}^{\beta}_n(k')]/\b{H}^{\beta}(k')$) is
 called \emph{regular} if any, hence all elements inside it are
 regular. We denote by $\r{M}_{m+1}(k)_{\r{reg}}$,
 $\r{GL}_{m+1}(k)_{\r{reg}}:=\r{M}_{m+1}(k)_{\r{reg}}\cap\r{GL}_{m+1}(k)$,
 $\r{S}_n(k')_{\r{reg}}$ (resp.
 $[\r{S}_n(k')_{\r{reg}}]/\b{H}(k')$) and
 $\r{U}_n^{\beta}(k')_{\r{reg}}$
 (resp. $[\r{U}_n^{\beta}(k')_{\r{reg}}]/\b{H}^{\beta}(k')$) the set
 of regular elements (resp. orbits).\\
 \end{defn}

To proceed, we need to recall some structural results in
\cite[Section 6]{RS07}, see also \cite{JR}, \cite{Yun09} and
\cite{Zh}. To include the whole action of $\b{H}$ (resp.
$\b{H}^{\beta}$), we need to consider the conjugation action (from
right) of $\r{GL}_{m,k'}$ (resp. $\r{U}^{\beta}_m$). We consider
more generally the conjugation action of $\r{GL}_{m,k}$. Recall that
by our choice of coordinates, the group $\r{GL}_{m,k}$ embeds into
$\r{GL}_{m+1,k}$ via
 \begin{align*}
 g\T\[\begin{array}{cc}
        g &  \\
         & 1
      \end{array}
 \].
 \end{align*}
Let $a_i(\xi)=\r{Tr}\(\bigwedge^i\xi\)$ for $1\leq i\leq m+1$ and
$b_i(\xi)=w^{\vee}_0\xi^i w_0$ for $1\leq i\leq m$. Let
$\b{T}_{\xi}=\det\[(w_0^{\vee}\xi^{i-1})(w_{j-1})\]_{i,j=1}^{m+1}$,
$\b{D}_{\xi}$ the matrix
$\[(w^{\vee}_0\xi^{i-1})(\xi^{j-1}w_0)\]_{i,j=1}^{m+1}$ and
$\Delta_{\xi}=\det\b{D}_{\xi}$. It is clear that if $\xi$ is
regular, then $\Delta_{\xi}\neq0$. Moreover, we have

 \begin{lem}
 Two regular elements $\xi$ and $\xi'$ are conjugate under
 $\r{GL}_m(k)$ if and only if $a_i(\xi)=a_i(\xi')$ and
 $b_i(\xi)=b_i(\xi')$. The $\r{GL}_m$-stabilizer of a regular
 element is trivial.
 \end{lem}
 \begin{proof}
 See \cite[Proposition 6.2 \& Theorem 6.1]{RS07} for the (equivalent version of the) first and
 second statements, respectively.\\
 \end{proof}

To include all unitary groups at the same time, we consider the set
$\f{U}_{m+1}$ of pairs $(\beta,\xi^{\beta})$ where
$\beta\in\r{Her}_m(k/k')^{\times}$ and
$\xi^{\beta}\in\r{U}^{\beta}_{m+1}(k')_{\r{reg}}:=\1
\xi^{\beta}\in\r{M}_{m+1}(k)_{\r{reg}}\;|\;\:^{\r{t}}\(\xi^\beta\)^{\tau}\beta'\xi^{\beta}=\beta'\2$,
where $\beta'$ is defined in \eqref{4beta}. The group $\r{GL}_m(k)$
acts on $\f{U}_{m+1}$ by
$(\beta,\xi^{\beta})g=(\:^{\r{t}}g^{\tau}\beta
g,g^{-1}\xi^{\beta}g)$. For
$\xi\in\r{S}_{m+1}(k')_{\r{reg}}:=\r{S}_{m+1}(k')\cap\r{M}_{m+1}(k)_{\r{reg}}$,
we denote by $\xi\Leftrightarrow(\beta,\xi^{\beta})$ if there exists
$g\in\r{GL}_m(k)$ such that $\xi=g^{-1}\xi^{\beta}g$. The following
lemma is considered in \cite{JR} and \cite{Zh}.

 \begin{lem}\label{lem4zhang}
 For $\xi\in\r{S}_{m+1}(k')_{\r{reg}}$, there exists a pair
 $(\beta,\xi^{\beta})$, unique up to the $\r{GL}_m(k)$-action, such that
 $\xi\Leftrightarrow(\beta,\xi^{\beta})$. Conversely, for every pair
 $(\beta,\xi^{\beta})\in\f{U}_{m+1}$, there exists an element
 $\xi\in\r{S}_{m+1}(k')_{\r{reg}}$, unique up to the
 $\r{GL}_m(k')$-conjugation, such that $\xi\Leftrightarrow(\beta,\xi^{\beta})$.
 \end{lem}
 \begin{proof}
 We first point out that two elements
 $\xi,\xi'\in\r{S}_{m+1}(k')_{\r{reg}}$ are conjugate under
 $\r{GL}_m(k)$ if and only if they are conjugate under
 $\r{GL}_m(k')$. Assume $g^{-1}\xi g=\xi'$, then $g^{-1}\xi g=g^{\tau,-1}\xi^{\tau,-1}g^{\tau}=g^{\tau,-1}\xi
 g^{\tau}$, hence $g=g^{\tau}$.

 The following proof is due to Zhang \cite{Zh}. It is easy too see that for
 $\xi\in\r{M}_{m+1}(k)_{\r{reg}}$,
 $\xi$ and $^{\r{t}}\xi$ have the same invariants $a_i,b_i$. Hence
 by the above lemma, there is a unique element $g\in\r{GL}_m(k)$
 such that
  $g^{-1}\xi g=\:^{\r{t}}\xi$
 If $\xi\in\r{S}_{m+1}(k')_{\r{reg}}$, then
 $^{\r{t}}\xi\in\r{S}_{m+1}(k')_{\r{reg}}$. Hence we have
 $g^{\tau}=g$. But we also have
  \begin{align*}
  ^{\r{t}}g\:^{\r{t}}\xi\:^{\r{t}}g^{-1}=\xi;\qquad g\:^{\r{t}}\xi
  g^{-1}=\xi
  \end{align*}
 which implies that $g=\:^{\r{t}}g$. Together, we have
 $g\in\r{Her}_m(k/k')^{\times}$. Moreover,
 $^{\r{t}}\xi^{\tau}(g^{-1})\xi=g^{-1}$. Hence
 $\xi\in\r{U}^{g^{-1}}_{m+1}(k')_{\r{reg}}$ and
 $\xi\Leftrightarrow(g^{-1},\xi)$.

 Conversely, given any $(\beta,\xi)\in\f{U}_{m+1}$, then
  \begin{align*}
  ^{\r{t}}\xi\beta^{\tau}\xi^{\tau}=\beta^{\tau}\Longrightarrow\beta^{\tau,-1}\:^{\r{t}}\xi\beta^{\tau}=\xi^{\tau,-1}.
  \end{align*}
 Moreover, there is a $\gamma\in\r{GL}_m(k)$ such that
 $\gamma^{-1}\xi\gamma=\:^{\r{t}}\xi$. We have
 $\beta^{\tau,-1}\gamma^{-1}\xi\gamma\beta^{\tau}=\xi^{\tau,-1}$, i.e,
 $\(\gamma\beta^{\tau}\)^{-1}\xi\(\gamma\beta^{\tau}\)=\xi^{\tau,-1}$. By regularity,
 $\gamma\beta^{\tau}\in\r{S}_m(k')$. Hence there exists $g\in\r{GL}_m(k)$
 such that $\gamma\beta^{\tau}=gg^{\tau,-1}$. Then
  \begin{align*}
  g^{\tau}g^{-1}\xi gg^{\tau,-1}\xi^{\tau}=\b{1}_{m+1}
  \Longrightarrow&g^{-1}\xi
  gg^{\tau,-1}\xi^{\tau}g^{\tau}=\b{1}_{m+1}
  \Longrightarrow\(g^{-1}\xi g\)\(g^{-1}\xi g\)^{\tau}=\b{1}_{m+1}
  \end{align*}
 i.e., $g^{-1}\xi g\in\r{S}_{m+1}(k')_{\r{reg}}$. The uniqueness
 part is obvious.\\
 \end{proof}

All these considerations lead to the following proposition:

 \begin{prop}\label{prop4orbit}
 (1). There is a natural bijection
  \begin{align*}
  [\r{S}_n(k')_{\r{reg}}]/\b{H}(k')\overset{\b{N}}{\longleftrightarrow}
  \coprod_{\beta\in[\r{Her}_m(k/k')^{\times}]}[\r{U}_n^{\beta}(k')_{\r{reg}}]/\b{H}^{\beta}(k').
  \end{align*}
 If $\b{N}\zeta^{\beta}=\zeta$, then we say that they
 \emph{matching} and denote by $\zeta\leftrightarrow\zeta^{\beta}$.

 (2). The set $\r{S}_n(k')_{\r{reg}}$ (resp.
 $\r{U}_n^{\beta}(k')_{\r{reg}}$) is non-empty and Zariski open in $\r{S}_n$
 (resp. $\r{U}_n^{\beta}$). Moreover, the $\b{H}$-stabilizer (resp.
 $\b{H}^{\beta}$-stabilizer) of regular $\zeta$ (resp. $\zeta^{\beta}$) is
 trivial.
 \end{prop}
 \begin{proof}
 (1). Start with an element $\zeta\in\r{S}_n(k')_{\r{reg}}$ and consider its
 normal form. We get $r$ invariants $t_1(\zeta),...,t_r(\zeta)$ and
 an element $\r{Pr}(\zeta)\in\r{S}_{m+1}(k')_{\r{reg}}$. By Lemma
 \eqref{lem4zhang}, there is a pair $(\beta,\xi^{\beta})$ such that
 $\xi=g^{-1}\xi^{\beta}g$. We fix $\beta$, then $\xi^{\beta}$ is uniquely determined up
 to the $\r{U}^{\beta}_m(k')$-conjugation. We define an element
 $\zeta^{\beta}$ by
  \begin{align*}
  \zeta^{\beta}=\[\begin{array}{ccccccc}
                     &  &  &  &  &  & t_1(\zeta) \\
                     &  &  &  &  & \iddots &  \\
                     &  &  &  & t_r(\zeta) & &  \\
                     &  &  & \xi^{\beta} &  &  &  \\
                     &  & t_r(\zeta)^{\tau,-1} &  &  &  &  \\
                     & \iddots &  &  &  &  &  \\
                    t_1(\zeta)^{\tau,-1} &  &  &  &  &  &
                  \end{array}
  \]\in\r{U}^{\beta}_n(k')_{\r{reg}}.
  \end{align*}
 By construction, $\zeta\T\zeta^{\beta}$ defines a map
 \begin{align*}
 \b{N}:[\r{S}_n(k')_{\r{reg}}]/\b{H}(k')\LR\coprod_{\beta\in[\r{Her}_m(k/k')^{\times}]}
 [\r{U}_n^{\beta}(k')_{\r{reg}}]/\b{H}^{\beta}(k')
 \end{align*}
 which is injective. The converse is similar.

 (2). The openness is due to the fact
 that the pre-regular elements in both cases correspond to the
 unique open cell in the Bruhat decomposition and the fact that the
 three conditions in Definition \ref{def4reg} are open.
 The non-emptiness is essentially exhibited in \cite[Section 3]{JR} combining with the
 exponential map. For the last
 part, we prove it for $\zeta\in\r{S}_n(k')_{\r{reg}}$ and the case
 for unitary groups is similar. We write $\zeta$ in its normal form. If
 $g^{-1}u^{-1}\zeta u^{\tau}g=\zeta$, then $u_1^{-1}g^{-1}\zeta
 gu_1^{\tau}=\zeta$ with $u_1=g^{-1}ug\in U_{1^r,m+1,1^r}(k)$. We
 have $g^{-1}\zeta g=\zeta$, $g^{-1}\r{Pr}(\zeta)g=\r{Pr}(\zeta)$ which
 implies that $g=\b{1}_n$ and hence $u=\b{1}_n$.\\
 \end{proof}

It is clear from the above discussion that the regular orbit
$\zeta\in[\r{S}_n(k')_{\r{reg}}]/\b{H}(k')$ or
$[\r{U}_n^{\beta}(k')_{\r{reg}}]/\b{H}^{\beta}(k')$ is determined by
its invariants $t_i(\zeta)$ ($i=1,...,r$),
$a_i(\zeta):=a_i(\r{Pr}(\zeta))$ ($i=1,...,m+1$) and
$b_i(\zeta):=b_i(\r{Pr}(\zeta))$ ($i=1,...,m$), and
$\zeta\leftrightarrow\zeta^{\beta}$ if and only if they have the
same invariants. We also denote
$\b{T}_{\zeta}=\b{T}_{\r{Pr}(\zeta)}$,
$\b{D}_{\zeta}=\b{D}_{\r{Pr}(\zeta)}$ and
$\Delta_{\zeta}=\Delta_{\r{Pr}(\xi)}$.\\

Now let $k'$ be a local field and $n$ odd. As suggested by the case
where $k/k'$ is split. We are going to formulate the conjecture on
the matching of functions. First, let us define a suitable
``transfer factor" $\b{t}$. Recall that we have a character
$\eta:k'^{\times}\R\d{C}^{\times}$. We define, for
$\zeta\in[\r{S}_n(k')_{\r{reg}}]/\b{H}(k')$,
 \begin{align*}
 \b{t}(\zeta)=\eta\(\b{T}_{\zeta}\cdot\(\det\r{Pr}(\zeta)\)^{-\frac{m}{2}}\)
 \end{align*}
which makes sense since $m$ is even and
$\b{T}_{\zeta}\cdot\(\det\r{Pr}(\zeta)\)^{-\frac{m}{2}}\in
k'^{\times}$.\\

 \begin{conj}[\textbf{Smooth matching}]\label{conj4sm}
 Let $k'$ be as above. Given any smooth, compactly supported function
 $F\in\c{H}(\r{S}_n(k'))$, there exist functions
 $\(f^{\beta}\)_{\beta}$ for each
 $\beta\in[\r{Her}_m(k/k')^{\times}]$ such that
  \begin{align*}
  \c{O}(F,\zeta)=\b{t}(\zeta)\c{O}(f^{\beta},\zeta^{\beta})
  \end{align*}
 for all $\zeta\in[\r{S}_n(k')_{\r{reg}}]/\b{H}(k')$ and
 $\zeta\leftrightarrow\zeta^{\beta}$. Conversely, given any
 functions $f^{\beta}\in\c{H}(\r{U}^{\beta}_n(k'))$ for each
 $\beta\in[\r{Her}_m(k/k')^{\times}]$, there exists a function
 $F\in\c{H}(\r{S}_n(k'))$ such that the above identity holds. We
 say that such $F$ and $(f^{\beta})_{\beta}$ match and denote by
 $F\leftrightarrow(f^{\beta})_{\beta}$.\\
 \end{conj}

 \begin{cor}[of Section \ref{sec4dd}]
 If $k/k'$ is split, then the conjecture of smoothing matching
 holds.\\
 \end{cor}

\subsection{The fundamental lemma}
\label{sec4fl}

As usual, to establish the equality between two relative trace
formulae, one need to prove the corresponding fundamental lemma. We
now formulate our fundamental lemma. We allow $m$ to be any
nonnegative integer.\\

Let $k'$ be a non-archimedean local field and $k/k'$ a separable
quadratic field extension. There are only two non-isomorphic
hermitian spaces of dimension $m>0$ over $k$ which is distinguished
by the factor $\epsilon(\beta)$. We will use the superscript $\pm$
instead of $\beta$ for $\epsilon(\beta)=\pm1$ in the following
notations.

We now assume that $k/k'$ is an unramified field extension and
$\psi':k'\R\d{C}^1$ is an unramified character. As before, we write
$\f{o}'$ (resp. $\f{o}$) the ring of integers of $k'$ (resp. $k$).
We denote by $\r{val}:k^{\times}\R\d{Z}$ the valuation map. We also
assume that $\beta_0\in\f{o}'$. Then $W^+$ will have a self-dual
$\f{o}$-lattice $L_W$ which extends to a self-dual $\f{o}$-lattice
$L_V$ of $V^+$. The unitary group $\r{U}^+_m$ (resp. $\r{U}^+_n$) is
unramified and has a model over $\f{o}'$. The group of
$\f{o}'$-points $\r{U}^+_m(\f{o}')$ (resp. $\r{U}^+_n(\f{o}')$) is a
hyperspecial maximal subgroup of $\r{U}^+_m(k')$ (resp.
$\r{U}^+_n(k')$). We also identity $\r{GL}_n(\f{o})$ with
$\r{GL}_n(L_V)$, a hyperspecial maximal subgroup of $\r{GL}_n(k)$
and let $\r{S}_n(\f{o}'):=\r{S}_n(k')\cap\r{GL}_n(\f{o})$.

We denote by $\c{H}(\r{U}^+_n(k')/\!/\r{U}^+_n(\f{o}'))$ (resp.
$\c{H}(\r{GL}_n(k)/\!/\r{GL}_n(\f{o}))$) the spherical Hecke algebra
of $\r{U}^+_n$ (resp. $\r{GL}_{n,k}$). There is a base change map
$b:\c{H}(\r{GL}_n(k)/\!/\r{GL}_n(\f{o}))\R\c{H}(\r{U}^+_n(k')/\!/\r{U}^+_n(\f{o}'))$
and recall that we have a linear map
$\sigma:\c{H}(\r{GL}_n(k)/\!/\r{GL}_n(\f{o}))\R\c{H}(\r{S}_n(k'))$
similarly defined as \eqref{4sigma} for the local case. Moreover, we
define
 \begin{align}\label{4tfactor}
 \b{t}(\zeta)=
 \begin{cases}
 (-1)^{\r{val}\(\b{T}_{\zeta}\cdot\prod_{i=1}^rt_i(\zeta)\)} &
 m\text{ is odd;}\\
 (-1)^{\r{val}\(\b{T}_{\zeta}\)} &
 m\text{ is even.}
 \end{cases}
 \end{align}

\begin{conj}[\textbf{The fundamental lemma}]
For any element $\tilde{F}\in\c{H}(\r{GL}_n(k)/\!/\r{GL}_n(\f{o}))$,
the functions $F=\sigma(\tilde{F})$ and $(f^+,f^-)$ match, where
$f^+=b(\tilde{F})$ and $f^-=0$.

In particular, we have
 \begin{align*}
 \c{O}(\CF_{\r{S}_n(\f{o}')},\zeta)=
  \begin{cases}
  \b{t}(\zeta)\c{O}(\CF_{\r{U}^+_n(\f{o}')},\zeta^+) & \zeta\leftrightarrow\zeta^+\in\r{U}^+_n(k');\\
  0 &\zeta\leftrightarrow\zeta^-\in\r{U}^-_n(k')
  \end{cases}
 \end{align*}
where
 \begin{align*}
 \c{O}(\CF_{\r{S}_n(\f{o}')},\zeta)=&\intl_{\b{H}(k')}\CF_{\r{S}_n(\f{o}')}
 ([\zeta]\b{h})\underline{\psi}(\b{h})\eta(\det\b{h})\r{d}\b{h};\\
 \c{O}(\CF_{\r{U}^+_n(\f{o}')},\zeta^+)=&
 \intl_{\b{H}^+(k')}\CF_{\r{U}^+_n(\f{o}')}([\zeta^+]\b{h}')\underline{\psi'}(\b{h}')\r{d}\b{h}'.
 \end{align*}
It is easy to see that $\zeta\leftrightarrow\zeta^+\in\r{U}^+_n(k')$
if and only if $\r{val}\(\Delta_{\zeta}\)$ is even.\\
\end{conj}

\begin{prop}\label{prop4half}
If $\r{val}\(\Delta_{\zeta}\)$ is odd, then
 \begin{align*}
 \c{O}(\CF_{\r{S}_n(\f{o}')},\zeta)=0.
 \end{align*}
\end{prop}
\begin{proof}
The following is a modification of an argument in \cite{Zh}. Let
 \begin{align*}
 w=\[\begin{array}{ccc}
        &  & w_r \\
        & \b{1}_{m+1} &  \\
       w_r &  &
     \end{array}
 \],
 \end{align*}
then it is easy to see that
$\CF_{\r{S}_n(\f{o}')}(s)=\CF_{\r{S}_n(\f{o}')}(w\:^{\r{t}}sw)$. If
we write $\zeta$ in its normal form \eqref{4normal}, then
 \begin{align}\label{4half}
 &\c{O}(\CF_{\r{S}_n(\f{o}')},\zeta)\notag\\
 =&\intl_{U_{1^r,m+1,1^r}(k)}\intl_{\r{GL}_m(k')}\CF_{\r{S}_n(\f{o}')}
 \(w\:^{\r{t}}\underline{u}^{\tau}\:^{\r{t}}g\:^{\r{t}}\zeta\:^{\r{t}}g^{-1}\:^{\r{t}}\underline{u}^{-1}w\)
 \underline{\psi}(\underline{u}^{-1})\eta(\det g)\r{d}g\r{d}\underline{u}\notag\\
 =&\intl_{U_{1^r,m+1,1^r}(k)}\intl_{\r{GL}_m(k')}\CF_{\r{S}_n(\f{o}')}
 \(\(w\:^{\r{t}}\underline{u}^{\tau,-1}w\)^{-1}\:^{\r{t}}g
 \(w\:^{\r{t}}\zeta w\)\:^{\r{t}}g^{-1}\(w\:^{\r{t}}\underline{u}^{\tau,-1}w\)^{\tau}\)
 \underline{\psi}(\underline{u}^{-1})\eta(\det
 g)\r{d}g\r{d}\underline{u}.
 \end{align}
But $w\:^{\r{t}}\zeta w$ and $\zeta$ are both of normal form and
have the same invariants. In the proof of Lemma \ref{lem4zhang}, we
see that there exists $h\in\r{GL}_m(k')$ such that $w\:^{\r{t}}\zeta
w=h^{-1}\zeta h$ and $\eta(\det h)=-1$. Moreover,
$\underline{\psi}(\underline{u})=\underline{\psi}(w\:^{\r{t}}\underline{u}^{\tau,-1}w)$.
After changing variables
$w\:^{\r{t}}\underline{u}^{\tau,-1}w\T\underline{u}$,
$h\:^{\r{t}}g^{-1}\T g$, we have
 \begin{align*}
 \eqref{4half}=-\intl_{U_{1^r,m+1,1^r}(k)}\intl_{\r{GL}_m(k')}\CF_{\r{S}_n(\f{o}')}
 \(\underline{u}^{-1}g^{-1}\zeta g\underline{u}^{\tau}\)
 \underline{\psi}(\underline{u}^{-1})\eta(\det g)\r{d}g\r{d}\underline{u}
 \end{align*}
which implies that $\c{O}(\CF_{\r{S}_n(\f{o}')},\zeta)=0$.\\
\end{proof}

\section{A relative trace formula for $\r{U}_n\times\r{U}_m$: Fourier-Jacobi periods}
\label{sec5}

\subsection{Fourier-Jacobi models and periods}
\label{sec5fmp}

Let us briefly recall the definition of Fourier-Jacobi models and
periods for unitary groups in \cite{GGP}. First, let us consider the
local situation, hence $k'$ is a local field. Let $V$ be a hermitian
space over $k$ of dimension $n$ with the hermitian form $(-,-)$ and
$W\subset V$ a subspace of dimension $m$ such that the restricted
hermitian form $(-,-)|_W$ is non-degenerate. We assume that the
orthogonal complement $W^{\perp}=X\oplus X^*$ such that $X$, $X^*$
are isotropic and of dimension $r$. Hence $n=m+2r$. The hermitian
form restricted on $W$ (resp. $X\oplus X^*$) identifies $W$ (resp.
$X^*$) with $\check{W}=W^{\vee}_{\tau}$ (resp. $\check{X}$). We
denote $\r{U}(V)$ (resp. $\r{U}(W)$) the unitary group of $V$ (resp.
$W$) which is a reductive group over $k'$. Let $P'_{r,m}$ be the
parabolic subgroup of $\r{U}(V)$ stabilizing $X$ and $U'_{r,m}$ its
maximal unipotent subgroup. Then $U'_{r,m}$ fits into the following
exact sequence:
 \begin{align*}
 \xymatrix@C=0.5cm{
   0 \ar[r] & \bigwedge^2_{\tau}X \ar[r]& U'_{r,m} \ar[r]& \r{Hom}_k(W,X) \ar[r] & 0
   }.
 \end{align*}

Let $\ell'_X:X\R k$ be any nontrivial homomorphism (if exists) and
let $U'_X$ be a maximal unipotent subgroup of $\r{GL}(X)$
stabilizing $\ell'_X$. The homomorphism $\ell'_X$ induces a
homomorphism
$\bigwedge^2\ell'_X:\bigwedge^2_{\tau}X\R\bigwedge^2_{\tau}k=k^-\overset{\jmath\cdot}{\LR}k'$
and a homomorphism
 \begin{align*}
 \r{Res}_{k/k'}\ell'_X:\r{Hom}_k(W,X)\LR\r{Hom}_k(W,k)=\r{Hom}_{k'}\(\r{Res}_{k/k'}W,k'\)=\(\r{Res}_{k/k'}W\)^{\vee}
 \end{align*}
The hermitian pair $\(-,-\)$ of $W$ induces a symplectic pair
$\widetilde{\r{Tr}}\(-,-\)$ on the $2m$-dimensional $k'$-vector
space $\r{Res}_{k/k'}W$. Hence $\(\r{Res}_{k/k'}W\)^{\vee}$ is
identified with $\r{Res}_{k/k'}W$ through this pair. Again, let
$\r{H}\(\r{Res}_{k/k'}W\)$ be the Heisenberg group, then we have the
following commutative diagram:
 \begin{align}\label{5diag}
 \xymatrix{
   0  \ar[r] & \bigwedge^2_{\tau}X \ar[d]_{\bigwedge^2\ell'_X} \ar[r] & U'_{r,m} \ar[d]
   \ar[r] & \r{Hom}_k(W,X) \ar[d]_{\r{Res}_{k/k'}\ell'_X} \ar[r] & 0 \\
   0 \ar[r] & k' \ar[r] & \r{H}\(\r{Res}_{k/k'}W\) \ar[r] & \r{Res}_{k/k'}W \ar[r]& 0   }
 \end{align}
For a nontrivial character $\psi':k'\R\d{C}^1$, we have a Weil
representation $\omega'_{\psi'}$ of
$\r{H}\(\r{Res}_{k/k'}W\)\rtimes\r{Mp}\(\r{Res}_{k/k'}W\)$. If we
choose a character $\mu:k^{\times}\R\d{C}^{\times}$ such that
$\mu|_{k'^{\times}}=\eta$, we will have a splitting map
 \begin{align}\label{5mu}
 \xymatrix{
                &         \r{Mp}\(\r{Res}_{k/k'}W\) \ar[d]     \\
  \r{U}(W) \ar@{^{(}->}[ur]^{\iota_{\mu}} \ar@{^{(}->}[r]^{\iota} & \r{Sp}\(\r{Res}_{k/k'}W\)            }
 \end{align}
(cf. \cite[Section 1,2]{HKS96}). By restriction, we get a Weil
representation $\omega'_{\psi',\mu}$ of
$\r{H}\(\r{Res}_{k/k'}W\)\rtimes\r{U}(W)$, and hence a
representation of $U'_{r,m}\rtimes\r{U}(W)$ through the middle
vertical map in \eqref{5diag}. Let $\lambda':U'_X\R\d{C}^{\times}$
be a generic character. Then we define
$\nu'_{\psi',\mu}=\omega'_{\psi',\mu}\otimes\lambda'$ which is a
smooth representation of
$H':=U'_{r,m}\rtimes\(U'_X\times\r{U}(W)\)$. As before, we have an
embedding $H'\HR\r{U}(V)\times\r{U}(W)$. Then up to conjugation by
the normalizer of $H'$ in $\r{U}(V)\times\r{U}(W)$,
$\nu'_{\psi',\mu}$ is determined by $\psi'$ modulo
$\r{Nm}k^{\times}$ and $\mu$.

Let $\pi$ (resp. $\sigma$) be an irreducible admissible
representation of $\r{U}(V)$ (resp. $\r{U}(W)$). A nontrivial
element in
$\r{Hom}_{H'}\(\pi\otimes\sigma\otimes\widetilde{\nu'_{\psi',\mu}},\d{C}^{\times}\)$
is called a \emph{Fourier-Jacobi model} of $\pi\otimes\sigma$. In
particular, when $k/k'$ is split, then the Fourier-Jacobi model is
just the $(r,r)$-Fourier-Jacobi model for general linear groups
introduced in Section \ref{sec3fm}. We have the following
multiplicity one result.

\begin{theo}
Let $k$ be a non-archimedean local field of characteristic zero and
$\pi$, $\sigma$ as above. Then
$\r{dim}_{\d{C}}\r{Hom}_{H'}\(\pi\otimes\sigma\otimes\widetilde{\nu'_{\psi',\mu}},\d{C}^{\times}\)\leq1$.
\end{theo}
\begin{proof}
See \cite[Section 14]{GGP}.\\
\end{proof}

Now, we discuss the global case. Let $k/k'$ be a quadratic extension
of number fields, $\psi':k'\B\d{A}'\R\d{C}^1$ nontrivial,
$\mu:k^{\times}\B\d{A}^{\times}\R\d{C}^{\times}$ such that
$\mu|_{\d{A}'^{\times}}=\eta$, and $\lambda':U'_X(k)\B
U'_X(\d{A})\R\d{C}^1$ a generic character. Then we have the pair
$(H',\nu'_{\psi',\mu})$ in the global situation. To define a global
period, we need to fix a model for the Weil representation. Let
$\b{L}\subset\(\r{Res}_{k/k'}W\)^{\vee}$ be a Lagrangian subspace.
Let $\c{S}(\b{L}(\d{A}'))$ be the space of Bruhat-Schwartz functions
on $\b{L}(\d{A}')$.

For $\phi\in\c{S}(\b{L}(\d{A}'))$, we define the theta series to be
 \begin{align*}
 \theta_{\psi',\lambda',\mu}(h',\phi)=\sum_{w\in\b{L}(k')}\lambda'(h')
 \(\omega'_{\psi',\mu}(h')\phi\)(w)
 \end{align*}
which is an automorphic form on $H'$.

Let $\pi$ (resp. $\sigma$) be an irreducible tempered representation
of $\r{U}(V)(\d{A}')$ (resp. $\r{U}(W)(\d{A}')$) which occurs with
multiplicity one in the space $\c{A}_0(\r{U}(V))$ (resp.
$\c{A}_0(\r{U}(W))$). We denote by $\c{A}_{\pi}$ (resp.
$\c{A}_{\sigma}$) the unique irreducible $\pi$ (resp.
$\sigma$)-isotypic subspace in $\c{A}_0(\r{U}(V))$ (resp.
$\c{A}_0(\r{U}(W))$).

\begin{defn}
The following absolutely convergent integral is called a
\emph{Fourier-Jacobi period} of $\pi\otimes\sigma$ (for a pair
$(H',\nu'_{\psi',\mu})$):
 \begin{align*}
 \FJ^{\nu'_{\psi',\mu}}_r(\varphi_{\pi},\varphi_{\sigma};\phi):=\intl_{H'(k')\B
 H'(\d{A}')}\varphi_{\pi}(\vep(h'))\varphi_{\sigma}(\kappa(h'))
 \theta_{\overline{\psi'},\overline{\lambda'},\mu^{-1}}(h';\phi)\r{d}h',\qquad
 \varphi_\pi\in\c{A}_{\pi},\varphi_{\sigma}\in\c{A}_{\sigma},\phi\in\c{S}(\b{L}(\d{A}'))
 \end{align*}
where $\r{d}h'$ is the Tamagawa measure on $H'(\d{A}')$. If there
exist $\varphi_{\pi}$, $\varphi_{\sigma}$, $\phi$ such that
$\FJ^{\nu'_{\psi',\mu}}_r(\varphi_{\pi},\varphi_{\sigma};\phi)\neq0$,
then we say $\pi\otimes\sigma$ has a nontrivial Fourier-Jacobi
period.
\end{defn}

It is obvious that $\FJ^{\nu'_{\psi',\mu}}_r$ defines an element in
 \begin{align*}
 \r{Hom}_{H'(\d{A}')}\(\pi\otimes\sigma\otimes\widetilde{\nu'_{\psi',\mu}},\d{C}^{\times}\)=\bigotimes_{v'\in\d{M}_{k'}}
 \r{Hom}_{H'_{v'}}\(\pi_{v'}\otimes\sigma_{v'}\otimes\widetilde{\nu'_{\psi'_{v'},\mu_{v'}}},\d{C}\).
 \end{align*}\\

We choose a basis $\1 v_1,...,v_r\2$ or $X$ as in Section
\ref{sec4bmp}. We denote by $\1\check{v}_1,...,\check{v}_r\2$ the
dual basis of $\check{X}$. We also choose a basis $\1 w_1,...,w_m\2$
of $W$. Let
$\beta=\[(w_i,w_j)\]_{i,j=1}^m\in\r{Her}_m(k/k')^{\times}$. We
identify $\r{U}(V)$ (resp. $\r{U}(W)$) with a unitary group of $n$
(resp. $m$) variables $\r{U}_n$ (resp. $\r{U}_m$) under the basis
$\{v_1,...,v_r,w_1,...,w_m,\check{v}_r,...,\check{v}_1\}$ and view
$\r{U}_m$ as a subgroup of $\r{U}_n$. Let
$U'_{1^r,m}=U'_{r,m}\rtimes U'_X$ be a unipotent subgroup of
$\r{U}_n$. Then the image of $H'(\d{A}')$ in $\r{U}_n(\d{A}')$
consists of the matrices
$h'=h'(n',b';u';g')=\underline{u}'(n',b';u')\cdot g'$ where
$g'\in\r{U}_m(\d{A}')$. Here,
 \begin{align*}
 \underline{u}'=\underline{u}'(n',b';u')=\[\begin{array}{ccc}
                                             \b{1}_r & n' &  w_r\(b'+\frac{n'n'_{\beta}}{2}\) \\
                                              & \b{1}_m & n'_{\beta}\\
                                              &  & \b{1}_r
                                           \end{array}
 \]
 \[\begin{array}{ccc}
     u' &  &  \\
      &  \b{1}_m &  \\
      &  & \check{u}'
   \end{array}
 \]\in U'_{1^r,m}(\d{A}')
 \end{align*}
where $n'\in\r{M}_{r,m}(\d{A})$, $b'\in\r{Her}^-_r(\d{A}/\d{A}')$,
$u'\in U'_X(\d{A})$; $n'_{\beta}=-\beta^{-1}\:^{\r{t}}n'^{\tau}w_r$
and $\check{u}'=w_r\:^{\r{t}}u'^{\tau,-1}w_r$. If $r>0$, let
$U^{\ddag}$ be the maximal unipotent subgroup of the parabolic
subgroup of $\r{U}(\1v_r,\check{v}_r\2\oplus W)$ stabilizing the
flag $0\subset\1 v_r\2$. Let $H^{\ddag}=U^{\ddag}\rtimes\r{U}(W)$
(resp. $H^{\ddag}=\r{U}(W)$) if $r>0$ (resp. $r=0$). Then there is a
map $H'\R H^{\ddag}$. We write $h^{\ddag}=\underline{u}^{\ddag}g'$
to be the image of $h'$ under this map. Then we have
 \begin{align*}
 \nu'_{\psi',\mu}(h')=\underline{\psi'}(\underline{u}')\omega'_{\psi',\mu}(h^{\ddag})
 =\psi'\(\widetilde{\r{Tr}}\(u'_{1,2}+\cdots+u'_{r-1,r}\)\)\omega'_{\psi',\mu}(h^{\ddag}).
 \end{align*}\\

\subsection{Decomposition of distributions}
\label{sec5dd}

This time, we start from the relative trace formula on general
linear groups. We identify $\r{GL}_{m,k}\subset\r{GL}_{n,k}$ with
$\r{GL}(W)\subset\r{GL}(V)$ and view
$\r{GL}_{n,k'}\subset\r{Res}_{k/k'}(\r{GL}_{n,k})$ (res.
$\r{GL}_{m,k'}\subset\r{Res}_{k/k'}(\r{GL}_{m,k})$) through the
basis $\1 v_1,...,v_r,w_1,...,w_m,\check{v}_r,...,\check{v}_1\2$.

Recall that the Weil representation $\omega_{\psi,\mu}$ realizes on
the space $\c{S}(W^{\vee}(\d{A}))$ where
$\psi=\psi'\circ\widetilde{\r{Tr}}$. We let
 \begin{align*}
 W^{\spadesuit}=\bigoplus_{i=1}^r k'w^{\vee}_i;\qquad
 W^{\heartsuit}=\bigoplus_{i=1}^r k^-w^{\vee}_i;\qquad
 W^{\clubsuit}=\bigoplus_{i=1}^r k'w_i;\qquad
 W^{\diamondsuit}=\bigoplus_{i=1}^rk^-w_i
 \end{align*}
which are vector spaces over $k'$. Then $W=W^{\clubsuit}\oplus
W^{\diamondsuit}$, $W^{\vee}=W^{\spadesuit}\oplus W^{\heartsuit}$.
We also let $W^{\dag}=W^{\spadesuit}\oplus W^{\clubsuit}$ be a
vector space over $k'$. We define a linear map from
$\c{S}(W^{\vee}(\d{A}))$ to $\c{S}(W^{\dag}(\d{A}'))$ by
$\Phi\T\Phi^{\dag}$, where
 \begin{align*}
 \Phi^{\dag}(w^{\spadesuit},w^{\clubsuit})=\intl_{W^{\heartsuit}(\d{A'})}
 \Phi(w^{\spadesuit},w^{\heartsuit})\psi\(w^{\heartsuit}(w^{\clubsuit})\)\r{d}w^{\heartsuit}
 \end{align*}
for $\Phi\in\c{S}(W^{\vee}(\d{A}))$ and the self-dual measure on
$W^{\heartsuit}(\d{A}')$, which is an isomorphism.

If $r>0$, let $U^{\dag}$ be the maximal unipotent subgroup of the
parabolic subgroup of $\r{GL}(\1v_r,\check{v}_r\2\oplus W)$
stabilizing the flag $0\subset\1 v_r\2\subset\1 v_r\2\oplus W$. Let
$H^{\dag}=U^{\dag}\rtimes\r{GL}_m$ (resp. $H^{\dag}=\r{GL}_m$) if
$r>0$ (resp. $r=0$). Then there is a map $H\R H^{\dag}$. We write
$h^{\dag}=\underline{u}^{\dag}g=\underline{u}^{\dag}
\(n^{\spadesuit},n^{\heartsuit},n^{\clubsuit},n^{\diamondsuit},b^{\dag}\)g$
to be the image of $h$ under this map, where
 \begin{align*}
 \underline{u}^{\dag}
\(n^{\spadesuit},n^{\heartsuit},n^{\clubsuit},n^{\diamondsuit},b^{\dag}\)=\[
\begin{array}{ccc}
  1 & n^{\spadesuit}+n^{\heartsuit} & b^{\dag} \\
   & \b{1}_m & \begin{array}{c}
                 n^{\clubsuit} \\
                 + \\
                 n^{\diamondsuit}
               \end{array}
    \\
   &  & 1
\end{array}
\]
 \end{align*}
with $n^{\spadesuit}\in\r{M}_{1,m}(k')$,
$n^{\heartsuit}\in\r{M}_{1,m}(k^-)$,
$n^{\clubsuit}\in\r{M}_{m,1}(k')$,
$n^{\diamondsuit}\in\r{M}_{m,1}(k^-)$.

We define a representation
$\omega^{\dag}_{\overline{\psi},\overline{\mu}}$ of $H(\d{A})$ on
$\c{S}(W^{\dag}(\d{A}'))$ by
 \begin{align*}
 \omega^{\dag}_{\overline{\psi},\overline{\mu}}(h)\Phi^{\dag}
 =\(\omega_{\overline{\psi},\overline{\mu}}(h)\Phi\)^{\dag}.
 \end{align*}
It is easy to see that
$\omega^{\dag}_{\overline{\psi},\overline{\mu}}$ factors through
$H^{\dag}$ and
 \begin{align}\label{5omega}
 \(\omega^{\dag}_{\overline{\psi},\overline{\mu}}(h^{\dag})\Phi^{\dag}\)(w^{\spadesuit},w^{\clubsuit})=\eta(\det
 g)\overline{\psi}\(b^{\dag}+w^{\spadesuit}n^{\diamondsuit}+n^{\heartsuit}w^{\clubsuit}\)
 \Phi^{\dag}\((w^{\spadesuit}+n^{\spadesuit})g,g^{-1}(w^{\clubsuit}-n^{\clubsuit})\)
 \end{align}
Moreover, we have the Poisson summation formula:
 \begin{align*}
 \sum_{w^{\dag}\in W^{\dag}(k')}\Phi^{\dag}(w^{\dag})=\sum_{w^{\flat}\in
 W^{\vee}(k)}\Phi(w^{\flat}).
 \end{align*}\\

Until the end of this section, we assume that $n$ is odd, hence $m$
is also odd. Since the other case is similar and will lead to the
same fundamental lemma, we omit it in the following discussion. We
proceed exactly as in Section \ref{sec4dd} and take $\mu$ to be the
one used in \eqref{5mu} which is unitary. Let
$F_n\in\c{H}(\r{GL}_n(\d{A}))$, $F_m\in\c{H}(\r{GL}_m(\d{A}))$ and
$\Phi\in\c{S}(W^{\vee}(\d{A}))$. We introduce the following
distribution
 \begin{align*}
 \c{J}^{\mu}_{\Pi,\Sigma}(s;F_n\otimes
 F_m;\Phi):=\sum\FJ^{\nu_{\mu}}_{r,r}(s;\rho(F_n)\varphi_{\Pi},\rho(F_m)\varphi_{\Sigma};\Phi)
 \overline{\c{P}(\varphi_{\Pi})\c{P}(\varphi_{\Sigma})}
 \end{align*}
where the sum is taken over orthonormal bases of $\c{A}_{\Pi}$ and
$\c{A}_{\Sigma}$.

We associate to $F_n\otimes F_m$ a kernel function
$\c{K}_{F_n\otimes F_m}(g_1,g_2;g_3,g_4)$ and consider the
distribution
 \begin{align}\label{5di1}
 \c{J}^{\mu}(s;F_n\otimes
 F_m;\Phi)=&\intl_{Z'_m(\d{A}')\r{GL}_m(k')\B\r{GL}_m(\d{A}')}\intl_{Z'_n(\d{A}')\r{GL}_n(k')\B\r{GL}_n(\d{A}')}
 \intl_{H(k)\B H(\d{A})}\notag\\
 &\c{K}_{F_n\otimes
 F_m}(\vep(h),\kappa(h);g_1,g_2)\theta_{\overline{\psi},\overline{\lambda},\overline{\mu}}(h,\Phi)
 |\det h|_{\d{A}}^{s-\frac{1}{2}}\r{d}h\r{d}g_1\r{d}g_2.
 \end{align}
Proceeding similarly as in \eqref{4di1}, we have
 \begin{align}\label{5di2}
 \eqref{5di1}=&\intl_{g_2}\intl_{g_1}\intl_{U_{1^r,m,1^r}(k)\B
 U_{1^r,m,1^r}(\d{A})}\intl_{\r{GL}_m(\d{A})}\sum_{\zeta\in
 \r{GL}_n(k)}\notag\\
 &F_n(g^{-1}g_2^{-1}\underline{u}^{-1}\zeta
 g_1)F_m(g^{-1})\theta_{\overline{\psi},\overline{\lambda},\overline{\mu}}(\underline{u}g_2g,\Phi)|\det
 g|_{\d{A}}^{s-\frac{1}{2}}|\det
 g_2|_{\d{A}}^{s-\frac{1}{2}}\r{d}g\r{d}\underline{u}\r{d}g_1\r{d}g_2\notag\\
 =&\intl_{\r{GL}_m(k')\B\r{GL}_m(\d{A}')}\intl_{U_{1^r,m,1^r}(k)\B
 U_{1^r,m,1^r}(\d{A})}\intl_{\r{GL}_m(\d{A})}\sum_{\zeta\in
 \r{S}_n(k')}\notag\\
 &\sigma(F_n)(g^{-1}g_2^{-1}\underline{u}^{-1}\zeta
 \underline{u}^{\tau}g_2g^{\tau})F_m(g^{-1})\theta_{\overline{\psi},\overline{\lambda},\overline{\mu}}
 (\underline{u}g_2g,\Phi)|\det g|_{\d{A}}^{s-\frac{1}{2}}|\det
 g_2|_{\d{A}}^{s-\frac{1}{2}}\r{d}g\r{d}\underline{u}\r{d}g_2.
 \end{align}
Unfolding $U_{1^r,m,1^r}(k)$, we have
 \begin{align}\label{5di3}
 \eqref{5di2}=&\sum_{\zeta\in[\r{S}_n(k')]/U_{1^r,m,1^r}(k)}
 \intl_{\r{GL}_m(k')\B\r{GL}_m(\d{A}')}\intl_{\r{Stab}_{\zeta}^{U_{1^r,m,1^r}}(k')\B
 U_{1^r,m,1^r}(\d{A})}\intl_{\r{GL}_m(\d{A})}\notag\\
 &\sigma(F_n)(g^{-1}g_2^{-1}([\zeta]\underline{u})g_2g^{\tau})F_m(g^{-1})
 \theta_{\overline{\psi},\overline{\lambda},\overline{\mu}}(\underline{u}g_2g,\Phi)|\det
 g|_{\d{A}}^{s-\frac{1}{2}}|\det
 g_2|_{\d{A}}^{s-\frac{1}{2}}\r{d}g\r{d}\underline{u}\r{d}g_2\notag\\
 =&\sum_{\zeta\in[\r{S}_n(k')]/U_{1^r,m,1^r}(k)}\intl_{\r{Stab}_{\zeta}^{U_{1^r,m,1^r}}(k')\B
 U_{1^r,m,1^r}(\d{A})}\intl_{\r{GL}_m(k')\B\r{GL}_m(\d{A}')}\intl_{\r{GL}_m(\d{A})}\notag\\
 &\sigma(F_n)(g^{-1}([g_2^{-1}\zeta g_2]\underline{u})g^{\tau})F_m(g^{-1})
 \theta_{\overline{\psi},\overline{\lambda},\overline{\mu}}
 \(g_2,\omega_{\overline{\psi},\overline{\mu}}(\underline{u}g)\Phi\)
 \underline{\psi}(\underline{u}^{-1})|\det g|_{\d{A}}^{s-\frac{1}{2}}|\det
 g_2|_{\d{A}}^{s-\frac{1}{2}}\r{d}g\r{d}g_2\r{d}\underline{u}.
 \end{align}

But we have
 \begin{align*}
 &\theta_{\overline{\psi},\overline{\lambda},\overline{\mu}}
 \(g_2,\omega_{\overline{\psi},\overline{\mu}}(\underline{u}g)\Phi\)
 =\sum_{w^{\flat}\in
 W^{\vee}(k)}\(\omega_{\overline{\psi},\overline{\mu}}(g_2)
 \omega_{\overline{\psi},\overline{\mu}}(\underline{u}g)\Phi\)(w^{\flat})\\
 =&\sum_{w^{\dag}\in
 W^{\dag}(k')}\(\omega_{\overline{\psi},\overline{\mu}}^{\dag}(g_2)
 \(\omega_{\overline{\psi},\overline{\mu}}(\underline{u}g)\Phi\)^{\dag}\)(w^{\dag})\\
 =&\sum_{x\in\r{M}_{1,m}(k')}\sum_{y\in\r{M}_{m,1}(k')}
 \eta(\det
 g_2)\(\omega_{\overline{\psi},\overline{\mu}}(\underline{u}g)\Phi\)^{\dag}(xg_2,g_2^{-1}y).
 \end{align*}

To proceed, we introduce a $k'$-variety
 \begin{align*}
 \r{S}_{n,m}=\r{S}_n\times\r{M}_{1,m,k'}\times\r{M}_{m,1,k'}.
 \end{align*}

As before, we let
$\b{H}=\r{Res}_{k/k'}(U_{1^r,m,1^r})\rtimes\r{GL}_{m,k'}$ which acts
on $\r{S}_{n,m}$ in the following way: for any $k'$-algebra $R$,
$\b{h}=\b{h}(\underline{u},g)$ with $\underline{u}\in
U_{1^r,m,1^r}(R\otimes k)$, $g\in\r{GL}_m(R)$ and
$[s,x,y]\in\r{S}_{n,m}(R)$, we define a right action
$[s,x,y]\b{h}=[g^{-1}\underline{u}^{-1}su^{\tau}g,xg,g^{-1}y]$. We
also define
$\underline{\psi}(\b{h})=\underline{\psi}(\underline{u}^{-1})$ and
$\det\b{h}=\det g$. Then
 \begin{align}\label{5di4}
 \eqref{5di3}=&\sum_{[\zeta,x,y]\in[\r{S}_{n,m}(k')]/\b{H}(k')}
 \intl_{\r{Stab}_{[\zeta,x,y]}^{\b{H}}(k')\B\b{H}(\d{A}')}\intl_{\r{GL}_m(\d{A})}\notag\\
 &\sigma(F_n)(g^{-1}[\zeta]\b{h}g^{\tau})F_m(g^{-1})
 \(\omega^{\dag}_{\overline{\psi},\overline{\mu}}(\b{h}g)\Phi^{\dag}\)(x,y)
 \underline{\psi}(\b{h})|\det\b{h}|^{s-\frac{1}{2}}_{\d{A}}|\det
 g|_{\d{A}}^{s-\frac{1}{2}}\r{d}g\r{d}\b{h}\notag\\
 =&:\sum_{[\zeta,x,y]\in[\r{S}_{n,m}(k')]/\b{H}(k')}
 \c{J}^{\mu}_{[\zeta,x,y]}(s;F_n\otimes F_m;\Phi).
 \end{align}

We denote by $\r{S}_{n,m}(k')_{\r{reg}}$ the set of all regular
$k'$-elements which will be defined in Section \ref{sec5co}. In
particular, the $\b{H}$-stabilizer $\r{Stab}_{[\zeta,x,y]}^{\b{H}}$
is trivial for $[\zeta,x,y]$ regular and the corresponding term
\begin{align*}
\c{O}_{\mu}(s;F_n\otimes
F_m;\Phi,[\zeta,x,y]):=\c{J}^{\mu}_{[\zeta,x,y]}(s;F_n\otimes
F_m;\Phi)
\end{align*}
is a weighted orbital integral. If $F_n=\otimes_{v'}F_{n,v'}$,
$F_m=\otimes_{v'}F_{m,v'}$ and $\Phi=\otimes_{v'}\Phi_{v'}$ are
factorizable, then
 \begin{align*}
 \c{O}_{\mu}(s;F_n\otimes
F_m;\Phi,[\zeta,x,y])=\prod_{v'\in\d{M}_{k'}}\c{O}_{\mu_{v'}}(s;F_{n,v'}\otimes
F_{m,v'};\Phi_{v'},[\zeta,x,y])
 \end{align*}
where the local orbital integrals are defined similarly as in
\eqref{5di4}. In summary, we have
 \begin{align*}
&\c{J}^{\mu}(s;F_n\otimes
F_m;\Phi)-\c{J}^{\mu}_{\r{irr}}(s;F_n\otimes F_m;\Phi)\\
=&\sum_{[\zeta,x,y]\in[\r{S}_{n,m}(k')_{\r{reg}}]/\b{H}(k')}
\prod_{v'\in\d{M}_{k'}}\c{O}_{\mu_{v'}}(s;F_{n,v'}\otimes
F_{m,v'};\Phi_{v'},[\zeta,x,y]).
 \end{align*}

When $s=\frac{1}{2}$, we discard $s$ in all notations. In
particular,
 \begin{align}\label{5di5}
 &\c{J}_{[\zeta,x,y]}^{\mu}(F_n\otimes F_m;\Phi)=\c{O}_{\mu}(F_n\otimes
 F_m;\Phi,[\zeta,x,y]) \notag\\
 =&\intl_{\b{H}(\d{A}')}\intl_{\r{GL}_m(\d{A})}\sigma(F_n)(g^{-1}[\zeta]\b{h}g^{\tau})F_m(g^{-1})
 \(\omega^{\dag}_{\overline{\psi},\overline{\mu}}(\b{h}g)\Phi^{\dag}\)(x,y)
 \underline{\psi}(\b{h})\r{d}g\r{d}\b{h}
 \end{align}
when $[\zeta,x,y]$ is regular.\\

Now we describe the relative trace formula for unitary groups. Let
$f_n\in\c{H}(\r{U}_n(\d{A}'))$, $f_m\in\c{H}(\r{U}_m(\d{A}'))$ and
$\phi_i\in\c{S}(\b{L}(\d{A}'))$ for $i=1,2$. We introduce a
distribution
 \begin{align*}
 \c{J}^{\psi',\mu}_{\pi,\sigma}(f_n\otimes
 f_m;\phi_1\otimes\phi_2):=\sum\FJ^{\nu'_{\psi',\mu}}_r(\rho(f_n)\varphi_{\pi},\rho(f_m)\varphi_{\sigma};\phi_1)
 \overline{\FJ^{\nu'_{\psi',\mu}}_r(\varphi_{\pi},\varphi_{\sigma};\phi_2)}
 \end{align*}
where the sum is taken over orthonormal bases of $\c{A}_{\pi}$ and
$\c{A}_{\sigma}$. As usual, we have a kernel function
$\c{K}_{f_n\otimes f_m}$. Let
 \begin{align}\label{5du1}
 &\c{J}^{\psi',\mu}(f_n\otimes f_m;\phi_1\otimes\phi_2)\notag\\
 :=&\iintl_{\(H'(k')\B
 H'(\d{A}')\)^2}\c{K}_{f_n\otimes f_m}(\vep(h'_1),\kappa(h'_1);\vep(h'_2),\kappa(h'_2))
 \theta_{\overline{\psi'},\overline{\lambda'},\overline{\mu}}(h'_1,\phi_1)
 \overline{\theta_{\overline{\psi'},\overline{\lambda'},\overline{\mu}}(h'_2,\phi_2)}\r{d}h'_1\r{d}h'_2\notag\\
 =&\iintl_{\(H'(k')\B
 H'(\d{A}')\)^2}\c{K}_{f_n\otimes f_m}(\vep(h'_1),\kappa(h'_1);\vep(h'_2),\kappa(h'_2))
 \theta_{\overline{\psi'},\overline{\lambda'},\overline{\mu}}(h'_1,\phi_1)
 \theta_{\psi',\lambda',\mu}\(h'_2,\overline{\phi_2}\)\r{d}h'_1\r{d}h'_2.
 \end{align}

Collapsing the summation over $\xi'$ and changing variable
$g'^{-1}_2g'_1\T g'_1$, then
 \begin{align}\label{5du3}
 \eqref{5du1}=&\intl_{\r{U}_m(k')\B\r{U}_m(\d{A}')}\iintl_{\(U'_{1^r,m}(k')\B
 U'_{1^r,m}(\d{A}')\)^2}\intl_{\r{U}_m(\d{A}')}\sum_{\zeta'\in\r{U}_n(k')}\notag\\
 &f_n(g'^{-1}_1g'^{-1}_2\underline{u}'^{-1}_1\zeta'\underline{u}'_2g'_2)f_m(g'^{-1}_1)
 \theta_{\overline{\psi'},\overline{\lambda'},\overline{\mu}}\(\underline{u}'_1g'_2,
 \omega'_{\overline{\psi'},\overline{\mu}}(g'_1)\phi_1\)\theta_{\psi',\lambda',\mu}
 \(\underline{u}'_2g'_2,\overline{\phi_2}\)\r{d}g'_1\r{d}\underline{u}'_1\r{d}\underline{u}'_2\r{d}g'_2.
 \end{align}

Recall that we define an $k'$-algebraic group $\b{H}'$ in Section
\ref{sec4dd} which acts on
$\r{U}_{n,m}:=\r{U}_n\times\r{Res}_{k/k'}\r{M}_{1,m,k}$ in the
following way: for any $k'$-algebra $R$,
$\b{h}'=\b{h}'(\underline{u}'_1,\underline{u}'_2,g')\in\b{H}'(R)$
and $[g,z]\in\r{U}_n(R)\times\r{M}_{1,m}(R\otimes k)$, we define the
right action
$[g,z]\b{h}'=[g'^{-1}\underline{u}'^{-1}_1g\underline{u}'_2g',zg']$.
We also define
 \begin{align*}
 \underline{\psi'}(\b{h}')=\underline{\psi'}\(\b{h}'(\underline{u}'_1,\underline{u}'_2;g')\)
 :=\underline{\psi'}(\underline{u}'^{-1}_1\underline{u}'_2)
 \end{align*}
and $\det\b{h}'=\det g'$.

There is a map $\b{H}'\R\b{H}^{\ddag}$, where
 \begin{align*}
 \b{H}^{\ddag}=H^{\ddag}\underset{\r{U}_m}{\times}H^{\ddag}
 \end{align*}
(resp. $\r{U}_m$) if $r>0$ (resp. $r=0$). We denote by
$\b{h}^{\ddag}=\b{h}^{\ddag}(\underline{u}_1^{\ddag},\underline{u}_2^{\ddag};g')$
the image of $\b{h}'=\b{h}'(\underline{u}'_1,\underline{u}'_2;g')$
under the above map, where
 \begin{align*}
 \underline{u}^{\ddag}_i=\underline{u}_i^{\ddag}(n^{\ddag}_i,b^{\ddag}_i)=\[
 \begin{array}{ccc}
   1 & n^{\ddag}_i & b^{\ddag}_i-\frac{n_i^{\ddag}\beta^{-1}\:^{\r{t}}n_i^{\ddag,\tau}}{2} \\
    & \b{1}_m & -\beta^{-1}\:^{\r{t}}n^{\ddag,\tau}_i \\
    &  & 1
 \end{array}
 \]
 \end{align*}
for $n_i^{\ddag}\in W^{\vee}(\d{A})$ and $b_i^{\ddag}\in\d{A}^-$.\\

\begin{lem}
Let $\omega^{\ddag}_{\overline{\psi'}}$ be the tensor product
$\omega'_{\overline{\psi'},\overline{\mu}}\otimes\omega'_{\psi',\mu}$
viewed as a smooth representation of $\b{H}'$, then it factors
through $\b{H}'\R\b{H}^{\ddag}$. We have an intertwining isomorphism
 \begin{align*}
 -^{\ddag}:\(\c{S}(\b{L}(\d{A}'))\)^{\otimes2}
 \LR\c{S}\(\(\r{Res}_{k/k'}W\)^{\vee}(\d{A}')\)=\c{S}(W^{\vee}(\d{A}))
 \end{align*}
such that:\\
(1) For $\phi\in\c{S}(W^{\vee}(\d{A}))$,
 \begin{align}\label{5omegau}
 &\(\omega^{\ddag}_{\overline{\psi'}}(\b{h}')\phi\)(z)=
 \(\omega^{\ddag}_{\overline{\psi'}}(\b{h}^{\ddag})\phi\)(z)
 =\(\omega^{\ddag}_{\overline{\psi'}}
 \(\b{h}^{\ddag}(\underline{u}_1^{\ddag},\underline{u}_2^{\ddag};g')\)\phi\)(z)\notag\\
 =&\overline{\psi}\(b_1^{\ddag}-b_2^{\ddag}+z\beta^{-1}
 \frac{^{\r{t}}n_2^{\ddag,\tau}-\:^{\r{t}}n_1^{\ddag,\tau}}{2}\)
 \phi\(\(z+\frac{n_1^{\ddag}+n_2^{\ddag}}{2}\)g'\)
 \end{align}
which does not depend on $\mu$, hence justifying the notation.\\
(2) We have
 \begin{align*}
 \theta_{\overline{\psi'},\overline{\lambda'},\overline{\mu}}\(\underline{u}'_1g',
 \phi_1\)\theta_{\psi',\lambda',\mu}
 \(\underline{u}'_2g',\phi_2\)=\underline{\psi'}(\b{h}')\sum_{z\in
 W^{\vee}(k)}\(\omega^{\ddag}_{\overline{\psi'}}(\b{h}')\(\phi_1\otimes\phi_2\)^{\ddag}\)(z)
 \end{align*}
where $\b{h}'=\b{h}'(\underline{u}'_1,\underline{u}'_2,g')$.
\end{lem}
\begin{proof}
The isomorphism is by \cite[Proposition 2.2 (i), (ii)]{HKS96}.
Actually one can construct an explicit intertwining operator by a
(partial) Fourier transform, which implies (1) and (2) by the
Poisson summation formula.
\end{proof}

By the above lemma and repeating the process in \eqref{5di2},
\eqref{5di3}, \eqref{5di4}, we have
 \begin{align}\label{5du4}
 \eqref{5du3}=&\sum_{[\zeta',z]\in[\r{U}_{n,m}(k')]/\b{H}'(k')}
 \intl_{\r{Stab}_{[\zeta',z]}^{\b{H}'}(k')\B\b{H}'(\d{A}')}\intl_{\r{U}_m(\d{A}')}\notag\\
 &f_n(g'^{-1}[\zeta']\b{h}')f_m(g'^{-1})\(\omega^{\ddag}_{\overline{\psi'}}(\b{h}')
 \(\omega'_{\overline{\psi'},\overline{\mu}}(g')\phi_1\otimes\overline{\phi_2}\)^{\ddag}\)(z)
 \underline{\psi'}(\b{h}')\r{d}g'\r{d}\b{h}'\notag\\
 =&:\sum_{[\zeta',z]\in[\r{U}_{n,m}(k')]/\b{H}'(k')}\c{J}^{\psi',\mu}_{[\zeta',z]}(f_n\otimes
 f_m;\phi_1\otimes\phi_2).
 \end{align}

We denote by $\r{U}_{n,m}(k')_{\r{reg}}$ the set of all regular
$k'$-elements which will be defined in Section \ref{sec5co}. In
particular, the $\b{H}'$-stabilizer $\r{Stab}_{[\zeta',z]}^{\b{H}'}$
is trivial for $[\zeta',z]$ regular and the corresponding term
\begin{align*}
\c{O}_{\psi',\mu}(f_n\otimes
f_m;\phi\otimes\phi_2,[\zeta',z]):=\c{J}^{\psi',\mu}_{[\zeta',z]}(f_n\otimes
f_m;\phi_1\otimes\phi_2)
\end{align*}
is a weighted orbital integral. If $f_n=\otimes_{v'}f_{n,v'}$,
$f_m=\otimes_{v'}f_{m,v'}$ and $\phi_i=\otimes_{v'}\phi_{i,v'}$ are
factorizable, then
 \begin{align*}
 \c{O}_{\psi',\mu}(f_n\otimes
f_m;\phi_1\otimes\phi_2,[\zeta',z])=\prod_{v'\in\d{M}_{k'}}\c{O}_{\psi'_{v'},\mu_{v'}}(f_{n,v'}\otimes
f_{m,v'};\phi_{1,v'}\otimes\phi_{2,v'},[\zeta',z])
 \end{align*}
where the local orbital integrals are defined similarly as in
\eqref{5du4}. In summary, we have
 \begin{align*}
&\c{J}^{\psi',\mu}(f_n\otimes
f_m;\phi_1\otimes\phi_2)-\c{J}^{\psi',\mu}_{\r{irr}}(f_n\otimes
f_m;\phi_1\otimes\phi_2)\\
=&\sum_{[\zeta',z]\in[\r{U}_{n,m}(k')_{\r{reg}}]/\b{H}'(k')}
\prod_{v'\in\d{M}_{k'}}\c{O}_{\psi'_{v'},\mu_{v'}}(f_{n,v'}\otimes
f_{m,v'};\phi_{1,v'}\otimes\phi_{2,v'},[\zeta',z]).
 \end{align*}

In general, if $\phi=\sum_i\phi_1^{(i)}\otimes\phi_2^{(i)}$ is a
finite sum, then we simply define the distribution $\c{J}$ and the
weighted orbital integral $\c{O}$ as the sum of the corresponding
terms defined above.\\

\subsection{Matching of orbits and functions}
\label{sec5co}

We first consider the following orbital problem which is a little
bit different from the
Jacquet-Rallis case, the one considered in Section \ref{sec4co}.\\

We define the affine space
$\f{M}_m:=\r{M}_m\times\r{M}_{1,m}\times\r{M}_{m,1}$.

 \begin{defn}\label{def5reg}
 An element $[\xi,x,y]\in\f{M}_m(k)$ is called \emph{regular} if it satisfies:
  \begin{itemize}
    \item $\xi$ is regular semisimple as an element of $\r{M}_m(k)$;
    \item the vectors $\1 x, x\xi,...,x\xi^{m-1}\2$ span the $k$-vector space $\r{M}_{1,m}(k)$;
    \item the vectors $\1 y,\xi y,...,\xi^{m-1}y\2$ span the $k$-vector space $\r{M}_{m,1}(k)$.
  \end{itemize}
 Let $a_i([\xi,x,y])=\r{Tr}\(\bigwedge^i\xi\)$ for $1\leq i\leq
 m$, $b_i([\xi,x,y])=x\xi^iy$ for $0\leq i\leq m-1$,
 \begin{align}\label{5bt}
  \b{T}_{[\xi,x,y]}=\det\[\begin{array}{c}
                          x \\
                          x\xi \\
                          \vdots\\
                          x\xi^{m-1}
                        \end{array}
  \],
 \end{align}
 $\b{D}_{[\xi,x,y]}$ be
 the matrix $\[x\xi^{i+j-2}y\]_{i,j=1}^m$ and
 $\Delta_{[\xi,x,y]}=\det\b{D}_{[\xi,x,y]}$. It is clear that
 $\Delta_{[\xi,x,y]}\neq0$ if $[\xi,x,y]$ is regular.
 The group $\r{GL}_m$ acts on $\f{M}_m$ from right side by $[\xi,x,y]g=[g^{-1}\xi
 g,xg,g^{-1}y]$ and $a_i$, $b_i$, $\b{D}_{[\xi,x,y]}$ and
 $\Delta_{[\xi,x,y]}$ are all invariants under this action. We
 denote by $\f{M}_m(k)_{\r{reg}}$ the set of all regular elements.
 \end{defn}

We have
 \begin{lem}
 Two regular elements $[\xi,x,y]$ and $[\xi',x',y']$ are in the same
 $\r{GL}_m(k)$-orbit if and only if they have the same invariants
 $a_i$ and $b_i$. The $\r{GL}_m$-stabilizer of a regular element is
 trivial.
 \end{lem}
 \begin{proof}
 See \cite[Proposition 6.2 \& Theorem 6.1]{RS07} for the first and
 second statements, respectively.\\
 \end{proof}

We define two more spaces:
 \begin{align*}
 \f{S}_m&=\1[\xi,x,y]\in\f{M}_m(k)_{\r{reg}}\;|\;\xi\in\r{S}_m(k'),x\in\r{M}_{1,m}(k'),y\in\r{M}_{m,1}(k')\2;\\
 \f{U}^{\natural}_m&=\1[\beta;\xi^{\beta},z,z^*]\;|\;\beta\in\r{Her}_m(k/k')^{\times},
 [\xi^{\beta},z,z^*]\in\f{M}_m(k)_{\r{reg}},
 \xi^{\beta}\in\r{U}^{\beta}_m(k'),z^*=\beta^{-1}\:^{\r{t}}z^{\tau}\2
 \end{align*}
where $\r{U}^{\beta}_m=\r{U}(W^{\beta})$. For $\f{U}^{\natural}_m$,
we also define a right $\r{GL}_m(k)$-action by
$[\beta;\xi^{\beta},z,z^*]g=[^{\r{t}}g^{\tau}\beta g;g^{-1}\xi
g,zg,g^{-1}z^*]$. For $[\xi,x,y]\in\f{S}_m$, we denote by
$[\xi,x,y]\Leftrightarrow[\beta;\xi^{\beta},z,z^*]$ if there exists
$g\in\r{GL}_m(k)$ such that $[\xi,x,y]=[\xi^{\beta},z,z^*]g$. We
have the following lemma which is similar to Lemma \ref{lem4zhang}:

 \begin{lem}\label{lem5zhang}
 For $[\xi,x,y]\in\f{S}_m$, there exists an element
 $[\beta;\xi^{\beta},z,z^*]\in\f{U}^{\natural}_m$, unique up to the $\r{GL}_m(k)$-action, such that
 $[\xi,x,y]\Leftrightarrow[\beta;\xi^{\beta},z,z^*]$. Conversely, for
 any $[\beta;\xi^{\beta},z,z^*]\in\f{U}^{\natural}_m$, there exists an element
 $[\xi,x,y]\in\f{S}_m$, unique up to the
 $\r{GL}_m(k')$-action, such that $[\xi,x,y]\Leftrightarrow[\beta;\xi^{\beta},z,z^*]$.
 \end{lem}
 \begin{proof}
 We first point out that $[\xi,x,y],[\xi',x',y']\in\f{S}_m$ are
 conjugate under $\r{GL}_m(k)$ if and only if they are conjugate under $\r{GL}_m(k')$. Assume
 that $[\xi,x,y]g=[\xi',x',y']$, then $g^{-1}\xi g=\xi'$ implies
 that $g^{\tau,-1}\xi g^{\tau}=\xi'$; $xg=x'$ implies that
 $xg^{\tau}=x'$; $g^{-1}y=y'$ implies that $g^{\tau,-1}y=y'$, hence
 $g=g^{\tau}$.

 It is easy to see that for $[\xi,x,y]\in\f{M}(k)_{\r{reg}}$,
 $[\xi,x,y]$ and $[^{\r{t}}\xi,\:^{\r{t}}y,\:^{\r{t}}x]$ have the
 same invariants, hence there is a unique $g\in\r{GL}_{m}(k)$ such
 that $g^{-1}\xi g=\:^{\r{t}}\xi$, $xg=\:^{\r{t}}y$,
 $g^{-1}y=\:^{\r{t}}x$. Now if $[\xi,x,y]\in\f{S}_m$,
 then
 $[^{\r{t}}\xi,\:^{\r{t}}y,\:^{\r{t}}x]\in\f{S}_m$,
 which implies that $g=g^{\tau}$. Moreover, we have $g=\:^{\r{t}}g$
 and $^{\r{t}}\xi^{\tau}\(g^{-1}\)\xi=g^{-1}$. Hence $g^{-1}\in\r{Her}_m(k/k')^{\times}$
 and $\xi\in\r{U}^{g^{-1}}_m(k')$. We
 also have $y=\(g^{-1}\)^{-1}\:^{\r{t}}x^{\tau}$
 which means that $[g^{-1};\xi,x,y]\in\f{U}^{\natural}_m$ and
 $[\xi,x,y]\Leftrightarrow[\jmath
 g^{-1};\xi,x,y]$.

 Conversely, given any $[\beta;\xi,z,z^*]\in\f{U}^{\natural}_m$.
 Since $^{\r{t}}\xi\beta^{\tau}\xi^{\tau}=\beta^{\tau}$,
 $\beta^{\tau,-1}\:^{\r{t}}\xi\beta^{\tau}=\xi^{\tau,-1}$, we have
  \begin{align*}
  [^{\r{t}}\xi,\:^{\r{t}}z^*,\:^{\r{t}}z]\beta^{\tau}=[\xi^{\tau,-1},z^{\tau},(z^*)^{\tau}]
  \end{align*}
 since $z^*=\beta^{-1}\:^{\r{t}}z^{\tau}$. Moreover, since
 $[\xi,z,z^*]$ and $[^{\r{t}}\xi,\:^{\r{t}}z^*,\:^{\r{t}}z]$ have
 the same invariants, there exists $\gamma\in\r{GL}_m(k)$ such that
 $[\xi,z,z^*]\gamma=[^{\r{t}}\xi,\:^{\r{t}}z^*,\:^{\r{t}}z]$,
 hence
  \begin{align*}
  [\xi,z,z^*]\(\gamma\beta^{\tau}\)=[\xi^{\tau,-1},z^{\tau},(z^*)^{\tau}]
  \end{align*}
 which implies that $\gamma\beta^{\tau}\in\r{S}_m(k')$, hence
 $\gamma\beta^{\tau}=gg^{\tau,-1}$ for some $g\in\r{GL}_m(k)$. Then
  \begin{align*}
  \(g^{-1}\xi g\)\(g^{-1}\xi g\)^{\tau}=\b{1}_m
  \end{align*}
 and $zgg^{\tau,-1}=z^{\tau}$ implies $zg=(zg)^{\tau}$;
 $g^{\tau}g^{-1}z^*=(z^*)^{\tau}$ implies
 $g^{-1}z^*=\(g^{-1}z^*\)^{\tau}$. In all, $[\xi,z,z^*]g=[g^{-1}\xi
 g,zg,g^{-1}z^*]\in\f{S}_m$. The uniqueness is obvious.\\
 \end{proof}

Recall that we denote $[\r{Her}_m(k/k')^{\times}]$ the set of
similarity classes. We write $W=W^{\beta}$, $V=V^{\beta}$ if the
representing matrix of $W$ is in the class $\beta$, and
$\r{U}_m^{\beta}$ (resp. $\r{U}_n^{\beta}$, $\r{U}_{n,m}^{\beta}$,
$\b{H}^{\beta}$) for $\r{U}(W^{\beta})$ (resp. $\r{U}(V^{\beta})$, $
\r{U}(V^{\beta})\times\r{Res}_{k/k'}\r{M}_{1,m,k}$, $\b{H}'$). We
have

\begin{defn}
An element $[\zeta,x,y]\in\r{S}_{n,m}(k')$ (resp.
$[\zeta^{\beta},z]\in\r{U}^{\beta}_{n,m}(k')$) is called
\emph{pre-regular} if the stabilizer of $\zeta$ (resp.
$\zeta^{\beta}$) under the action of $\r{Res}_{k/k'}U_{1^r,m,1^r}$
(resp. $\(U'_{1^r,m}\)^2$) is trivial.\\
\end{defn}

Applying Lemma \ref{lem4weyl} to $\b{P}=P_{1^r,m,1^r}$, the
$U_{1^r,m,1^r}(k)$-orbit of $\zeta$ for which $[\zeta,x,y]$ is
pre-regular must contain a unique element of the form
\eqref{4normal} with $t_i(\zeta)\in k^{\times}$ and
$\r{Pr}(\zeta)\in\r{S}_{m}(k')$. We call the triple
$[\r{Pr}(\zeta),x,y]\in\f{M}_m(k)$ the \emph{normal form} of
$[\zeta,x,y]$.

Applying Lemma \ref{lem4weylu} to $\b{P}'=P^{\beta}_{1^r,m}$, the
$\(U'_{1^r,m}\)^2$-orbit of $\zeta^{\beta}$ for which
$[\zeta^{\beta},z]$ is pre-regular must contain a unique element of
the form \eqref{4unormal} with $t_i(\zeta^{\beta})\in k^{\times}$
and $\r{Pr}(\zeta^{\beta})\in\r{U}^{\beta}_{m}(k')$. We call the
quadruple
$[\beta,\r{Pr}(\zeta^{\beta}),z,z^*]\in\r{Her}_m(k/k')^{\times}\times\f{M}_m(k)$
the \emph{normal form} of $[\zeta^{\beta},z]$.\\

\begin{defn}
An element $[\zeta,x,y]\in\r{S}_{n,m}(k')$ (resp.
$[\zeta^{\beta},z]\in\r{U}^{\beta}_{n,m}(k')$) is called
\emph{regular} if it is pre-regular and its normal form
$[\r{Pr}(\zeta),x,y]\in\f{S}_m$ (resp.
$[\beta;\r{Pr}(\zeta^{\beta}),z,z^*]\in\f{U}^{\natural}_m$). We have
the notions $\r{S}_{n,m}(k')_{\r{reg}}$,
$\r{U}_{n,m}^{\beta}(k')_{\r{reg}}$ for the sets of regular elements.\\
\end{defn}

As before, we have the following proposition whose proof we omit.

\begin{prop}\label{prop5orbit}
 (1). There is a natural bijection
  \begin{align*}
  [\r{S}_{n,m}(k')_{\r{reg}}]/\b{H}(k')
  \overset{\b{N}}{\longleftrightarrow}
  \coprod_{\beta\in[\r{Her}_m(k/k')^{\times}]}[\r{U}_{n,m}^{\beta}(k')_{\r{reg}}]/\b{H}^{\beta}(k').
  \end{align*}
 If $\b{N}[\zeta^{\beta},z]=[\zeta,x,y]$, then we say that they
 \emph{match} and denote by $[\zeta,x,y]\leftrightarrow[\zeta^{\beta},z]$.

 (2). The set $\r{S}_{n,m}(k')_{\r{reg}}$ (resp.
 $\r{U}_{n,m}^{\beta}(k')_{\r{reg}}$) is non-empty and Zariski open in
 $\r{S}_{n,m}$ (resp. $\r{U}_{n,m}^{\beta}$). Moreover, the $\b{H}$-stabilizer (resp.
 $\b{H}^{\beta}$-stabilizer) of regular $[\zeta,x,y]$ (resp. $[\zeta^{\beta},z]$) is
 trivial.\\
\end{prop}

It is clear that the regular orbit
$[\zeta,x,y]\in[\r{S}_{n,m}(k')_{\r{reg}}]/\b{H}(k')$ (resp.
$[\zeta^{\beta},z]\in[\r{U}_{n,m}^{\beta}(k')_{\r{reg}}]/\b{H}^{\beta}(k')$)
is determined by its invariants $t_i(\zeta)$ (resp.
$t_i(\zeta^{\beta})$) ($i=1,...,r$),
$a_i([\zeta,x,y]):=a_i([\r{Pr}(\zeta),x,y])$ (resp.
$a_i([\zeta^{\beta},z]):=a_i([\r{Pr}(\zeta^{\beta}),z,z^*])$)
($i=1,...,m$) and $b_i([\zeta,x,y]):=b_i([\r{Pr}(\zeta),x,y])$
(resp. $b_i([\zeta^{\beta},z]):=b_i([\r{Pr}(\zeta^{\beta}),z,z^*])$)
($i=0,...,m-1$), and $[\zeta,x,y]\leftrightarrow[\zeta^{\beta},z]$
if and only if they have the same invariants. We also denote by
$\b{T}_{[\zeta,x,y]}=\b{T}_{[\r{Pr}(\zeta),x,y]}$,
$\b{D}_{[\zeta,x,y]}=\b{D}_{[\r{Pr}(\zeta),x,y]}$ and
$\Delta_{[\zeta,x,y]}=\Delta_{[\r{Pr}(\zeta),x,y]}$.\\

As at the end of Section \ref{sec4dd}, we would like to compare the
weighted orbital integral $\c{O}_{\mu_{v'}}(F_{n,v'}\otimes
F_{m,v'};\Phi_{v'},[\zeta,x,y])$ and
$\c{O}_{\psi'_{v'},\mu_{v'}}(f_{n,v'}\otimes
f_{m,v'};\phi_{1,v'}\otimes\phi_{2,v'},[\zeta',z])$ when $v'$ splits
into two places $v_{\bullet}$ and $v_{\circ}$ of $k$.

As before, we identify $\r{S}_{n,v'}$ with $\r{GL}_{n,v'}$ (resp.
$\r{U}_{n,v'}$ with $\r{GL}_{n,v'}$) and $\r{S}_{n,m,v'}$ with
$\r{GL}_{n,v'}\times\r{M}_{1,m,v'}\times\r{M}_{m,1,v'}$ (resp.
$\r{U}_{n,m,v'}$ with $\r{GL}_{n,v'}\times\(\r{M}_{1,m,v'}\)^2$).
Then
$[\zeta,x,y]\leftrightarrow[\zeta,(x,\:^{\r{t}}y\beta_{\circ})]$
where $\beta=(\beta_{\bullet},\beta_{\circ})$.

By \eqref{5omega} and \eqref{5di5}, we have
 \begin{align}\label{5split}
 &\;\c{O}_{\mu_{v'}}(F_{n,v'}\otimes
 F_{m,v'};\Phi_{v'},[\zeta,x,y])\notag\\
 =&\intl_{\r{GL}_{m,v'}}\iintl_{\(U_{1^r,m,1^r,v'}\)^2}\iintl_{\(\r{GL}_{m,v'}\)^2}
 \sigma(F_{n,v'})(g_{\bullet}^{-1}g^{-1}\underline{u}_{\bullet}^{-1}\zeta\underline{u}_{\circ}gg_{\circ})
 F_{m,v_{\bullet}}(g_{\bullet}^{-1})F_{m,v_{\circ}}(g_{\circ}^{-1})\notag\\
 &\(\omega_{\overline{\psi},\overline{\mu}}(g_{\bullet},g_{\circ})\Phi_{v'}\)^{\dag}
 \(\(x+\frac{(n_r)_{\bullet}+(n_r)_{\circ}}{2}\)g,
 g^{-1}\(y-\frac{\(n^*_r\)_{\bullet}+\(n^*_{r}\)_{\circ}}{2}\)\)\notag\\
 &\underline{\underline{\psi}}'\(\underline{u}_{\bullet}^{-1}\underline{u}_{\circ}\)
 \overline{\psi'}\(j\(\(b_{r,r}\)_{\bullet}-\(b_{r,r}\)_{\circ}+x\frac{\(n^*_r\)_{\bullet}-\(n^*_{r}\)_{\circ}}{2}
 +\frac{(n_r)_{\bullet}-(n_r)_{\circ}}{2}y\)\)
 \r{d}g_{\bullet}\r{d}g_{\circ}\r{d}\underline{u}_{\bullet}\r{d}\underline{u}_{\circ}\r{d}g
 \end{align}
where
 \begin{align*}
 \underline{\underline{\psi}}'(\underline{u})=\psi'\(j\(u_{1,2}+\cdots+u_{r-1,r}+u^*_{r,r-1}+\cdots+u^*_{2,1}\)\);
 \qquad\jmath=(j,-j)
 \end{align*}
and
 \begin{align*}
 \underline{u}=&\[
 \begin{array}{ccccccccc}
   1 & u_{1,2} & \ddots  & n_{1,1} & \cdots & n_{1,m} & b_{1,r} & \cdots & b_{1,1} \\
    & \ddots & u_{r-1,r} & \vdots &  & \vdots         & \vdots &  & \vdots \\
                  &  & 1 & n_{r,1} & \cdots & n_{r,m} & b_{r,r} & \cdots & b_{r,1} \\
    &                &   & 1 &  &                     & n^*_{1,r} & \cdots & n^*_{1,1} \\
    &                &   &   & \ddots &               & \vdots &  & \vdots \\
    &                &   &   &     & 1                & n^*_{r,r} & \cdots & n^*_{r,1} \\
    &  &  &  &  &                                     & 1 & u^*_{r,r-1} & \ddots \\
    &  &  &  &  &                                     &  & \ddots & u^*_{2,1} \\
    &  &  &  &  &                                     &  &  & 1
 \end{array}
 \];\\
 n_r=&\[\begin{array}{ccc}
         n_{r,1} & \cdots & n_{r,r}
       \end{array}
 \];\qquad
 n^*_r=\:^{\r{t}}\[\begin{array}{ccc}
                     n^*_{1,r} & \cdots & n^*_{r,r}
                   \end{array}
 \].
 \end{align*}

On the other hand, by \eqref{5omegau} and \eqref{5du4}, we have
 \begin{align}\label{5splitu}
 &\;\c{O}_{\psi'_{v'},\mu_{v'}}(f_{n,v'}\otimes
 f_{m,v'};\phi_{1,v'}\otimes\phi_{2,v'},[\zeta,(x,\:^{\r{t}}y\beta_{\circ})])\notag\\
 =&\intl_{\r{GL}_{m,v'}}\iintl_{\(U_{1^r,m,1^r,v'}\)^2}\intl_{\r{GL}_{m,v'}}
 f_{n,v'}(g_{\bullet}^{-1}g^{-1}\underline{u}_{\bullet}^{-1}\zeta\underline{u}_{\circ}g)
 f_{m,v'}(g_{\bullet}^{-1})\notag\\
 &\(\(\omega'_{\overline{\psi'},\overline{\mu}}(g_{\bullet})\phi_{1,v'}\)
 \otimes\overline{\phi_{2,v'}}\)^{\ddag}
 \(\(x+\frac{(n_r)_{\bullet}+(n_r)_{\circ}}{2}\)g,
 \:^{\r{t}}\(y-\frac{\(n^*_r\)_{\bullet}+\(n^*_{r}\)_{\circ}}{2}\)\:^{\r{t}}g^{-1}\beta_{\circ}\)\notag\\
 &\underline{\underline{\psi}}'\(\underline{u}_{\bullet}^{-1}\underline{u}_{\circ}\)
 \overline{\psi'}\(j\(\(b_{r,r}\)_{\bullet}-\(b_{r,r}\)_{\circ}+x\frac{\(n^*_r\)_{\bullet}-\(n^*_{r}\)_{\circ}}{2}
 +\frac{(n_r)_{\bullet}-(n_r)_{\circ}}{2}y\)\)
 \r{d}g_{\bullet}\r{d}\underline{u}_{\bullet}\r{d}\underline{u}_{\circ}\r{d}g.
 \end{align}

To compare \eqref{5split} and \eqref{5splitu}, we need to invoke the
original models of the Weil representations. In particular
 \begin{align*}
 \(\omega_{\overline{\psi},\overline{\mu}}(g_{\bullet},g_{\circ})\Phi_{v'}\)^{\dag}(x,y)=
 \mu(\det g_{\bullet}^{-1}g_{\circ})|\det g_{\bullet}g_{\circ}|_{v'}^{\frac{1}{2}}
 \intl_{\r{M}_{1,m,v'}}\Phi_{v'}\((x+z)g_{\bullet},(x-z)g_{\circ}\)\psi'(jzy)\r{d}z
 \end{align*}
where $\mu=(\mu,\mu^{-1})$ with abuse of notation. If we identify
$\b{L}_{v'}$ with $\r{M}_{1,m,v_{\bullet}}$ (possibly through a
Fourier transform), we have
 \begin{align*}
 \(\(\omega'_{\overline{\psi'},\overline{\mu}}(g_{\bullet})\phi_{1,v'}\)
 \otimes\overline{\phi_{2,v'}}\)^{\ddag}(x,\:^{\r{t}}y\beta_{\circ})=
 \mu(\det g_{\bullet}^{-1})|\det
 g_{\bullet}|_{v'}^{\frac{1}{2}}\intl_{\r{M}_{1,m,v'}}
 \phi_{1,v'}\((x+z)g_{\bullet}\)\overline{\phi_{2,v'}}(x-z)\psi'(jzy)\r{d}z.
 \end{align*}

Now if we suppose that $\Phi_{v'}=\Phi_{\bullet}\otimes\Phi_{\circ}$
is factorizable with respect to the two components. Let
$F_{m,v_{\bullet}}=f_{m,v'}$, $\Phi_{\bullet}=\phi_{1,v'}$ and
assume that the following function on
$(h,w)\in\r{GL}_{n,v'}\times\r{M}_{1,m,v'}$:
 \begin{align*}
 \intl_{\r{GL}_{m,v'}}\sigma(F_{n,v'})(hg_{\circ})F_{m,v_{\circ}}(g_{\circ}^{-1})\Phi_{\circ}(wg_{\circ})
 \mu(\det g_{\circ})|\det g_{\circ}|_{v'}^{\frac{1}{2}}\r{d}g_{\circ}
 \end{align*}
is equal to $\sum_i
f_{n,v'}^{(i)}\otimes\overline{\phi^{(i)}_{2,v'}}$, then
 \begin{align*}
 \c{O}_{\mu_{v'}}(F_{n,v'}\otimes
 F_{m,v'};\Phi_{\bullet}\otimes\Phi_{\circ},[\zeta,x,y])=\sum_i
 \c{O}_{\psi'_{v'},\mu_{v'}}(f^{(i)}_{n,v'}\otimes
 f_{m,v'};\phi_{1,v'}\otimes\phi_{2,v'}^{(i)},[\zeta,(x,\:^{\r{t}}y\beta_{\circ})]),
 \end{align*}
hence the smooth matching (of functions) holds!\\

\begin{rem}
As we see in the above calculation, there are several differences in
the case of Fourier-Jacobi periods:
 \begin{itemize}
   \item The data of test functions involve Schwartz functions on the groups as well as on the
   linear spaces;
   \item Even for the split case, smooth matching is not obvious and we need to choose
   linear combination of functions to make corresponding weighted orbital
   integrals equal.
 \end{itemize}

For almost all split places $v'$ where everything is unramified and
the test functions are the characteristic functions on corresponding
maximal compact subgroups or lattices, then the two orbital
integrals are equal.

Inspired by the split case, we conjecture that, similar to
Conjecture \ref{conj4sm}, the smooth matching of functions holds for
all places $v'$. We omit the explicit form of this conjecture in the
current case.\\
\end{rem}

\subsection{The fundamental lemma}
\label{sec5fl}

We now state the fundamental lemma for the Fourier-Jacobi periods.
We use all the notations in the beginning of Section \ref{sec4fl},
except that we define
 \begin{align}\label{5tfactor}
 \b{t}([\zeta,x,y])=
 \begin{cases}
 (-1)^{\r{val}\(\b{T}_{[\zeta,x,y]}\cdot\prod_{i=1}^rt_i(\zeta)\)} &
 m\text{ is even;}\\
 (-1)^{\r{val}\(\b{T}_{[\zeta,x,y]}\)} &
 m\text{ is odd.}
 \end{cases}
 \end{align}
For simplicity, we only consider the fundamental lemma for unit
elements which is

 \begin{conj}[\textbf{The fundamental lemma}]\label{conj5fl}
 Assume that $k/k'$, $\psi'$, $\mu$ are unramified and
 $\jmath\in\f{o}$. Then we have
  \begin{align*}
  &\c{O}_{\mu}(\CF_{\r{S}_n(\f{o}')};\CF_{\r{M}_{1,m}(\f{o}')}\otimes\CF_{\r{M}_{m,1}(\f{o}')},[\zeta,x,y])\\
  =&\begin{cases}
  \b{t}([\zeta,x,y])\c{O}_{\psi',\mu}(\CF_{\r{U}^+_n(\f{o}')};\CF_{\r{M}_{1,m}(\f{o})},[\zeta^+,z])
  & [\zeta,x,y]\leftrightarrow[\zeta^+,z]\in\r{U}^+_{n,m}(k');\\
  0 &[\zeta,x,y]\leftrightarrow[\zeta^-,z]\in\r{U}^-_{n,m}(k')
  \end{cases}
 \end{align*}
where
 \begin{align*}
 \c{O}_{\mu}(\CF_{\r{S}_n(\f{o}')};\CF_{\r{M}_{1,m}(\f{o}')}\otimes\CF_{\r{M}_{m,1}(\f{o}')},[\zeta,x,y])
 =&\intl_{\b{H}(k')}\CF_{\r{S}_n(\f{o}')}([\zeta]\b{h})
 \(\omega^{\dag}_{\overline{\psi},\overline{\mu}}(\b{h})\(\CF_{\r{M}_{1,m}(\f{o}')}\otimes\CF_{\r{M}_{m,1}(\f{o}')}\)\)
 (x,y)\underline{\psi}(\b{h})\r{d}\b{h};\\
 \c{O}_{\psi',\mu}(\CF_{\r{U}^+_n(\f{o}')};\CF_{\r{M}_{1,m}(\f{o})},[\zeta^+,z])
 =&\intl_{\b{H}^+(k')}\CF_{\r{U}^+_n(\f{o}')}([\zeta^+]\b{h}')
 \(\omega^{\ddag}_{\overline{\psi'},\overline{\mu}}(\b{h}')\CF_{\r{M}_{1,m}(\f{o})}\)(z)
 \underline{\psi'}(\b{h}')\r{d}\b{h}'.
 \end{align*}
It is easy to see that
$[\zeta,x,y]\leftrightarrow[\zeta^+,z]\in\r{U}^+_{n,m}(k')$ if and
only if $\r{val}\(\Delta_{[\zeta,x,y]}\)$ is even.

In particular, when $n=m$, the above orbital integrals become the
following much simpler ones
 \begin{align*}
 \c{O}_{\mu}(\CF_{\r{S}_n(\f{o}')};\CF_{\r{M}_{1,n}(\f{o}')}\otimes\CF_{\r{M}_{n,1}(\f{o}')},[\zeta,x,y])
 =&\intl_{\r{GL}_n(k')}\CF_{\r{S}_n(\f{o}')}(g^{-1}\zeta g)
 \CF_{\r{M}_{1,n}(\f{o}')}(xg)\CF_{\r{M}_{n,1}(\f{o}')}(g^{-1}y)\eta(\det g)\r{d}g;\\
 \c{O}_{\psi',\mu}(\CF_{\r{U}^+_n(\f{o}')};\CF_{\r{M}_{1,n}(\f{o})},[\zeta^+,z])
 =&\intl_{\r{U}^+_n(k')}\CF_{\r{U}^+_n(\f{o}')}(g'^{-1}\zeta^+g')
 \CF_{\r{M}_{1,n}(\f{o})}(zg')\r{d}g'.
 \end{align*}\\
 \end{conj}

 \begin{prop}\label{prop5half}
  Let $n=m$ or $m=0$, if $\r{val}\(\Delta_{[\zeta,x,y]}\)$ is odd, then
  \begin{align*}
  \c{O}_{\mu}(\CF_{\r{S}_n(\f{o}')};\CF_{\r{M}_{1,n}(\f{o}')}\otimes\CF_{\r{M}_{n,1}(\f{o}')},[\zeta,x,y])=0.
  \end{align*}
 \end{prop}
 \begin{proof}
 It follows from the similar argument in the proof of Proposition
 \ref{prop4half} by noticing that
 $\CF_{\r{M}_{1,n}(\f{o}')}(x)=\CF_{\r{M}_{n,1}(\f{o}')}(\:^{\r{t}}x)$.\\
 \end{proof}

It seems that, when $n>m$, the above vanishing result is not that
easy to proof.\\

\subsection{Proof of the fundamental lemma for $\r{U}_n\times\r{U}_n$}
\label{sec5nn}

In this section, we prove the fundamental lemma for the case $n=m$
in positive characteristics. The proof uses a similar idea of
relating orbiting integrals to certain problems of counting
lattices, following \cite{Yun09}, and then reduces to an identity
which has already been proved by Yun. In other word, the fundamental
lemmas for both the case $n=m$ and $n=m+1$
(the Jacquet-Rallis case) will be implied by the same identity.\\

Let $k'$ be a $p$-adic local field with ring of integers $\f{o}'$,
uniformizer $\varpi$, and $q=|\f{o}'/\varpi\f{o}'|$. Let $k/k'$ be
an unramified quadratic field extension with ring of integers
$\f{o}$ and $0\neq\jmath\in\f{o}$ with $\jmath^{\tau}=-\jmath$. Let
$\r{val}$ be the valuation on $k^{\times}$ normalized such that
$\r{val}(\varpi)=1$. For two full rank $\f{o}$-lattice
$\Lambda_1,\Lambda_2$ in some finite dimensional $k$-vector space
$V$, we define
 \begin{align*}
 \r{leng}_{\f{o}}(\Lambda_1:\Lambda_2):=\r{leng}_{\f{o}}(\Lambda_1/\Lambda_1\cap\Lambda_2)-
 \r{leng}_{\f{o}}(\Lambda_2/\Lambda_1\cap\Lambda_2)
 \end{align*}
where for $\Lambda_1\supset\Lambda_2$,
 \begin{align*}
 \r{leng}_{\f{o}}(\Lambda_1:\Lambda_2)=\frac{|\Lambda_1/\Lambda_2|}{|\f{o}/\varpi\f{o}|}
 =\frac{|\Lambda_1/\Lambda_2|}{q^2}.
 \end{align*}
We have also the same notion for $\f{o}'$ (but of course replacing
$q^2$ by $q$).\\

Let us first review some constructions in \cite[Section 2]{Yun09}.
Let $n\geq 1$, for any pair $(a,b)$ with $a=(a_i),b=(b_i)\in\f{o}^n$
and $a_n\in\f{o}^{\times}$, we define an $\f{o}$-algebra
 \begin{align*}
 \b{R}_{a,\f{o}}=\f{o}[t,t^{-1}]/(t^n-a_1t^{n-1}+\cdots+(-1)^na_n).
 \end{align*}
Hence $Z'_a:=\r{Spec}\:\b{R}_{a,\f{o}}$ is a subscheme of
$\r{Spec}\:\f{o}\times\d{G}_{\r{m}}$ which is finite flat over
$\r{Spec}\:\f{o}$ of degree $n$. Let
$\b{R}_{a,\f{o}}^{\vee}=\r{Hom}_{\f{o}}(\b{R}_{a,\f{o}},\f{o})$ and
define an element $\b{b}\in\b{R}_{a,\f{o}}^{\vee}$ from the datum
$b$ by the formula
 \begin{align}\label{5bb}
 \b{b}:\b{R}_{a,\f{o}}&\LR \f{o}\\
 t^i&\T b_i,\qquad i=0,...,n-1.
 \end{align}
It induces an $\b{R}_{a,\f{o}}$-linear homomorphism
$\gamma'_{a,b}:\b{R}_{a,\f{o}}\R\b{R}_{a,\f{o}}^{\vee}$ by the
pairing
 \begin{align*}
 \b{R}_{a,\f{o}}\otimes\b{R}_{a,\f{o}}&\LR\f{o}\\
 (u,v)&\T\b{b}(uv).
 \end{align*}
Let $\vartheta$ be the involution on
$\r{Res}_{\f{o}/\f{o}'}\(\r{Spec}\:\f{o}\times\d{G}_{\r{m}}\)$ which
is the product of $\tau$ on $\r{Spec}\:\f{o}$ and the involution
$t\T t^{-1}$ on $\d{G}_{\r{m}}$. The subscheme invariant under
$\vartheta$ is the unitary group (scheme) $\r{U}_{1,\f{o}/\f{o}'}$
over $\r{Spec}\:\f{o}'$.

Recall that
$\r{S}_n(\f{o}')=\1\zeta\in\r{M}_n(\f{o}')\;|\;\zeta\zeta^{\tau}=\b{1}_n\2$.
For
$[\zeta,x,y]\in\r{S}_n(\f{o}')\times\r{M}_{1,n}(\f{o}')\times\r{M}_{n,1}(\f{o}')$,
let $a_i=a_i([\zeta,x,y])=\r{Tr}\(\bigwedge^i\zeta\)$ and
$b_i=b_i([\zeta,x,y])=x\zeta^iy$ be the invariants. Then the
corresponding subscheme $Z'_a$ is stable under $\vartheta$, hence
determining a subscheme $Z_a$ of $\r{U}_{1,\f{o}/\f{o}'}$, finite
flat of degree $n$ over $\r{Spec}\:\f{o}'$. Let $\b{R}_a$ be the
coordinate ring of $Z_a$ which is a finite flat $\f{o}'$-algebra of
rank $n$ such that $\b{R}_a\otimes_{\f{o}'}\f{o}=\b{R}_{a,\f{o}}$.
The $\b{R}_{a,\f{o}}$-linear map $\gamma'_{a,b}$ descends to an
$\b{R}_a$-linear map $\gamma_{a,b}:\b{R}_a\R\b{R}_a^{\vee}$, where
$\b{R}_a^{\vee}=\r{Hom}_{\f{o}'}(\b{R}_a,\f{o}')$.

Let us assume that $\Delta_{[\zeta,x,y]}\neq0$, then the image of
$\gamma_{a,b}$ is co-finite in $\b{R}_a^{\vee}$. If we identify
$\b{R}_a$ as a submodule of $\b{R}_a^{\vee}$, then
$\r{leng}_{\f{o}'}(\b{R}_a^{\vee}:\b{R}_a)=\r{val}\(\Delta_{[\zeta,x,y]}\)$.
For each $0\leq i\leq\r{val}\(\Delta_{[\zeta,x,y]}\)$, let
 \begin{align*}
 \b{M}_{i,a,b}:=\1\b{R}_a\text{-lattices }\Lambda\;|\;\b{R}_a\subset\Lambda\subset\b{R}_a^{\vee}
 \text{ and }\r{leng}_{\f{o}'}(\b{R}_a^{\vee}:\Lambda)=i\2.
 \end{align*}
We remark that for the orbital integral
$\c{O}_{\mu}(\CF_{\r{S}_n(\f{o}')};\CF_{\r{M}_{1,m}(\f{o}')}\otimes\CF_{\r{M}_{m,1}(\f{o}')},[\zeta,x,y])$
to be nonzero, there must be an element locating in
$\r{S}_n(\f{o}')\times\r{M}_{1,n}(\f{o}')\times\r{M}_{n,1}(\f{o}')$
in the $\r{GL}_n(k')$-orbit of $[\zeta,x,y]$. Then we have

\begin{prop}\label{prop5count}
Let
$[\zeta,x,y]\in\r{S}_n(\f{o}')\times\r{M}_{1,n}(\f{o}')\times\r{M}_{n,1}(\f{o}')$
be regular and hence such that
$\r{val}\(\Delta_{[\zeta,x,y]}\)\geq0$, then the orbital integral
 \begin{align*}
 \c{O}_{\mu}(\CF_{\r{S}_n(\f{o}')};\CF_{\r{M}_{1,m}(\f{o}')}\otimes\CF_{\r{M}_{m,1}(\f{o}')},[\zeta,x,y])
 =(-1)^{\r{val}\(\b{T}_{[\zeta,x,y]}\)}\sum_{i=0}^{\r{val}\(\Delta_{[\zeta,x,y]}\)}(-1)^i\left|\b{M}_{i,a,b}\right|
 \end{align*}
where $\b{T}_{[\zeta,x,y]}$ is defined in \eqref{5bt}.
\end{prop}
\begin{proof}
Let $\b{V}=\r{M}_{n,1}(\f{o}')$ be the $\f{o'}$-module and we
identity $\b{V}^{\vee}$ with $\r{M}_{1,n}(\f{o}')$ by matrix
multiplication. Then $V=\b{V}(k):=\b{V}\otimes_{\f{o}'}k$. Recall
that
 \begin{align}\label{5recall}
 \c{O}_{\mu}(\CF_{\r{S}_n(\f{o}')};\CF_{\r{M}_{1,n}(\f{o}')}\otimes\CF_{\r{M}_{n,1}(\f{o}')},[\zeta,x,y])
 =&\intl_{\r{GL}_n(k')}\CF_{\r{S}_n(\f{o}')}(g^{-1}\zeta g)
 \CF_{\r{M}_{1,n}(\f{o}')}(xg)\CF_{\r{M}_{n,1}(\f{o}')}(g^{-1}y)\eta(\det
 g)\r{d}g
 \end{align}
and the measure is the one such that $\r{GL}_n(\f{o}')$ gets volume
$1$. We define
 \begin{align*}
 X^i_{[\zeta,x,y]}&:=\1 g\in\r{GL}_n(k')/\r{GL}_n(\f{o}')\;|\;g^{-1}\zeta g\in\r{S}_n(\f{o}'),
 xg\in\b{V}^{\vee},g^{-1}y\in\b{V},\r{val}\(\det g\)=i\2;\\
 Y^i_{[\zeta,x,y]}&:=\1\f{o}'\text{-lattice }L\subset\b{V}(k')\;|\;\zeta
 L\subset\f{o}L,x\in L^{\vee},y\in L,\r{leng}_{\f{o}'}(L:\b{V})=i\2
 \end{align*}
where $L^{\vee}=\1 v\in\b{V}^{\vee}(k')\;|\;vv'\in\f{o'}\text{ for
any }v'\in L\2$. There is a bijection
$X^i_{[\zeta,x,y]}\overset{\sim}{\LR}Y^i_{[\zeta,x,y]}$ given by
$g\T g\b{V}$.

On the other hand, we define an $k$-linear map
 \begin{align*}
 \b{y}':\b{R}_a(k)&\LR\b{V}(k)\\
 t^i&\T \zeta^iy
 \end{align*}
which is bijective since $[\zeta,x,y]$ is regular. Here,
$\b{R}_a(k)$ is the underlying $k$-module of the $k$-algebra
$\b{R}_a\otimes_{\f{o}'}k$. By the definition of $\gamma'_{a,b}$,
the following $k$-linear map
 \begin{align*}
 \b{x}'=\gamma'_{a,b}\circ\(\b{y}'^{\vee}\)^{-1}:\b{R}_a(k)\overset{\gamma'_{a,b}}{\LR}\b{R}_a^{\vee}(k)
 \overset{\(\b{y}'^{\vee}\)^{-1}}{\LR}\b{V}^{\vee}(k)
 \end{align*}
is given by $t^i\T x\zeta^i$. It is clear that $\b{y}'$ (resp.
$\b{x}'$) descends to a $k'$-linear map
$\b{y}:\b{R}_a(k')\LR\b{V}(k')$ (resp.
$\b{x}:\b{R}_a(k')\LR\b{V}^{\vee}(k')$).

For any $L\in Y^i_{[\zeta,x,y]}$, we claim that
$\b{y}^{-1}(L)\in\b{M}_{\r{val}\(\b{T}_{[\zeta,x,y]}\)-i,a,b}$. In
fact, $\b{y}^{-1}(L)$ is a lattice stable under $\b{R}_a$ by the
construction. Hence we only need to show that
$\b{R}_a\subset\b{y}^{-1}(L)\subset\b{R}^{\vee}_a$. Since $y\in L$,
$\b{y}^{-1}(y)=1_{\b{R}_a}\in\b{y}^{-1}(L)$, where $1_{\b{R}_a}$ is
the identity element of the algebra $1_{\b{R}_a}$. Hence
 \begin{align*}
 \b{R}_a=\b{R}_a\cdot1_{\b{R}_a}\subset\b{R}_a\cdot\b{y}^{-1}(L)\subset\b{y}^{-1}(L).
 \end{align*}
Similarly, since $x\in L^{\vee}$,
$\b{x}^{-1}(x)=1_{\b{R}_a}\in\b{x}^{-1}(L^{\vee})=\(\b{y}^{-1}(L)\)^{\vee}$.
But $\(\b{y}^{-1}(L)\)^{\vee}$ is also stable under $\b{R}_a$ which
implies that
 \begin{align*}
  \b{R}_a=\b{R}_a\cdot1_{\b{R}_a}\subset\b{R}_a\cdot\(\b{y}^{-1}(L)\)^{\vee}\subset\(\b{y}^{-1}(L)\)^{\vee}
 \end{align*}
i.e., $\b{y}^{-1}(L)\subset\b{R}_a^{\vee}$. Moreover
 \begin{align*}
 \r{leng}_{\f{o}'}(\b{R}_a^{\vee}:\b{y}^{-1}(L))=\r{leng}_{\f{o}'}(\b{R}_a^{\vee}:\b{y}^{-1}(\b{V}))-
 \r{leng}_{\f{o}'}(L:\b{V})=\r{val}\(\b{T}_{[\zeta,x,y]}\)-i.
 \end{align*}

Conversely, for any
$\Lambda\in\b{M}_{\r{val}\(\b{T}_{[\zeta,x,y]}\)-i,a,b}$, let
$L=\b{y}(\Lambda)$. Then
$\b{R}_{a,\f{o}}\cdot\Lambda=\(\b{R}_a\cdot\Lambda\)\otimes\f{o}\subset\Lambda\otimes\f{o}$,
hence $\zeta L\subset\f{o}L$. The fact $\b{R}_a\subset\Lambda$
implies that $y=\b{y}\(1_{\r{R}_a}\)\in L$ and the fact
$\Lambda\subset\b{R}_a^{\vee}$ implies that
$x=\b{x}\(1_{\b{R}_a}\)\in L^{\vee}$. The length part is clear.
Hence, we prove the claim.

Then
 \begin{align*}
 \eqref{5recall}=&\sum_{i=0}^{\r{val}\(\Delta_{[\zeta,x,y]}\)}(-1)^i\left|X^i_{[\zeta,x,y]}\right|
 =\sum_{i=0}^{\r{val}\(\Delta_{[\zeta,x,y]}\)}(-1)^i\left|Y^i_{[\zeta,x,y]}\right|\\
 =&(-1)^{\r{val}\(\b{T}_{[\zeta,x,y]}\)}\sum_{i=0}^{\r{val}\(\Delta_{[\zeta,x,y]}\)}(-1)^i\left|\b{M}_{i,a,b}\right|.
 \end{align*}\\
\end{proof}

Now we consider the orbital integral on the unitary group. We fix an
element $\beta^+\in\r{Her}_n(k)\cap\r{GL}_n(\f{o})$ which defines a
unitary group scheme $\r{U}^+_{n,\f{o}/\f{o}'}$ whose generic fibre
$\r{U}^+_{n,\f{o}/\f{o}'}\times k'\cong\r{U}^+_n$ viewed as a
subgroup of $\r{Res}_{k/k'}\r{GL}_{n,k}$. Now we simply view
$\r{U}^+_n$ as defined over $\f{o}'$. For
$[\zeta^+,z]\in\r{U}^+_n(\f{o}')\times\r{M}_{1,n}(\f{o})$, we recall
that $z^*=\(\beta^+\)^{-1}\:^{\r{t}}z^{\tau}\in\r{M}_{n,1}(\f{o})$
and the invariants $(a,b)\in\f{o}^{2n}$. Then the corresponding
subscheme $Z'_a=\r{Spec}\:\b{R}_{a,\f{o}}$ is also stable under
$\vartheta$, hence determining a subscheme $Z_a$ of
$\r{U}_{1,\f{o}/\f{o}'}$. Let $\b{R}_a$ be the coordinate ring of
$Z_a$. We have an $\b{R}_{a,\f{o}}$-linear map
$\gamma'_{a,b}:\b{R}_{a,\f{o}}\R\b{R}_{a,\f{o}}^{\vee}$ identifying
$\b{R}_{a,\f{o}}$ as a submodule of $\b{R}_{a,\f{o}}^{\vee}$.

We define the following morphism between $k$-modules
 \begin{align*}
 \b{z}^*:\b{R}_a(k)&\LR\b{V}(k)\\
 t^i&\T\(\zeta^+\)^iz^*
 \end{align*}
which is an isomorphism if $[\zeta^+,z]$ is regular, which we will
assume in the following discussion. The hermitian form on $\b{V}(k)$
defined by $\beta^+$ induces the following hermitian form on
$\b{R}_a(k)$ through $\b{z}^*$:
 \begin{align*}
 (u,v)_{\b{R}}=\b{b}\(uv^{\vartheta}\)
 \end{align*}
where $\b{b}:\b{R}_a(k)\R k$ is defined in \eqref{5bb}.

For an $\f{o}$-lattice $\Lambda^+\subset\b{R}_a(k)$, the dual
lattice under $(-,-)_{\b{R}}$ is the $\f{o}$-lattice
 \begin{align*}
 \(\Lambda^+\)^{\vee}:=\1
 v\in\b{R}_a(k)\;|\;(v,\Lambda^+)_{\b{R}}\subset\f{o}\2.
 \end{align*}
We call $\Lambda^+$ self-dual if $\(\Lambda^+\)^{\vee}=\Lambda^+$.
We define
 \begin{align*}
 \b{N}_{a,b}:=\1\text{self-dual }R_{a,\f{o}}\text{-lattice }\Lambda^+\;|\;
 \b{R}_{a,\f{o}}\subset\Lambda^+\subset\b{R}_{a,\f{o}}^{\vee}\2.
 \end{align*}
We remark that for the orbital integral
$\c{O}_{\psi',\mu}(\CF_{\r{U}^+_n(\f{o}')};\CF_{\r{M}_{1,n}(\f{o})},[\zeta^+,z])$
to be nonzero, there must be an element locating in
$\r{U}^+_n(\f{o}')\times\r{M}_{1,n}(\f{o})$ in the
$\r{U}^+_n(k')$-orbit of $[\zeta^+,z]$. Then we have

\begin{prop}\label{prop5countu}
Let $[\zeta^+,z]\in\r{U}^+_n(\f{o}')\times\r{M}_{1,n}(\f{o})$ be
regular and hence such that $\r{val}\(\Delta_{[\zeta^+,z]}\)\geq0$
is even, then the orbital integral
 \begin{align*}
 \c{O}_{\psi',\mu}(\CF_{\r{U}^+_n(\f{o}')};\CF_{\r{M}_{1,n}(\f{o})},[\zeta^+,z])
 =\left|\b{N}_{a,b}\right|.
 \end{align*}
\end{prop}
\begin{proof}
Recall that
 \begin{align}\label{5recallu}
 \c{O}_{\psi',\mu}(\CF_{\r{U}^+_n(\f{o}')};\CF_{\r{M}_{1,n}(\f{o})},[\zeta^+,z])
 =&\intl_{\r{U}^+_n(k')}\CF_{\r{U}^+_n(\f{o}')}(g'^{-1}\zeta^+g')
 \CF_{\r{M}_{1,n}(\f{o})}(zg')\r{d}g'
 \end{align}
and the measure is the one such that $\r{U}^+_n(\f{o}')$ gets volume
$1$. We define
 \begin{align*}
 X_{[\zeta^+,z]}&=\1g'\in\r{U}^+_n(k')/\r{U}^+_n(\f{o}')\;|\;g'^{-1}\zeta^+g'\in\r{U}^+_n(\f{o}'),
 g^{-1}z^*\in\b{V}(\f{o})\2;\\
 Y_{[\zeta^+,z]}&=\1\text{self-dual }\f{o}\text{-lattice }
 L^+\subset\b{V}(k)\;|\;\zeta^+L^+\subset\f{o}L^+,
 z^*\in L^+\2.
 \end{align*}
There is a bijection
$X_{[\zeta^+,z]}\overset{\sim}{\LR}Y_{[\zeta^+,z]}$ given by $g'\T
g'\b{V}(\f{o})$. We have another bijection
$Y_{[\zeta^+,z]}\overset{\sim}{\LR}\b{N}_{a,b}$ given by
$L^+\T\(z^*\)^{-1}\(L^+\)$ whose proof is similar to that in
Proposition \ref{prop5count}. Hence we have
 \begin{align*}
 \eqref{5recallu}=\left|X_{[\zeta^+,z]}\right|=\left|Y_{[\zeta^+,z]}\right|=\left|\b{N}_{a,b}\right|.
 \end{align*}\\
\end{proof}

\begin{theo}\label{theo5fl}
If $\r{char}(k)=p>\r{max}\1 n,2\2$, the fundamental lemma for
$\r{U}_n\times\r{U}_n$ (cf. Conjecture \ref{conj5fl}), i.e., the
following identity
  \begin{align*}
  &\c{O}_{\mu}(\CF_{\r{S}_n(\f{o}')};\CF_{\r{M}_{1,n}(\f{o}')}\otimes\CF_{\r{M}_{n,1}(\f{o}')},[\zeta,x,y])\\
  =&\begin{cases}
  (-1)^{\r{val}\(\b{T}_{[\zeta,x,y]}\)}\c{O}_{\psi',\mu}(\CF_{\r{U}^+_n(\f{o}')};\CF_{\r{M}_{1,n}(\f{o})},[\zeta^+,z])
  & [\zeta,x,y]\leftrightarrow[\zeta^+,z]\in\r{U}^+_{n,n}(k');\\
  0 &[\zeta,x,y]\leftrightarrow[\zeta^-,z]\in\r{U}^-_{n,n}(k')
  \end{cases}
  \end{align*}
holds for $[\zeta,x,y]$ regular.
\end{theo}
\begin{proof}
The second case of the above identity has already been proved in
Proposition \ref{prop5half}. But Proposition \ref{prop5count} leads
to another proof by noticing that, first
$\r{val}\(\Delta_{[\zeta,x,y]}\)$ is odd and $[\zeta,x,y]$ satisfies
the assumption in that proposition (otherwise, the orbital integral
will be $0$ automatically), and second $\Lambda \T\Lambda^{\vee}$
induces a bijection
$\b{M}_{i,a,b}\overset{\sim}{\LR}\b{M}_{\r{val}\(\Delta_{[\zeta,x,y]}\)-i,a,b}$.

For the first part, we need to prove the identity
 \begin{align}\label{5compare}
 \sum_{i=0}^{\r{val}\(\Delta_{a,b}\)}(-1)^i\left|\b{M}_{i,a,b}\right|
 =\left|\b{N}_{a,b}\right|
 \end{align}
assuming that $[\zeta,x,y]$ (resp. $[\zeta^+,z]$) satisfies the
assumption in Proposition \ref{prop5count} (resp. \ref{prop5countu})
and they have the same invariants $(a,b)$. Here, $\Delta_{a,b}$ is
the determinant of the map $\gamma'_{a,b}$ under the basis $\1
1,t,...,t^{n-1}\2$ of $\b{R}_a(k)$ and the dual basis of
$\b{R}^{\vee}_a(k)$ and hence
$\r{val}\(\Delta_{a,b}\)=\r{val}\(\Delta_{[\zeta,x,y]}\)$. We remark
that the equality \eqref{5compare} is a property purely of the
$\f{o}'$-algebra $\b{R}_a$ and the $\b{R}_a$-linear map
$\gamma_{a,b}:\b{R}_a\R\b{R}_a^{\vee}$.

By the argument in \cite[Proposition 2.6.1]{Yun09}, we can find
$\widetilde{a}_i,\widetilde{b}_i\in\jmath^i\f{o}'$ and an
isomorphism of $\f{o}'$-algebras
$\rho:R_{\widetilde{a}}\overset{\sim}{\LR}\b{R}_a$ such that the
following diagram is commutative:
 \begin{align*}
 \xymatrix{
   R_{\widetilde{a}} \ar[d]^{\wr}_{\rho} \ar[r]^{\gamma_{\widetilde{a},\widetilde{b}}}
   & R_{\widetilde{a}}^{\vee}  \\
   \b{R}_a \ar[r]^{\gamma_{a,b}} & \b{R}_a^{\vee}   \ar[u]^{\wr}_{\rho^{\vee}}}
 \end{align*}
where the algebra $R_{\widetilde{a}}$ and the map
$\gamma_{\widetilde{a},\widetilde{b}}$ are defined in \cite[Remark
2.2.5]{Yun09}. Since $[\zeta,x,y]$ is regular which means in
particular that $\b{R}_a\otimes k$ is an \'{e}tale $k$-algebra,
hence $R_{\widetilde{a}}\otimes k$ is also \'{e}tale. Applying the
fundamental result \cite[Corollary 2.7.2]{Yun09} (and also their
notations) to $R_{\widetilde{a}}$ and
$\gamma_{\widetilde{a},\widetilde{b}}$, we have that
 \begin{align*}
 \sum_{i=0}^{\r{val}\(\Delta_{\widetilde{a},\widetilde{b}}\)}(-1)^i
 \left|M^{\r{loc}}_{i,\widetilde{a},\widetilde{b}}\right|
 =\left|N^{\r{loc}}_{\widetilde{a},\widetilde{b}}\right|
 \end{align*}
for $\r{char}(k)=p>\r{max}\1 n,2\2$, which implies \eqref{5compare} immediately.\\
\end{proof}

\begin{rem}
If $\r{char}(k)=0$ and $p$ is sufficiently large with respect to
$n$, the transfer principle in \cite[Appendix]{Yun09} should also
apply to our case and hence imply the fundamental lemma in
characteristic $0$.\\
\end{rem}

\section{Appendix: A brief summary on local Whittaker integrals}
\label{sec6}

In the appendix, we summarize some facts about certain integrals of
local Whittaker functions which will be used in this paper. All the
results are contained in \cite{JS81}, \cite{JPSS83}, \cite{JS90},
\cite{CPS04} and \cite{J09}.\\

Let $k$ be a local field, $\psi:k\R\d{C}^1$ be a nontrivial
character. We denote $|\;|=|\;|_k$, $\r{M}_{r,s}=\r{M}_{r,s}(k)$.
Let $\pi$ (resp. $\sigma$) be an irreducible admissible
representation of $\r{GL}_n=\r{GL}_n(k)$ (resp.
$\r{GL}_m=\r{GL}_m(k)$). Let
$\c{W}(\psi)=\r{Ind}_{U_{1^n}}^{\r{GL}_n}(\psi)$ be the space of all
smooth functions $W(g)$ on $\r{GL}_n$ satisfying
$W(ug)=\underline{\psi}(u)W(g):=\psi(u_{1,2}+\cdots+u_{n-1,n})W(g)$
for all $u=(u_{ij})\in U_{1^n}$, the group of upper-triangular
matrices with all $1$ on the diagonal. It is a smooth representation
of $\r{GL}_n$ by right translation. Let $V_{\pi}$ be the space where
$\pi$ realizes. If $k$ is archimedean, then we take $V_{\pi}$ as the
canonical Casselman-Wallach completion of the corresponding
Harish-Chandra module of $\pi$. A fundamental theorem of
Gelfand-Kazhdan and Shalika posits that there is at most one
$\r{GL}_n$-equivariant map, up to a constant multiple, from
$V_{\pi}$ to $\c{W}(\psi)$. If it exists, then we say $\pi$ is
\emph{generic}. Being generic is independent of $\psi$ we choose.
Same arguments apply to $\sigma$. In what follows, we will assume
that $\pi$ and $\sigma$ are generic. We denote by $\c{W}(\pi,\psi)$
(resp. $\c{W}(\sigma,\overline{\psi})$) the nontrivial image of
$V_{\pi}$ (resp. $V_{\sigma}$) in $\c{W}(\psi)$ (resp.
$\c{W}(\overline{\psi})$). Moreover,
 \begin{align*}
 \c{W}(\widetilde{\pi},\overline{\psi})=\1\left.\widetilde{W}(g):=W(w_n\:^{\r{t}}g^{-1})\;\right|\;
 W\in\c{W}(\pi,\psi)\2
 \end{align*}
where $w_n=\[\begin{array}{cc}
                & 1 \\
               w_{n-1} &
             \end{array}
\]$ is the longest Weyl element of $\r{GL}_n$. Moreover, we let
$e_m=[0,...,0,1]\in\r{M}_{1,m}$ and $w_{n,m}=\[\begin{array}{cc}
                                                                             \b{1}_m &  \\
                                                                              &
                                                                              w_{n-m}
                                                                           \end{array}
\]$.\\

Let $W\in\c{W}(\pi,\sigma)$, $W^-\in\c{W}(\sigma,\overline{\psi})$
and $\Phi\in\c{S}(\r{M}_{1,m})$, we consider the following kinds of
integrals
 \begin{itemize}
   \item For $n>m$ and $0\leq r\leq n-m-1$,
    \begin{align}\label{6wb}
    \Psi_r(s;W,W^-)
    =\intl_{U_{1^m}\B\r{GL}_m}\intl_{\r{M}_{r,m}}W\(\[\begin{array}{ccc}
                                g & 0 & 0 \\
                            x & \b{1}_r & 0 \\
                            0 & 0 & \b{1}_{n-m-r}
                          \end{array}
    \]\)W^-(g)|\det
    g|^{s-\frac{n-m}{2}}\r{d}x\r{d}g
    \end{align}
   \item For $n>m$ and $1\leq r\leq n-m$,
     \begin{align}\label{6wf}
 \Psi_r(s;W,W^-;\Phi)=\intl_{U_{1^m}\B\r{GL}_m}
 \intl_{\r{M}_{r-1,m}}\intl_{\r{M}_{1,m}}
 W\(\[\begin{array}{cccc}
                            g & 0 & 0 & 0 \\
                            x & \b{1}_{r-1} & 0 & 0 \\
                            y & 0 & 1 & 0 \\
                            0 & 0 & 0 & \b{1}_{n-m-r}
                          \end{array}
 \]\)W^-(g)\Phi(y)|\det g|^{s-\frac{n-m}{2}}\r{d}y\r{d}x\r{d}g
 \end{align}
   \item For $n\geq m$,
     \begin{align}\label{6wf0}
 \Psi_0(s;W,W^-;\Phi)
 =\intl_{U_{1^m}\B\r{GL}_m}W\(\[\begin{array}{cc}
                                                                     g & 0 \\
                                                                     0 &
                                                                     \b{1}_{n-m}
                                                                   \end{array}
 \]\)W^-(g)\Phi(e_mg)|\det
 g|^{s-\frac{n-m}{2}}\r{d}g
 \end{align}
 \end{itemize}

We denote by
 \begin{align*}
 \c{I}_r(\pi\times\sigma)=&\1 \Psi_r(s;W,W^-)\;|\;W\in\c{W}(\pi,\sigma),
 W^-\in\c{W}(\sigma,\overline{\psi})\2;\qquad 0\leq r\leq n-m-1\\
 \c{I}_r^{\natural}(\pi\times\sigma)=&\1 \Psi_r(s;W,W^-;\Phi)\;|\;W\in\c{W}(\pi,\sigma),
 W^-\in\c{W}(\sigma,\overline{\psi}),\Phi\in\c{S}(\r{M}_{1,m})\2;\qquad 0\leq r\leq n-m
 \end{align*}
which are linear spaces over $\d{C}$.\\

\begin{prop}[\cite{JPSS83}, \cite{JS90}, \cite{CPS04},
\cite{J09}]\label{prop6jpss} We only state the following results for
$k$ non-archimedean, while the statement for $k$ archimedean can be
found, for example, in \cite[Section 2]{J09}. Let $\pi$ and $\sigma$
be as above and
$\chi_{\sigma}$ the central character of $\sigma$, then:\\
(1) Each element in $\c{I}_r(\pi\times\sigma)$ and
$\c{I}_r^{\natural}(\pi\times\sigma)$ is absolutely convergent for
$\Re(s)$ large and has a meromorphic continuation to the entire
complex plane;\\
(2) There exists a unique function $L(s,\pi\times\sigma)$ of the
form $P(q^{-s})^{-1}$ where $P\in\d{C}[X]$ and $q$ is the
cardinality of the residue field of $k$, such that
 \begin{align*}
 \c{I}_r(\pi\times\sigma)=\c{I}^{\natural}_r(\pi\times\sigma)=L(s,\pi\times\sigma)\d{C}[q^{-s},q^s]
 \end{align*}
for any $r$; in particular, for any $r$ and $s_0\in\d{C}$, there
exist $W$, $W^-$ and possibly $\Phi$ such that
$\Psi_r(s;W,W^-)/L(s,\pi\times\sigma)|_{s=s_0}$ or
$\Psi_r(s;W,W^-;\Phi)/L(s,\pi\times\sigma)|_{s=s_0}$ is in $\d{C}^{\times}$.\\
(3) There is a factor $\epsilon(s,\pi\times\sigma,\psi)$, only
depending on $\pi$, $\sigma$ and $\psi$, of the form $cq^{-fs}$ such
that
 \begin{align*}
 \frac{\Psi_{n-m-1-r}(1-s;\widetilde{W},\widetilde{W^-})}{L(1-s,\widetilde{\pi}\times\widetilde{\sigma})}=
 \chi_{\sigma}(-1)^{n-1}\epsilon(s,\pi\times\sigma,\psi)\frac{\Psi_r(s;W,W^-)}{L(s,\pi\times\sigma)}
 \end{align*}
for $n>m$ and
 \begin{align}\label{6natural}
 \frac{\Psi_{n-m-r}(1-s;\widetilde{W},\widetilde{W^-};\widehat{\Phi})}
 {L(1-s,\widetilde{\pi}\times\widetilde{\sigma})}=
 \chi_{\sigma}(-1)^{n-1}\epsilon(s,\pi\times\sigma,\psi)\frac{\Psi_r(s;W,W^-;\Phi)}{L(s,\pi\times\sigma)}
 \end{align}
for $n\geq m$, where $\widehat{\Phi}$ is the $\psi$-Fourier
transform of $\Phi$:
 \begin{align*}
 \widehat{\Phi}(y)=\intl_{\r{M}_{1,m}}\Phi(x)\psi(x\:^{\r{t}}y)\r{d}x
 \end{align*}
with the self-dual measure $\r{d}x$.
\end{prop}
\begin{proof}
The proof of these statements can be found in the literature
mentioned above, except \eqref{6natural} when $n>m$. For
completeness, we will give a proof of \eqref{6natural} when $n>m$
below, following \cite{JPSS83}.

By the functional equation, we can assume that $0\leq r<n-m$. We let
 \begin{align*}
 W_1=\intl_{\r{M}_{m,1}}\rho\(\[\begin{array}{cccc}
                                  \b{1}_m &  & u &  \\
                                   & \b{1}_r &  &  \\
                                   &  & 1 &  \\
                                   &  &  & \b{1}_{n-m-r-1}
                                \end{array}
 \]\)W\widehat{\Phi}(-\:^{\r{t}}u)\r{d}u
 \end{align*}
which is in $\c{W}(\pi,\psi)$. Then
$\Psi_r(s;W_1,W^-)=\Psi_r(s;W,W^-;\Phi)$. To prove \eqref{6natural},
we only need to prove that
$\Psi_{n-m-1-r}(s;\rho(w_{n,m})\widetilde{W_1},\widetilde{W^-})
=\Psi_{n-m-r}(s;\rho(w_{n,m})\widetilde{W},\widetilde{W^-};\widehat{\Phi})$
which is true by the Fourier inverse formula.\\
\end{proof}

When $k$ is archimedean or the representations $\pi$ and $\sigma$
are unramified, then the representation $\pi\boxtimes\sigma$, hence
its $L$-factor $L(s,\pi\boxtimes\sigma)$ and $\epsilon$-factor
$\epsilon(s,\pi\boxtimes\sigma,\psi)$ are defined using the
Langlands parameter. We have the following.

\begin{prop}[\cite{JS81}, \cite{JS90}, \cite{CPS04},
\cite{J09}]\label{prop6unr} (1) If $k$ is archimedean, then
$L(s,\pi\times\sigma)=L(s,\pi\boxtimes\sigma)$ and
$\epsilon(s,\pi\times\sigma,\psi)=\epsilon(s,\pi\boxtimes\sigma,\psi)$;\\
(2) If $k$ is non-archimedean with $\f{o}$ its ring of integers, let
$\pi$ (resp. $\sigma$) be an unramified representation associated to
a semisimple conjugacy class $A_{\pi}\in\r{GL}_n(\d{C})$ (resp.
$A_{\sigma}\in\r{GL}_m(\d{C})$). Let $W_{\circ}$ (resp.
$W^-_{\circ}$) be the unique $\r{GL}_n(\f{o})$- (resp.
$\r{GL}_m(\f{o})$-) fixed Whittaker functions such that
$W(\b{1}_n)=1$ (resp. $W^-(\b{1}_m)=1$) and $\Phi_{\circ}$ be the
characteristic function of $\r{M}_{1,m}(\f{o})$, then
 \begin{align*}
 \Psi_r(s;W_{\circ},W^-_{\circ})=\Psi_r(s;W_{\circ},W^-_{\circ};\Phi_{\circ})
 =\det\(1-q^{-s}A_{\pi}\otimes A_{\sigma}\)^{-1}=
 L(s,\pi\times\sigma)=L(s,\pi\boxtimes\sigma)
 \end{align*}
for any possible $r$.
\end{prop}
\begin{proof}
In (2), the proof for the integral $\Psi_0(s;W_{\circ},W^-_{\circ})$
when $n>m$, and for the integral
$\Psi_0(s;W_{\circ},W^-_{\circ};\Phi_{\circ})$ when $n=m$ can be
found in \cite{JS81}. The rest will follow easily as in Proposition
\ref{prop6jpss}.\\
\end{proof}

\begin{cor}\label{6cor}
(1) When $n>m$, the Whittaker integral \eqref{6wb} defines a nonzero
element in $\r{Hom}_H(\pi\otimes\sigma,\nu)$, hence an
$(r,n-m-1-r)$-Bessel model for $\pi$, $\sigma$ generic, by choosing
a suitable basis like in Section
\ref{sec2bi}.\\
(2) When $n>m$ and $r>0$ (resp. $n\geq m$ and $r=0$), the Whittaker
integral \eqref{6wf} (resp. \eqref{6wf0}) (with $\sigma$ replaced by
$\sigma\otimes\mu^{-1}$) defines a nonzero element in
$\r{Hom}_H(\pi\otimes\sigma\otimes\widetilde{\nu_{\mu}},\d{C})$,
hence an $(r,n-m-r)$- (resp. $(0,n-m)$-) Fourier-Jacobi model for
$\pi$, $\sigma$ generic, by choosing a suitable basis like in
Section \ref{sec3fi}.
\end{cor}

\end{document}